\documentclass[12pt,a4paper,reqno]{amsart}

\newcommand\version{January 15, 2026}

\usepackage{amsmath,amsfonts,amsthm,amssymb,amsxtra}
\usepackage{enumerate}
\usepackage{bbm} 
\usepackage{esint}
\usepackage{bm} 
\usepackage{color}
\usepackage{hyperref}
\usepackage[normalem]{ulem}
\usepackage[autostyle=false, style=english]{csquotes}
\MakeOuterQuote{"}

\newtheorem{theorem}{Theorem}[section]
\newtheorem{proposition}[theorem]{Proposition}
\newtheorem{lemma}[theorem]{Lemma}

\theoremstyle{definition}

\theoremstyle{remark}

\newtheorem{remark}[theorem]{Remark}

\numberwithin{equation}{section}

\newcommand{\C}{\mathbb{C}}

\renewcommand{\epsilon}{\varepsilon}

\newcommand{\N}{\mathbb{N}}

\renewcommand{\phi}{\varphi}
\newcommand{\R}{\mathbb{R}}

\newcommand{\cs}{\mathcal{S}}

\newcommand{\Z}{\mathbb{Z}}

\newcommand{\me}[1]{\mathrm{e}^{#1}}
\newcommand{\one}{\mathbf{1}}

\newcommand{\lesi}{\lesssim}

\newcommand{\f}{\frac}

\newcommand{\vc}{\infty}
\newcommand{\Rd}{\mathbb{R}^d}

\newcommand{\La}{\Lambda_\kappa}
\newcommand{\Ln}{\Lambda_0}

\begin{document}


\title[Hardy operators and Sobolev norms --- \version]{Equivalence of Sobolev norms for \\ Kolmogorov operators with scaling-critical drift}

\author[T. A. Bui]{The Anh Bui}
\address[The Anh Bui]{School of Mathematical and Physical Sciences, Macquarie University, NSW 2109, Australia}
\email{the.bui@mq.edu.au}

\author[X. T. Duong]{Xuan Thinh Duong}
\address[Xuan Thinh Duong]{School of Mathematical and Physical Sciences, Macquarie University, NSW 2109, Australia}
\email{xuan.duong@mq.edu.au}

\author[K. Merz]{Konstantin Merz}
\address[Konstantin Merz]{Institute for Theoretical Physics, ETH Zurich, Wolfgang--Pauli-Strasse~27, 8093 Zurich, Switzerland}
\email{konstantin.merz@ethz.ch}

\begin{abstract}
  We consider the ordinary or fractional Laplacian plus a homogeneous, scaling-critical drift term. This operator is non-symmetric but homogeneous, and generates scales of $L^p$-Sobolev spaces which we compare with the ordinary homogeneous Sobolev spaces. Unlike in previous studies concerning Hardy operators, i.e., ordinary or fractional Laplacians plus scaling-critical scalar perturbations, handling the drift term requires an additional, possibly technical, restriction on the range of comparable Sobolev spaces, which is related to the unavailability of gradient bounds for the associated semigroup.
\end{abstract}

\maketitle
\vspace*{-2em}
\tableofcontents

\section{Introduction and main result}
\label{s:introduction}

\subsection{Introduction}

We consider Kolmogorov operators, i.e., ordinary or fractional Laplacians plus gradient perturbations, given by $(-\Delta)^{\alpha/2}+b(x)\cdot\nabla$, acting on functions in $\R^d$ with $d\in\N:=\{1,2,\ldots\}$. Here, $\alpha\in(0,2]$ and $b:\R^d\to\R^d$ is a vector field. Their study is motivated, among others, by probability theory, where they arise as generators of Brownian motion or $\alpha$-stable processes with drift. Kolmogorov operators are also important in physics, biology, finance, and further applied sciences, where the inclusion of gradient perturbations allows for modeling additional forces or influences acting on the system, thereby providing a more comprehensive description of the underlying processes. For example, in turbulent fluids, Kolmogorov operators describe anomalous diffusion, i.e., particles spreading at a rate different from what Brownian motion or $\alpha$-stable processes predict \cite{MetzlerKlafter2000}.
In biology, gradient perturbations can represent directed movement or taxis \cite{Klagesetal2008}, and in finance, they can represent drifts in asset price dynamics \cite{ContTankov2004}. For further examples and references, we refer, e.g., to the review \cite{Stinga2019}.

In this paper, we consider the following Kolmogorov operator, formally given by
\begin{align}
  \Lambda_\kappa := (-\Delta)^{\alpha/2} + \frac{\kappa}{|x|^\alpha} \, x \cdot \nabla,
\end{align}
for $d\in\N$ and $\alpha\in(1,2]$ with $\alpha<(d+2)/2$ and a coupling constant $\kappa\in\R$.
We call the gradient perturbation "attractive" for $\kappa>0$ and "repulsive" for $\kappa<0$, following terminology in \cite{KinzebulatovSemenov2023F,Kinzebulatovetal2024}, and motivated by considering the action of $x\cdot\nabla$ on positive functions decaying at infinity. The assumption $\alpha>1$ is necessary for the heat kernel of $\La$---which will be a key technical tool in our paper---to have a chance to be comparable with that of $\Ln$. We defer a more detailed explanation to Remark~\ref{rem:heatkernelalphaleq1}. Besides, $\alpha\geq1$ is used in technical integral estimates, especially in Section~\ref{s:newboundsdifferenceskernels}, below. The assumption $\alpha<(d+2)/2$ is crucial as it ensures that $\La$ generates a holomorphic semigroup and poses only an additional restriction when $d=1$ or $d=2$. We will review the precise definition of $\La$ in Subsection~\ref{ss:mainresult} below.

Importantly, $\La$ is homogeneous (of degree $-\alpha$). Therefore, one expects $\La$ and $\Ln$ to be comparable to each other in some sense. That this is indeed the case, is the main result of the present paper. In Theorem~\ref{eqsob} below, we compare the $L^p(\R^d)$-Sobolev spaces generated by powers of $\La$ with the ordinary, homogeneous Sobolev spaces.
Our result joins a line of recent research \cite{Killipetal2018,Franketal2021,Merz2021,BuiDAncona2023} on the so-called Hardy operator,
\begin{align}
  L_\kappa := (-\Delta)^{\alpha/2} + \frac{\kappa}{|x|^\alpha} \quad \text{in} \ L^2(\R^d),
\end{align}
i.e., the fractional Laplacian plus the scalar-valued, so-called, Hardy potential $\kappa/|x|^\alpha$. We also refer to \cite{FrankMerz2023,BuiMerz2023} for studies concerning the regional fractional Laplacian \cite{Bogdanetal2003} on the half-space $\R_+^d$ with Hardy potential $\kappa/x_d^\alpha$ depending on the distance to the bounding half-plane.
As $(-\Delta)^{\alpha/2}$ and $|x|^{-\alpha}$ share the same scaling behavior, these operators compete with each 
on every length scale, thereby leading to the emergence of a critical coupling constant. More precisely, by the sharp Hardy--Kato--Herbst inequality \cite{Hardy1919,Hardy1920,Kato1966,Herbst1977}\footnote{See also \cite{Kovalenkoetal1981,Yafaev1999,Franketal2008H,FrankSeiringer2008} for other proofs of \eqref{eq:hardy} with the optimal constant $\kappa_*$. Formula~\eqref{eq:hardy} is often simply called Hardy inequality.}
\begin{align}
  \label{eq:hardy}
  \begin{split}
    & \left\| (-\Delta)^{\alpha/4} f \right\|_{L^2(\R^d)}^2
    \geq \kappa_* \left\| |x|^{-\alpha/2} f \right\|_{L^2(\R^d)}^2 \\
    & \quad \text{for all}\ \alpha\in(0,2]\cap(0,d) \text{ and } f\in C_c^\infty(\R^d)
  \end{split}
\end{align}
with
\begin{align}
  \kappa_* = \kappa_*(d,\alpha) := 2^\alpha\, \frac{\Gamma\left(\frac{d+\alpha}{4}\right)^2}{\Gamma\left(\frac{d-\alpha}{4}\right)^2},
\end{align}
the quadratic form corresponding to $L_\kappa$ is bounded from below if and only if $\kappa\geq-\kappa_*$. Moreover, if $\kappa\geq-\kappa_*$, then this form is nonnegative and $L_\kappa$ can be realized as a self-adjoint operator by a theorem of Friedrichs. While $L_\kappa$ has no eigenvalues, there is a strictly monotonously increasing function 
\begin{align}
  \label{eq:defdeltakappa}
  (-\infty,\kappa_*]\ni\kappa\mapsto\delta_\kappa\in(-M,(d-\alpha)/2],
\end{align}
where 
\begin{align}
  \label{eq:defM}
  M :=
  \begin{cases}
    \alpha & \quad \text{if} \ \alpha<2, \\
    \infty & \quad \text{if} \ \alpha=2,
  \end{cases}
\end{align}
such that $L_\kappa|x|^{-\delta_\kappa}=0$ holds pointwise almost everywhere. As $|x|^{-\delta_\kappa}$ is, however, not square-integrable, this function is sometimes called generalized eigenfunction or generalized ground state, and $0$ is called the corresponding generalized eigenvalue of $L_\kappa$. 
We refer, e.g., to \cite[p.~2286]{Franketal2021} for an expression of $\delta_\kappa$ and merely record $\delta_{-\infty}=-M$, $\delta_0=0$, and $\delta_{\kappa_*}=(d-\alpha)/2$.

\subsection{Main result}
\label{ss:mainresult}
To state our main result, Theorem~\ref{eqsob}, we introduce notation related to the homogeneity of $\La$ and, in fact, two, corresponding critical coupling constants. 
In the following, we parameterize $\kappa$ via
$$
\kappa = \Psi(\beta)
$$
where
\begin{align}
  \label{eq:defbeta}
  \Psi(\beta):=
  \begin{cases}
    \frac{2^\alpha}{\beta-\alpha} \cdot \frac{\Gamma\left(\frac\beta2\right)\Gamma\left(\frac d2-\frac{\beta-\alpha}{2}\right)}{\Gamma\left(\frac d2-\frac\beta2\right)\Gamma\left(\frac{\beta-\alpha}{2}\right)} & \quad \text{for $\alpha\in(1,2)$ and} \ \beta\in(\alpha,d+\alpha), \\
    d-\beta & \quad \text{for $\alpha=2$ and} \ \beta\in\R.
  \end{cases}
\end{align}
Thus, as $\beta$ runs through the domain of $\Psi(\beta)$, so does the coupling constant $\kappa=\Psi(\beta)$; below, we will discuss the behavior of $\Psi(\beta)$ in detail.
The parameter $\beta$ enters in the generalized eigenvalue equation $\Lambda_\kappa^*|x|^{\beta-d}=0$ with $\Lambda_\kappa^*=\Lambda_0-\kappa\nabla\cdot (|x|^{-\alpha}x)$ being the formal adjoint of $\La$, sometimes called Fokker--Planck operator. As $\delta_\kappa$ plays an important role in the analysis of $L_\kappa$, we anticipate that the parameter $\beta$ will be important in the present study.
We note that the fraction of the four Gamma functions in~\eqref{eq:defbeta} becomes maximal for $\beta=(d+\alpha)/2$, i.e., the midpoint of the interval $(\alpha,d)$; cf.~\cite{Franketal2021}.
Thus, we denote the corresponding value of $\kappa$ by
\begin{align}
  \label{eq:kappacrit}
  \kappa_{\rm c} := \Psi\left(\frac{d+\alpha}{2}\right)
  = 
  \begin{cases} 
    \frac{2^{\alpha+1}}{d-\alpha} \cdot \frac{\Gamma\left(\frac{d+\alpha}{4}\right)^2}{\Gamma\left(\frac{d-\alpha}{4}\right)^2} & \quad \text{if} \ \alpha<2, \\
    \frac{d-2}{2} & \quad \text{if} \ \alpha=2,
  \end{cases}
\end{align}
which turns out to be one of the critical coupling constants as we will discuss shortly. However, $\Psi(\beta)$ is not symmetric around the midpoint $\beta=(d+\alpha)/2$, see, e.g., \cite[p.~1868]{Kinzebulatovetal2021} and \cite[p.~347]{KinzebulatovSemenov2023F} for plots.
Note further that 
$\lim_{\beta\nearrow d}\Psi(\beta)=\lim_{\beta\searrow d}\Psi(\beta)=0$, $\lim_{\beta\nearrow d+\alpha}\Psi(\beta)=-\infty$, and
\begin{align}
  \kappa_{\rm c} = \frac{2}{d-\alpha}\kappa_*(d,\alpha).
\end{align}

The following lemma---whose proof is deferred to Appendix~\ref{a:proofmonotonicity}---asserts that $\Psi(\beta)$ is decreasing for $\beta>(d+\alpha)/2$. Thus, for each $\kappa\leq\kappa_{\rm c}$ there is a unique $\beta$ such that $\kappa=\Psi(\beta)$.

\begin{lemma}
  \label{monotonicitykappabeta}
  Let $\alpha\in(1,2]$ and $M=\alpha$ if $\alpha<2$ and $M=\infty$ if $\alpha=2$.
  Then, the map $((d+\alpha)/2,d+M)\ni\beta\mapsto\Psi(\beta)$, given by~\eqref{eq:defbeta} is strictly monotonously decreasing.
\end{lemma}

\smallskip
We now explain how the scaling criticality of the gradient perturbation $\kappa|x|^{-\alpha} x\cdot\nabla$, called Hardy drift in \cite{Kinzebulatovetal2021}, manifests itself in the emergence of two critical coupling constants. 
\\
(1) The first, and for us more important, critical coupling constant is $\kappa_{\rm c}$. For $\kappa<\kappa_{\rm c}$ the Hardy drift is a Rellich perturbation of $(-\Delta)^{\alpha/2}$, i.e., 
\begin{align}
  \label{eq:rellichperturbation}
  \|\kappa|x|^{-\alpha}x\cdot\nabla(\zeta+(-\Delta)^{\alpha/2})^{-1}\|_{L^2(\R^d)\to L^2(\R^d)} < 1.
\end{align}
This follows from Hardy's inequality \eqref{eq:hardy} in the form
\begin{align}
  \label{eq:hardyconsequence}
  \|\kappa|x|^{1-\alpha}(-\Delta)^{-(\alpha-1)/2}\|_{L^2(\R^d)\to L^2(\R^d)} < 1,
\end{align}
whenever $\kappa<\kappa_*(d,2(\alpha-1))=\kappa_{\rm c}$ and $\alpha<(d+2)/2$.  Consequently, $\La$ can be constructed as the algebraic sum $(-\Delta)^{\alpha/2}+\kappa|x|^{-\alpha}x\cdot\nabla$ in $L^2(\R^d)$ with domain being the Sobolev space $H^{\alpha}(\R^d)$, and
$\Lambda_\kappa$ generates a holomorphic semigroup, denoted by $\me{-t\La}$, in $L^2(\R^d)$ whenever $\kappa<\kappa_{\rm c}$; see also \cite[Proposition~8]{Kinzebulatovetal2021}, \cite[\S 8]{KinzebulatovSemenov2023F}, and \cite[Theorem~4.2]{KinzebulatovSemenov2020}\footnote{While \cite{KinzebulatovSemenov2020,Kinzebulatovetal2021,KinzebulatovSemenov2023F} state their results only for $d\in\{3,4,\ldots\}$ and $\alpha\in(1,2]$, an inspection of their proofs, taking \eqref{eq:hardyconsequence} into account, show that all their results actually hold for all $d\in\N$ and $\alpha\in(1,2]$ with $\alpha<(d+2)/2$.}. Since the holomorphic semigroup corresponding to $\La$ plays an important technical role in the following, we restrict our attention from now on to $\kappa<\kappa_{\rm c}$, or, equivalently, to
\begin{align}
  \beta\in\left(\frac{d+\alpha}{2},d+M\right)
  \quad \text{with} \ d\in\N \text{ and } \alpha\in(1,2]\cap(1,(d+2)/2)
\end{align}
in view of \eqref{eq:kappacrit} and Lemma~\ref{monotonicitykappabeta}. Moreover, Kinzebulatov, Semenov, and Szczypkowski \cite{KinzebulatovSemenov2020,Kinzebulatovetal2021,KinzebulatovSemenov2023F} proved that $\me{-t\La}$ is an $L^\infty$ contraction and extends by continuity to a $C_0$ semigroup on $L^r(\R^d)$ for all $r\in[2,\infty)$ whenever $\kappa<\kappa_{\rm c}$. Furthermore, for $\alpha<2$, they proved that the kernel $\me{-t\La}(x,y)$ is bounded from above and below by positive constants times the heat kernel of $\Ln$ times a weight, depending on only one of the spatial variables. The reason that this weight depends only on one of the variables is due to the non-symmetry of $\La$. See~\eqref{eq:heatkernel} below for these bounds.
Analogous bounds for $\alpha=2$ do not appear to be available yet. However, in view of the factorization of the bounds for $\me{-t\La}$ with $\alpha<2$, we expect them to be of the form \eqref{eq:heatkernelalpha2}.
\\
(2) The second critical coupling constant is $2\kappa_{\rm c}$.
While the Hardy drift is not a Rellich perturbation anymore for $\kappa\geq\kappa_{\rm c}$, the operators $\La$ and $\me{-t\La}$ can still be defined for all $\kappa\in[\kappa_{\rm c},2\kappa_{\rm c})$. More precisely, for $\alpha<2$, \cite{Kinzebulatovetal2021} define $\me{-t\La}$ as the limit of the heat kernels associated to $(-\Delta)^{\alpha/2}+\kappa(|x|^2+\epsilon)^{-\alpha/2}x\cdot\nabla+\alpha\kappa\epsilon(|x|^2+\epsilon)^{-\alpha/2-1}$, defined on the Sobolev space $(1-\Delta)^{-\alpha/2}L^p(\R^d)$ for all $p\in[1,\infty)$, as $\epsilon\searrow0$ in $L^r(\R^d)$ for all $r\in(r_{\kappa},\infty)$ with an explicit $r_{\kappa}\geq2$. Correspondingly, for $\kappa\in[\kappa_{\rm c},2\kappa_{\rm c})$, $\La$ is defined to be the generator of the so-constructed heat kernel $\me{-t\La}$. We refer to \cite[Section~3]{Kinzebulatovetal2021} and \cite[Section~4]{KinzebulatovSemenov2020} for the precise procedures for $\alpha<2$ and $\alpha=2$, respectively. For $\alpha<2$, the work \cite{Kinzebulatovetal2021}
shows the heat kernel estimates for $\kappa<\kappa_{\rm c}$ discussed above extend to all $\kappa<2\kappa_{\rm c}$. One of the reasons why $2\kappa_{\rm c}$ also deserves to be called a critical constant is that for $\alpha=2$ and $\kappa>2\kappa_{\rm c}$, appropriately defined weak solutions to the corresponding parabolic equation cease to be unique \cite[p.~1588]{KinzebulatovSemenov2020}. Moreover, for $d=3$ and $\kappa=2\kappa_{\rm c}$, the properties of the corresponding semigroup are drastically different from the properties of $\me{t\Delta}$ and $\me{-t\La}$ with $\kappa<2\kappa_{\rm c}$, see, e.g., \cite{FitzsimmonsLi2019}. 
Noteworthy, for $\alpha\in(1,2)$, an optimal analog of Hardy's inequality \eqref{eq:hardy} in $L^p$ in~\cite{Bogdanetal2022} allows to extend certain estimates used in \cite{Kinzebulatovetal2021} to construct $\me{-t\La}$ and $\La$ even for $\kappa>2\kappa_{\rm c}$; see \cite[Remark~6]{Kinzebulatovetal2021}.

\smallskip
We now state our main result using the notation
\begin{align}\label{eq-d beta}
  d_\beta := \frac{d}{(d-\beta)\vee0}
  \quad \text{and} \quad
  (d_\beta)' := 1 \vee \frac{d}{\beta}
\end{align}
for $\beta>0$ with $A\vee B:=\max\{A,B\}$ and the convention $1/0=\infty$. 

\begin{theorem}[Equivalence of Sobolev norms]
  \label{eqsob}
  Let $d\in\N$, $\alpha\in(1,2]$ with $\alpha<(d+2)/2$, $s\in(0,1]$, $\beta\in ((d+\alpha)/2,d+M)$, and $\kappa=\Psi(\beta)$ be defined by \eqref{eq:defbeta}.
  If $\alpha=2$, assume that the upper heat kernel bound \eqref{eq:heatkernelalpha2} for $\me{-t\La}$ holds.
  If
  \begin{align}
    (d_\beta)' < p < \frac{d}{\alpha s} \wedge d_\beta,
  \end{align}
  then the following statements hold.
  \begin{enumerate}
  \item Assume furthermore that $\alpha s<\alpha-1$. Then,
    \begin{align}
      \label{eq:equivalencesobolev1}
      \|(\Lambda_0)^{s} f\|_{L^p(\R^d)} 
      \lesssim_{d,\alpha,s,\beta,p} \|(\Lambda_\kappa)^{s} f\|_{L^p(\R^d)}
      \quad \text{for all} \ f\in C_c^\infty(\R^d).
    \end{align}
    
  \item We have 
    \begin{align}
      \label{eq:equivalencesobolev2}
      \|(\Lambda_\kappa)^{s} f\|_{L^p(\R^d)}
      \lesssim_{d,\alpha,s,\beta,p} \|(\Lambda_0)^{s} f\|_{L^p(\R^d)}
      \quad \text{for all} \ f\in C_c^\infty(\R^d).
    \end{align}
  \end{enumerate}
\end{theorem}

Thus, Theorem~\ref{eqsob} says that the Sobolev spaces generated by powers of $\La$ and $\Ln$ are comparable with each other when the coupling constant is not too large in a quantitative sense; in particular, the singularity of the drift perturbation (which is also reflected in the important bounds for the heat kernel of $\Lambda_\kappa$) limits the range of admissible Sobolev exponents $s$. More precisely, when $\beta<d$ (corresponding to $\kappa>0$, i.e., an attractive gradient perturbation), then $d/(\alpha s)\wedge d_\beta=d/(\alpha s)$ and the range of admissible powers $s\in(0,\beta/\alpha)$ becomes smaller as $\beta$ runs from $d$ to $(d+\alpha)/2$. However, since $s\leq1$, the condition $s<\beta/\alpha$ only poses an additional restriction when $d=1$ since to have $(d+\alpha)/(2\alpha)<1$ (which is the value of $\beta/\alpha$ when $\beta=(d+\alpha)/2$), one needs $\alpha>d$, which is only possible if $d=1$.

Here and in the following, we write $A\wedge B:=\min\{A,B\}$. Moreover, for $A,B\geq0$ and a parameter $\tau$, we write $A\lesssim_\tau B$ whenever there is a $\tau$-dependent constant $c_\tau>0$ such that $A\leq c_\tau B$. The notation $A\sim_\tau B$ means $B\lesssim_\tau A\lesssim_\tau B$. The dependence on fixed parameters like $d,\alpha,s,\beta,p$ is usually omitted. Generic (real positive) constants will often be denoted by $c$ or $C$. For brevity, we sometimes write $\|f\|_p$ instead of $\|f\|_{L^p(\R^d)}$.

\smallskip
The rest of this introduction is structured as follows. First, we state auxiliary tools, so-called reversed and generalized Hardy inequalities, and use them to prove Theorem~\ref{eqsob}. These inequalities are also of independent interest. Afterwards, we compare Theorem~\ref{eqsob} to earlier, related results and outline implications. Finally, we outline the rest of the paper.

\subsection{Proof of Theorem~\ref{eqsob}}
\label{ss:proofideas}

The ideas to prove Theorem~\ref{eqsob} are similar to those in \cite{Killipetal2018,Franketal2021,Merz2021,BuiDAncona2023,FrankMerz2023,BuiMerz2023}, where heat kernel bounds were paramount. For $\alpha<2$, the heat kernel of $\La$ decays polynomially, while for $\alpha=2$ it obeys Gaussian bounds (compare \eqref{eq:heatkernel} and \eqref{eq:heatkernelalpha2} below). Thus, since the proofs of our results are significantly simpler for $\alpha=2$, we will prove Theorem~\ref{eqsob} and the following auxiliary statements involved in its proof only for $\alpha\in(1,2)$, and omit the extension of the proofs to the case $\alpha=2$.

Compared to the above-mentioned studies, our proof of Theorem~\ref{eqsob} requires more technical effort. First, as opposed to \cite{Franketal2021,FrankMerz2023}, due to the non-symmetry of $\Lambda_\kappa$, establishing the equivalence of $L^2$-Sobolev norms via the spectral theorem is impossible. Therefore, we will resort to the continuous square function estimates established in \cite{BuiDAncona2023} and again applied in \cite{BuiMerz2023}. Due to the non-symmetry of $\Lambda_\kappa$, we need to investigate the boundedness of the following two square functions: The first is associated with $\Lambda_\kappa$ and the latter is associated with its adjoint $\Lambda^*_\kappa$. For $\gamma>0$, they are defined as
\begin{align}
  S_{\Lambda_\kappa,\gamma}f(x)
  = \left(\int_0^\vc|(t\Lambda_\kappa)^{\gamma}\me{-t\Lambda_\kappa}f|^2\f{dt}{t}\right)^{1/2}
\end{align}
and
\begin{align}
  S_{\Lambda^*_\kappa,\gamma}f(x)
  = \left(\int_0^\vc|(t\Lambda^*_\kappa)^{\gamma}\me{-t\Lambda^*_\kappa}f|^2\f{dt}{t}\right)^{1/2}.
\end{align}
The operator $\La$ generates a holomorphic semigroup and is maximal accretive, i.e., ${\rm Re}\langle f,\Lambda_\kappa f\rangle\geq0$ for all $f\in{\rm dom}(\Lambda_\kappa)$ \cite{Kinzebulatovetal2021}, hence it has a bounded functional calculus on $L^2$ which implies (see \cite{McIntosh1986}) that the square functions  $S_{\La,\gamma}$ and $S_{\La^*,\gamma}$ are bounded on $L^2(\R^d)$. In the following theorem, which we prove in Section~\ref{s:squarefunctions}, we show that $S_{\La,\gamma}$ and $S_{\La^*,\gamma}$ are also bounded on $L^p(\R^d)$ with $p\neq2$ and, in particular, obtain a square function representation of $\|\Lambda_\kappa^{s} f\|_{p}$.

\begin{theorem}
  \label{squarefunctions}
  Let $\gamma\in(0,1]$, $\alpha\in (1,2]$ with $\alpha<(d+2)/2$, $\beta\in ((d+\alpha)/2,d+M)$, and $\kappa=\Psi(\beta)$ be defined by~\eqref{eq:defbeta}.
  If $\alpha=2$, assume that the upper heat kernel bound \eqref{eq:heatkernelalpha2} holds.
  Then,  
  \begin{equation}
    \label{eq-square function}
    \|S_{\Lambda_\kappa,\gamma}f\|_{L^p(\R^d)}\lesi \|f\|_{L^p(\R^d)} \quad \text{for all} \ (d_\beta)'<p<\vc
  \end{equation}
  and
  \begin{equation}
    \label{eq-square function duality}
    \|S_{\Lambda^*_\kappa,\gamma}f\|_{L^p(\R^d)}\lesi\|f\|_{L^p(\R^d)} \quad \text{for all} \ 1<p<d_\beta.
  \end{equation}
  Consequently, for all $s\in (0,1)$ and $(d_\beta)'<p<d_\beta$,
  \begin{align}
    \label{eq:squarefunctions2}
    \left\|\Big(\int_0^\vc t^{-2s}|t\Lambda_\kappa  \me{-t\Lambda_\kappa}f|^2\f{dt}{t}\Big)^{1/2}\right\|_{L^p(\R^d)}\sim \|\Lambda_\kappa^{s} f\|_{L^p(\R^d)}.
  \end{align}
\end{theorem}

Thus, to prove Theorem~\ref{eqsob}, it suffices to compare the square functions associated to $\Lambda_\kappa^s$ and $(-\Delta)^{\alpha s/2}$ with each other. To that end, we prove the following two statements, called reversed and generalized Hardy inequalities. We first state the reversed Hardy inequalities, expressed in terms of the square functions $S_{\Lambda_\kappa,\gamma}$. These inequalities give an upper bound for the difference of the square functions associated to $\La$ and $\Ln$ in terms of the scalar Hardy potential. While one would expect upper bounds in terms of the gradient perturbation, the bounds below are sufficient for our purposes. In particular, the advantage to incorporate the Hardy potential is that it is sign-definite and a mere multiplication operator. In turn, these properties enable us to straightforwardly prove that powers of the Hardy potential are bounded by powers of $\La$, leading to the generalized Hardy inequality in Theorem~\ref{thm-HardyIneq} below.

\begin{theorem}[Reversed Hardy inequality]
  \label{thm-difference}
  Let $\alpha\in(1,2]$ with $\alpha<(d+2)/2$, $\beta\in ((d+\alpha)/2,d+M)$, $s\in(0,1)$, and $\kappa=\Psi(\beta)$ be defined by \eqref{eq:defbeta}.
  If $\alpha=2$, assume that the upper heat kernel bound \eqref{eq:heatkernelalpha2} holds.
  Then the following statements hold for all $p\in((d_\beta)',\vc)$.
  \begin{enumerate}
  \item If $\alpha s<\alpha-1$, then
    \begin{align}
      \label{eq:thm-difference}
      \left\|\left(\int_0^\vc t^{-2s}\left|\left(t\La \me{-t\La} -t\Lambda_0\me{-t\Lambda_0}\right)f\right|^2\f{dt}{t}\right)^{1/2}\right\|_{L^p(\R^d)}
      \lesi \left\|\f{f}{|x|^{\alpha s}}\right\|_{L^p(\R^d)}.
    \end{align}
    
  \item For any $\gamma\in(0,1)$ such that $1-\gamma\leq \alpha s<\alpha-\gamma$, we have
    \begin{align}
      \label{eq:thm-difference2}
      \begin{split}
        & \left\|\left(\int_0^\vc t^{-2s}\left|\left(t\La \me{-t\La} -t\Lambda_0\me{-t\Lambda_0}\right)f\right|^2\f{dt}{t}\right)^{1/2}\right\|_{L^p(\R^d)} 
         \lesssim \left\|\frac{|\Ln^{\frac{1-\gamma}{\alpha}}f(x)|}{|x|^{\alpha s+\gamma-1}}\right\|_{L^p(\R^d)} + \left\|\f{f}{|x|^{\alpha s}}\right\|_{L^p(\R^d)}.
      \end{split}
    \end{align}
  \end{enumerate}
\end{theorem}

The proof of \eqref{eq:thm-difference} uses ideas in \cite{Franketal2021,Merz2021,BuiDAncona2023} and relies on pointwise estimates for the difference of the heat kernels associated to $\La$ and $\Ln$, as well as their time derivatives. These bounds are summarized in Proposition~\ref{prop-difference} below. To obtain these estimates, we use Duhamel's formula
$$
\me{-t\La}(x,y) - \me{-t\Ln}(x,y) = -\kappa\int_0^t d\tau \int_{\R^d}dz\, \me{-(t-\tau)\La}(x,z) |z|^{-\alpha}z\cdot\nabla_z\me{-\tau\Ln}(z,y).
$$
Due to the gradient perturbation, we only obtain suitable bounds if $\alpha s<\alpha-1$. In Appendix~\ref{a:optimallowerbound}, we discuss the necessity of this assumption and argue that our pointwise kernel bounds in Proposition~\ref{prop-difference} do not hold if $\alpha s\geq\alpha-1$.
On the other hand, \eqref{eq:thm-difference2} also covers the case $\alpha s\geq\alpha-1$ as this bound also involves the operator $\Ln$ on the right-hand side of the estimate. The proof of this inequality is technically more involved and does not rely on kernel bounds. Instead, we estimate $|(t\La\me{-t\La}-t\Ln\me{-t\Ln})f(x)|$ pointwise. While one might suspect that there is a bound similar to \eqref{eq:thm-difference2} involving $\||x|^{1-\gamma-\alpha s}|\La^{\frac{1-\gamma}{\alpha}}f|\|_p$ on the right-hand side, we are, unfortunately, not able to prove such a bound yet. One idea to obtain such a bound is to rewrite the above Duhamel formula as
$$
\me{-t\La}(x,y) - \me{-t\Ln}(x,y) = \kappa\int_0^t ds \int_{\R^d}dz\, \me{-(t-s)\Ln}(x,z) |z|^{-\alpha}z\cdot\nabla_z\me{-s\La}(z,y).
$$
This approach requires, however, gradient bounds for $\me{-t\La}$, e.g., those stated in \eqref{eq:heatkernelgradientconjecture0}, see also Appendix~\ref{a:reversedhardyoptimal} for a more detailed discussion. As such bounds are currently unavailable and are likely difficult to obtain, we believe it is an interesting problem to derive sharp bounds for $|\nabla_x\me{-t\La}(x,y)|$.

\smallskip
We now present generalized Hardy inequalities giving upper bounds for the scalar-valued Hardy potential in terms of $\La$.

\begin{theorem}[Generalized Hardy inequality]
  \label{thm-HardyIneq}
  Let $\alpha\in(1,2]$ with $\alpha<(d+2)/2$, $\beta\in ((d+\alpha)/2,d+M)$, $s\in(0,1]$, and $\kappa=\Psi(\beta)$ be defined by \eqref{eq:defbeta}.
  If $\alpha=2$, assume that the upper heat kernel bound \eqref{eq:heatkernelalpha2} holds.
  If $p\in((d_\beta)',d/(\alpha s))$, then
  \begin{align}
    \label{eq:thm-HardyIneq}
    \||x|^{-\alpha s} g\|_{L^p(\R^d)} \lesi \|(\Lambda_\kappa)^{s}g\|_{L^p(\R^d)} \quad \text{for all}\ g\in C_c^\infty(\R^d)
  \end{align}
  and, for all $1>\gamma\geq 1-\alpha s\geq0$,
  \begin{align}
    \label{eq:newgenhardynew2}
    \||x|^{\alpha(\frac{1-\gamma}{\alpha}-s)} |\La^{\frac{1-\gamma}{\alpha}} g| \|_{L^p(\R^d)}
    \lesssim \|(\La)^s g\|_{L^p(\R^d)}
    \quad \text{for all}\ g\in C_c^\infty(\R^d).
  \end{align}
\end{theorem}

To prove \eqref{eq:equivalencesobolev1} for $s=1$ without using the above-outlined square function representation, the estimate $\||x|^{-\alpha}x\cdot\nabla g\|_{p} \lesssim \|\La g\|_{p}$ would be beneficial. As our proof of Theorem~\ref{thm-HardyIneq} relies on Riesz and heat kernel bounds for $\me{-t\La}$, this would require gradient estimates for $\me{-t\La}$, too. In Proposition~\ref{thm-HardyIneqNew}, we prove a new generalized Hardy inequality for $|x|^{-\alpha}x\cdot\nabla$ assuming that suitable bounds for the gradient of $\me{-t\La}$ are available.

\smallskip
With Theorems~\ref{squarefunctions}, \ref{thm-difference}, and~\ref{thm-HardyIneq} at hand, we now give the

\begin{proof}[Proof of Theorem~\ref{eqsob}]
  For $s\in(0,1)$, the square function estimates in \eqref{eq:squarefunctions2}, the triangle inequality, the reversed and generalized Hardy inequalities (Theorems~\ref{thm-difference}--\ref{thm-HardyIneq}) yield, whenever $\alpha s<\alpha-1$,
  \begin{align}
    \begin{split}
      \| (\Lambda_0)^{s}u\|_{L^p(\R^d)}
      & \sim \left\|\Big(\int_0^\vc t^{-2s}|t\Lambda_0 \me{-t\Lambda_0}f|^2\f{dt}{t}\Big)^{1/2}\right\|_{L^p(\R^d)} \\
      & \le \left\|\Big(\int_0^\vc t^{-2s}|t\La \me{-t\La}f|^2\f{dt}{t}\Big)^{1/2}\right\|_{L^p(\R^d)} \\
      & \quad + \left\|\Big(\int_0^\vc t^{-2s}|(t \Lambda_0 \me{-t\Lambda_0} -t \Lambda_\kappa \me{-t\Lambda_\kappa})f|^2\f{dt}{t}\Big)^{\frac12} \right\|_{L^p(\R^d)} \\
      & \lesssim \| (\La)^{s}u\|_{L^p(\R^d)} + \left\||x|^{-\alpha s} f \right\|_{L^p(\R^d)}
        \lesssim \| (\La)^{s}u\|_{L^p(\R^d)}.
    \end{split}
  \end{align}
  This proves \eqref{eq:equivalencesobolev1}.
  Similarly, for any $\gamma\in(0,1)$ with $1-\gamma \leq \alpha s<\alpha-\gamma$ if $\alpha s\geq\alpha-1$,
  \begin{align}
    \begin{split}
      \| (\La)^{s}u\|_{L^p(\R^d)}
      & \sim \left\|\Big(\int_0^\vc t^{-2s}|t\La \me{-t\La}f|^2\f{dt}{t}\Big)^{1/2}\right\|_{L^p(\R^d)} \\
      & \le \left\|\Big(\int_0^\vc t^{-2s}|t\Ln \me{-t\Ln}f|^2\f{dt}{t}\Big)^{1/2}\right\|_{L^p(\R^d)} \\
      & \quad + \left\|\Big(\int_0^\vc t^{-2s}|(t \Lambda_0 \me{-t\Lambda_0} -t \Lambda_\kappa \me{-t\Lambda_\kappa})f|^2\f{dt}{t}\Big)^{\frac12} \right\|_{L^p(\R^d)} \\
      & \lesssim \| (\Ln)^{s}u\|_{L^p(\R^d)} + \left\||x|^{-\alpha s} f \right\|_{L^p(\R^d)} + \left\|\frac{|\Ln^{\frac{1-\gamma}{\alpha}}f|}{|x|^{\alpha s+\gamma-1}}\right\|_{L^p(\R^d)} \\ 
      & \lesssim \| (\Ln)^{s}u\|_{L^p(\R^d)},
    \end{split}
  \end{align}  
  which proves \eqref{eq:equivalencesobolev2} for all $s\in(0,1)$. For $s=1$, the square function estimates are not necessary. We merely use the triangle inequality and Hardy's inequality (see \eqref{eq:newgenhardynew2} with $\kappa=0$, $s=1$, and $\gamma=0$) to get the desired estimate.
\end{proof}

\begin{remark}
  In view of the above discussion, it is an open problem to decide whether \eqref{eq:equivalencesobolev1} holds for a wider range of parameters. If the gradient bound in \eqref{eq:heatkernelgradientconjecture0} was true for some $\gamma>2\alpha-d$, then, in view of the above proof, \eqref{eq:thm-difference2conj}, Proposition~\ref{thm-HardyIneqNew}, and \eqref{eq:newgenhardynew2}, Formula~\eqref{eq:equivalencesobolev1} would hold for all $s\leq1$ and $\beta>d-1$.
\end{remark}

\subsection{Discussion of Theorem~\ref{eqsob}}

In this subsection, we put Theorem~\ref{eqsob} into context and compare it with earlier, related results.
In Section~\ref{s:applications}, we present two concrete applications of Theorem~\ref{eqsob}.

\subsubsection{Comparison with earlier results}

We compare Theorem~\ref{eqsob} with the results in~\cite{Killipetal2018,Franketal2021,Merz2021,BuiDAncona2023}, where the $L^p$-Sobolev spaces generated by the Hardy operator $L_\kappa=(-\Delta)^{\alpha/2}+\kappa/|x|^\alpha$ with $\kappa\geq-\kappa_*$ were considered\footnote{We point out two misprints in \cite[(3)]{BuiDAncona2023}: There, the operators $(-\Delta)^{\alpha s/2}$ and $(\mathcal{L}_a)^{\alpha s/4}$ should be replaced with $(-\Delta)^{\alpha s/4}$ and $(\mathcal{L}_a)^{s/2}$, respectively.}.
Using the parameterization of $\kappa$ via~\eqref{eq:defdeltakappa}, the following statements hold for $s\in(0,1]$ and $1<p<\infty$:
\begin{enumerate}
\item If $\frac{d}{d-(\delta_\kappa\vee0)}<p<\frac{d}{(\delta_\kappa+\alpha s)\vee 0}$, then
  \begin{align}
    \label{eq:eqsobold1}
    \|(L_0)^{s}f \|_{L^p(\R^d)}
    \lesssim \| (L_\kappa)^{s}f\|_{L^p(\R^d)}
    \quad \text{for all} \ f\in C_c^\infty(\R^d).
  \end{align}
  
\item If $\frac{d}{d-(\delta_\kappa\vee0)}<p<\frac{d}{\alpha s\vee \delta_\kappa}$, then
  \begin{align}
    \label{eq:eqsobold2}
    \| (L_{\kappa})^{s}f\|_{L^p(\R^d)} 
    \lesssim \| (L_0)^{s}f \|_{L^p(\R^d)}
    \quad \text{for all} \ f\in C_c^\infty(\R^d).
  \end{align}
\end{enumerate}
Let us compare the assumptions on $p$ in these statements for $L_\kappa$ with the conditions in Theorem~\ref{eqsob}. Keeping in mind that $\beta$ has a similar interpretation as $\delta_\kappa$, we note two differences.
\begin{enumerate}
\item The lower bound $p>\frac{d}{d-(\delta_\kappa\vee0)}$ in \eqref{eq:eqsobold1}--\eqref{eq:eqsobold2} is replaced with $p>1\vee d/\beta$ in \eqref{eq:equivalencesobolev1}--\eqref{eq:equivalencesobolev2}.
  
\item The upper bounds on $p$ in \eqref{eq:eqsobold1}--\eqref{eq:eqsobold2} are replaced with $p<d/(\alpha s)\wedge d/((d-\beta)\vee0)$ in \eqref{eq:equivalencesobolev1}--\eqref{eq:equivalencesobolev2}.
\end{enumerate}
The lower bounds on $p$ in \eqref{eq:eqsobold1}--\eqref{eq:eqsobold2} in \cite{Merz2021,BuiDAncona2023} and in \eqref{eq:equivalencesobolev1}--\eqref{eq:equivalencesobolev2} in Theorem~\ref{eqsob} arise in the reversed Hardy inequality in \cite[Proposition~3]{Merz2021} or \cite[Theorem~4.2]{BuiDAncona2023} and the generalized Hardy inequality \cite[Proposition~1]{Merz2021} involving $L_\kappa$, and their analogs in Theorems~\ref{thm-difference} and~\ref{thm-HardyIneq} involving $\La$.

The different upper bounds on $p$ in \eqref{eq:eqsobold1}--\eqref{eq:eqsobold2} come from the reversed Hardy inequality requiring $p<d/\max\{\delta_\kappa,0\}$ and the generalized Hardy inequality requiring $p<d/\max\{\alpha s+\delta_\kappa,0\}$. In the present situation, the reversed Hardy inequality holds for all $p<\infty$ and the generalized Hardy inequality holds for $p<d/(\alpha s)$. In particular, these two upper bounds are independent of $\beta$, which is due to the non-symmetry of $\La$ and the resulting sharp bounds for the associated semigroup in \eqref{eq:heatkernel}. 
Additionally, due to our use of square function estimates (Theorem~\ref{squarefunctions}), we also assume $p<d/((d-\beta)\vee0)$ for the validity of \eqref{eq:equivalencesobolev1}--\eqref{eq:equivalencesobolev2}.

Another novelty compared to the previous results on equivalence of Sobolev norms is the additional restriction $\alpha s<\alpha-1$ for the validity of \eqref{eq:equivalencesobolev1}, which is due to the gradient in the perturbation $\kappa|x|^{-\alpha}x\cdot\nabla$ as we discussed in the proof of Theorem~\ref{eqsob} above.

\subsubsection{Impact of Theorem~\ref{eqsob}}
\label{ss:impact}

Homogeneous operators like $\La$ or $L_\kappa$ frequently arise as model operators or scaling limits in more complicated problems. Hence, their analysis is paramount to advance studies of the more complicated problems they originate from. In many applications, one has to deal with functions, e.g., powers greater than one, of these operators, which are, however, difficult to handle. The equivalence of Sobolev norm results in \cite{Killipetal2018,Franketal2021,Merz2021,BuiDAncona2023,FrankMerz2023,BuiMerz2023} and Theorem~\ref{eqsob} here are useful to overcome this obstacle as they allow to replace the usually difficult to handle operator functions of $\La$ or $L_\kappa$ with easier to handle operator functions of $(-\Delta)^{\alpha/2}$.

\smallskip
For instance, the equivalence of Sobolev norms for $L_\kappa$ established in \cite{Killipetal2018} for $\alpha=2$ was paramount for Killip, Miao, Murphy, Visan, Zhang, and Zheng \cite{Killipetal2017,Killipetal2017T} to study global well-posedness and scattering of nonlinear Schr\"odinger equations with the inverse-square potentials $\kappa/|x|^2$. 

In mathematical physics, the equivalence of Sobolev norms for $L_\kappa$ with $\alpha=1$ in \cite{Franketal2021} was crucial, e.g., to study the electron density of large, relativistically described atoms close to the nucleus, and, in particular, to prove Lieb's strong Scott conjecture \cite{Lieb1981,Simon1984} in \cite{Franketal2020P,MerzSiedentop2022}; see also \cite{Franketal2023} for a shorter proof and the recent review \cite{Franketal2023T} for an introduction to the Scott conjectures.
We explain this point in more detail to outline potential applications of Theorem~\ref{eqsob}. The strong Scott conjecture asserts that the (appropriately rescaled) probability density function of finding one of $Z$ electrons in an atom of nuclear charge $Z$ at a position $x$ on the length scale $Z^{-1}$---called one-particle ground state density---is, to leading order as $Z\to\infty$, described by the probability density function of the same atom, but where the electron-electron interactions are neglected. In this scenario, the energy of a single electron is described by $\sqrt{1-\Delta}-1+\kappa/|x|$ in $L^2(\R^3)$ where $\sqrt{1-\Delta}-1$ describes the kinetic energy of the electron and is a bounded perturbation of $\sqrt{-\Delta}$, and where $\kappa<0$ is the effective coupling strength between the electron and the nucleus. Thus, the resolution of the strong Scott conjecture can be interpreted as the successful derivation of an effective one-particle equation: The prohibitively difficult study of a genuine many-particle quantity---the ground state density---is reduced to the analysis of a one-particle quantity, namely the sums of squares of the eigenfunctions of the single-particle operator $\sqrt{1-\Delta}-1+\kappa/|x|$.
Let us explain why the equivalence of Sobolev norms is crucial to prove the strong Scott conjecture. On a technical level, the proof in \cite{Franketal2020P} relies on first order perturbation theory\footnote{For single or finitely many eigenvalues, this is known as Feynman--Hellmann theorem.} for the infinitely many eigenvalues of $\sqrt{1-\Delta}-1+\kappa|x|^{-1}-\lambda U(x)$ for a sufficiently well-behaved perturbation $U:\R^d\to\R$, regarded as functions of $\lambda\in\R$ in a neighborhood of zero. Among others, this perturbation theory requires to prove relative trace class bounds of the perturbation $U$ with respect to powers of $\sqrt{1-\Delta}-1+\kappa|x|^{-1}$; more concretely, one has to show that $(\sqrt{1-\Delta}-1+\kappa|x|^{-1}+M)^{-s}U(\sqrt{1-\Delta}-1+\kappa|x|^{-1}+M)^{-s}$ is trace class for some $s>1/2$ and $M>0$ such that $\sqrt{1-\Delta}-1+\kappa|x|^{-1}+M$ is a positive and hence invertible operator\footnote{Actually, one has to show this relative trace class condition only in every fixed angular momentum channel, i.e., when the operator is restricted to the space of square-integrable functions of the form $f(|x|)Y_\ell(x/|x|)$, where $Y_\ell$ denotes the $\ell$-th spherical harmonic on the unit sphere.}. To verify the relative trace class condition, one aims to replace $(\sqrt{1-\Delta}-1+\kappa|x|^{-1}+M)^{-s}$ with the Fourier multiplier $(\sqrt{1-\Delta}-1+M)^{-s}$. While this replacement is straightforward for $s\leq1/2$ (by using Hardy's inequality) justifying it for $s>1/2$ crucially relies on the equivalence of Sobolev norms in \cite{Franketal2021} (i.e., the analog of Theorem~\ref{eqsob} where the gradient perturbation $|x|^{-\alpha}x\cdot\nabla$ is replaced with the Hardy potential $|x|^{-\alpha}$).

In the context of the present work, many-particle equations involving $(-\Delta)^{\alpha/2}+b(x)\cdot\nabla$ arise in physics e.g., to describe plasmas or stellar matter, see, e.g., \cite{Dolbeault1991,BouchutDolbeault1995}, statistical properties of laser light \cite{Risken1989}, diffusion in random media, see, e.g., \cite{Albeverioetal2003}, or in biology, e.g., to describe chemotaxis by the Patlak--Keller--Segel model, see, e.g., \cite{FournierJourdain2017,Tardy2024}. While these works are concerned with the case $\alpha=2$, it has become clear that physical, chemical, or biological systems where particles appear to move slowly for certain time intervals and suddenly fly as in a jet flow should rather be described by fractional Laplacians; see, e.g., \cite{Schertzeretal2001} for references pointing to research in atmospheric science, anomalous diffusion, geophysics, and maritime science.
The mean-field limit and propagation of chaos of such many-particle equations---that is, the derivation of effective, nonlinear effective one-particle equations describing these many-particle systems---first envisioned by Boltzmann, and mathematically first formalized by Kac \cite{Kac1956} and McKean \cite{McKean1966,McKean1967}\footnote{See \cite{Aziz1969Vol2} for a reprint of \cite{McKean1967}.}, has been investigated in great detail for sufficiently regular drifts $b(x)$, excluding scaling-critical drifts considered here; for $\alpha=2$, see, e.g., the lecture notes \cite{Sznitman1991,Meleard1996}, the recent works \cite{Bolleyetal2010,Mischleretal2015,Cattiaux2024,Duongetal2025}, and \cite{ChaintronDiez2022} for a review; for $\alpha\in(1,2)$, see, e.g., the recent works \cite{Cavallazzi2025,Haoetal2024,Haoetal2024P}.
The mean-field limit equations---called nonlinear Fokker--Planck equations---have been studied in the fundamental works of McKean \cite{McKean1966} and Vlasov \cite{Vlasov1968}---see also \cite{Funaki1984,Sznitman1991} and the monograph \cite{CarmonaDelarue2018Vol1,CarmonaDelarue2018Vol2}---for $\alpha=2$ and \cite{BarbuRockner2024} for $\alpha\in(1,2)$; 
see also \cite{CaffarelliVasseur2010,Silvestre2012,Silvestre2012H,WeiTian2015} for the analysis of the linear Fokker--Planck equation involving $(-\Delta)^{\alpha/2}-b(x)\cdot\nabla$ with $\alpha<2$. 
In this light, the results in the present paper will not only be useful to study nonlinear equations involving $\La$ but also to advance the analysis of many-particle equations involving $\La$ and to derive effective one-particle equations in the presence of critically singular drift terms like $|x|^{-\alpha}x\cdot\nabla$.

\subsection*{Organization}

In Section~\ref{s:preliminaries}, we recall kernel bounds for $\me{-t\La}$ and use them, together with a Phragm\'en--Lindel\"of argument, to prove bounds for complex-time heat kernels. We use these bounds to prove novel bounds for the kernels of $(t\La)^k\me{-t\La}$ and corresponding $L^p\to L^q$-estimates. We apply the latter bounds in Section~\ref{s:squarefunctions} to prove the square function estimates for $(\La)^{s}$ in Theorem~\ref{squarefunctions}.
In Section~\ref{s:newboundsdifferenceskernels}, we prove bounds for the kernel of the difference $t\Ln \me{-t\Ln}- t\La \me{-t\La}$.
These bounds are crucial to prove the reversed Hardy inequality (Theorem~\ref{thm-difference}) in Section~\ref{s:reversedhardy}.
In Section~\ref{s:generalizedhardy}, we prove the generalized Hardy inequality (Theorem~\ref{thm-HardyIneq}). In the appendices, we prove auxiliary results and discuss the necessity of $\alpha s<\alpha-1$ for the validity of \eqref{eq:equivalencesobolev1}.

\section{Estimates involving the heat kernel of $\Lambda_\kappa$}
\label{s:preliminaries}

In this section, we collect known pointwise estimates for the heat kernel $\me{-t\Lambda_\kappa}$ and derive novel bounds for the complex-time heat kernel and $(t\Lambda_\kappa)^k\me{-t\Lambda_\kappa}$ with $k\in\N$. 

To begin, we introduce some notation.
We denote the average of a measurable function $f$ over a measurable set $E$ with $0<|E|<\infty$ by
$$
\fint_E f(x)dx=\f{1}{|E|}\int_E f(x)dx.
$$
Given a ball $B$, we associate annuli $S_j(B):=2^{j}B\backslash 2^{j-1}B$ for $j=1, 2, 3, \ldots$. For $j=0$, we write $S_0(B)=B$.
For $p>0$, the Hardy-Littlewood maximal function $\mathcal{M}_p$ is defined as
$$
\mathcal{M}_p f(x) := \sup_{B\ni x}\Big(\f{1}{|B|}\int_B|f(y)|^p\,dy\Big)^{1/p}, \ x\in \mathbb{R}^d,
$$
where the supremum is taken over all balls $B$ containing $x$. When $p=1$, we write $\mathcal M$ instead of $\mathcal M_1$.

\subsection{Spatially averaged estimates}
\label{s:averagedestimates}

In this section, we prove spatially weighted estimates for a family of operators whose kernels behave similarly to the heat kernel bounds of $\Lambda_\kappa$, presented in \eqref{eq:heatkernel} below. We will use them to prove bounds for time-derivatives of the heat kernel of $\La$ in Proposition~\ref{thm-ptk} and the square function estimates in Section~\ref{s:squarefunctions} below.

We start with some auxiliary estimates, whose manual proofs we omit.

\begin{lemma}
  \label{lem1-Tt}
  Let $d\in\N$ and $\alpha>0$. Then the following estimates hold.
  \begin{enumerate}
  \item[(a)] For all $\epsilon\in (0,d)$ there exists $C=C_\epsilon>0$ so that for all $t>0$,
    $$
    \int_{B(0,t)}\Big(\f{t}{|x|}\Big)^{d-\epsilon} dx\leq C t^{d}.
    $$
    
  \item[(b)] For all $\epsilon>0$, there exists $C=C_\epsilon>0$ such that
    $$
    \int_{\Rd} \f{1}{t^{d/\alpha}}\Big(\f{t^{1/\alpha}+|x-y|}{t^{1/\alpha}}\Big)^{-d-\epsilon} dy \leq C
    $$
    for all $x\in\R^d$.
    
  \item[(c)] For all $\epsilon>0$, there exists $C=C_\epsilon>0$ such that
    $$
    \int_{\Rd} \f{1}{t^{d/\alpha}}\Big(\f{t^{1/\alpha}+|x-y|}{t^{1/\alpha}}\Big)^{-d-\epsilon}f(y) dy \leq C\mathcal Mf(x)
    $$
    for all $t>0$, $x\in\R^d$, and $f\in L^1_{\rm loc}(\mathbb R^d)$.
  \end{enumerate}
\end{lemma}

Using these estimates, we prove the following estimates.

\begin{theorem}
  \label{thm-Tt}
  Let $d\in\N$, $\alpha>0$, and $\beta>0$. Let $\{T_t\}_{t>0}$ be a family of linear operators on $L^2(\Rd)$ with their associated kernels $T_t(x,y)$. Assume that there exist $C>0$ and $\theta \in (-\vc,d)$ such that for all $t>0$ and $x,y \in \Rd\backslash \{0\}$,
  \begin{equation}\label{kernelTt}
    |T_t(x,y)|\leq Ct^{-d/\alpha}\Big(\f{t^{1/\alpha}+|x-y|}{t^{1/\alpha}}\Big)^{-d-\beta} \Big(1+\f{t^{1/\alpha}}{|y|}\Big)^\theta.
  \end{equation}
  Assume that $\f{d}{d-\theta}\vee 1<p\leq q< \vc$, where $d_\theta$ is defined as in \eqref{eq-d beta}. Then, for any ball $B$, every $t>0$, and $j\ge 2$, we have
  \begin{equation}\label{eq1-Tt}
    \Big(\fint_{S_j(B)}|T_tf|^q\Big)^{1/q}
    \lesssim \max\Big\{\Big(\f{r_B}{t^{1/\alpha}}\Big)^{d/p}, \Big(\f{r_B}{t^{1/\alpha}}\Big)^{d}\Big\} \Big(1+\f{2^jr_B}{t^{1/\alpha}}\Big)^{-d-\beta} \Big(\fint_B|f|^p\Big)^{1/p}
  \end{equation}
  for all $f\in L^p(\Rd)$ supported in $B$, and
  \begin{equation}\label{eq2-Tt}
    \Big(\fint_{B}|T_tf|^q\Big)^{1/q}
    \lesssim \max\Big\{\Big(\f{2^jr_B}{t^{1/\alpha}}\Big)^d,\Big(\f{2^jr_B}{t^{1/\alpha}}\Big)^{d/p} \Big\} \Big(1+\f{2^jr_B}{t^{1/\alpha}}\Big)^{-d-\beta} \Big(\fint_{S_j(B)}|f|^p\Big)^{1/p}
  \end{equation}
  for all $f\in L^p(\mathbb R^d)$ supported in $S_j(B)$.
\end{theorem}

\begin{proof}
Note that the estimate \eqref{kernelTt} boils down to 
\begin{equation*}
    |T_t(x,y)|\leq Ct^{-d/\alpha}\Big(\f{t^{1/\alpha}+|x-y|}{t^{1/\alpha}}\Big)^{-d-\beta},
  \end{equation*}
 as long as $\theta \in (-\vc, 0)$.
 
Hence, it suffices to prove the theorem for $\theta\in [0,d)$. In this case, $\f{d}{d-\theta}\vee 1 = \f{d}{d-\theta}$.

  Since $|x-y|\sim 2^jr_B$ whenever $x\in S_j(B)$ and $y\in B$, 
  \begin{equation}
    \label{eq1-proof of Tt estimates}
    \begin{aligned}
      \|T_tf\|_{L^q(S_j(B))}
      & \leq \left\{\int_{S_j(B)}\Big[\int_B  t^{-d/\alpha}\Big(\f{t^{1/\alpha}+|x-y|}{t^{1/\alpha}}\Big)^{-d-\beta}\Big(1+\f{t^{1/\alpha}}{|y|}\Big)^\theta|f(y)|dy\Big]^qdx\right\}^{1/q}\\
      & \leq t^{-d/\alpha} \Big(1+\f{2^jr_B}{t^{1/\alpha}}\Big)^{-d-\beta}\left\{\int_{S_j(B)}\Big[\int_B  \Big(1+\f{t^{1/\alpha}}{|y|}\Big)^\theta|f(y)|dy\Big]^qdx\right\}^{1/q}\\
      & \leq t^{-d/\alpha}|2^jB|^{1/q} \Big(1+\f{2^jr_B}{t^{1/\alpha}}\Big)^{-d-\beta} \int_B  \Big(1+\f{t^{1/\alpha}}{|y|}\Big)^\theta|f(y)|dy.
    \end{aligned}
  \end{equation}
  We now estimate the integral on the right hand side of \eqref{eq1-proof of Tt estimates}. We have
  \[
    \begin{aligned}
      & \int_B \Big(1+\f{t^{1/\alpha}}{|y|}\Big)^\theta|f(y)|dy \\
      & \quad = \int_{B\cap B(0,t^{1/\alpha})}   \Big(1+\f{t^{1/\alpha}}{|y|}\Big)^\theta|f(y)|dy+\int_{B\cap B(0,t^{1/\alpha})^c}   \Big(1+\f{t^{1/\alpha}}{|y|}\Big)^\theta|f(y)|dy \\
      & \quad =: I_1+I_2.
    \end{aligned}
  \]	
  By H\"older's inequality and Lemma~\ref{lem1-Tt} (since $\theta p'<d$),
  \[
    \begin{aligned}
      I_1 & \le \Big[\int_{B(0,t^{1/\alpha})}   \Big(1+\f{t^{1/\alpha}}{|x|}\Big)^{\theta p'}dy\Big]^{1/p'}\|f\|_{L^p(B)}\le Ct^{ \f{d}{\alpha p'}} |B|^{1/p}\Big(\fint_B|f|^p\Big)^{1/p}.
    \end{aligned}
  \]
  We proceed to consider $I_2$. Here, we have $\Big(1+\f{t^{1/\alpha}}{|y|}\Big)^\theta\sim 1$. Hence,
  $$
  \begin{aligned}
    I_2&\lesi \|f\|_{L^1(B)} \lesi |B| \Big(\fint_B|f|^p\Big)^{1/p}.
  \end{aligned}
  $$
  Consequently,
  \[
    \begin{aligned}
      \int_B  \Big(1+\f{t^{1/\alpha}}{|y|}\Big)^\theta|f(y)|dy&\lesi \Big[t^{ \f{d}{\alpha p'}} |B|^{1/p}+ |B|\Big]\Big(\fint_B|f|^p\Big)^{1/p}.
    \end{aligned}
  \]
  Inserting this into \eqref{eq1-proof of Tt estimates} gives
  \[
    \begin{aligned}
      \|T_tf\|_{L^q(S_j(B))}
      & \leq t^{-d/\alpha}|2^jB|^{1/q} \Big(1+\f{2^jr_B}{t^{1/\alpha}}\Big)^{-d-\beta}\Big[t^{ \f{d}{\alpha p'}} |B|^{1/p}+ |B|\Big]\Big(\fint_B|f|^p\Big)^{1/p}\\ 
      & \leq  |2^jB|^{1/q} \Big(1+\f{2^jr_B}{t^{1/\alpha}}\Big)^{-d-\beta}\Big[t^{- \f{d}{\alpha p}} |B|^{1/p}+ t^{-d/\alpha}|B|\Big]\Big(\fint_B|f|^p\Big)^{1/p}\\
      & \leq  |2^jB|^{1/q} \Big(1+\f{2^jr_B}{t^{1/\alpha}}\Big)^{-d-\beta}\max\Big\{\Big(\f{r_B}{t^{1/\alpha}}\Big)^{d/p},\Big(\f{r_B}{t^{1/\alpha}}\Big)^{d} \Big\}\Big(\fint_B|f|^p\Big)^{1/p}.
    \end{aligned}
  \]
  Similarly, we obtain
  \begin{equation}
    \label{eq2-proof of Tt estimates}
    \begin{aligned}
      \|T_tf\|_{L^q(B)}
      & \leq \left\{\int_{B}\Big[\int_{S_j(B)}  t^{-d/\alpha}\Big(\f{t^{1/\alpha}+|x-y|}{t^{1/\alpha}}\Big)^{-d-\beta}\Big(1+\f{t^{1/\alpha}}{|y|}\Big)^\theta|f(y)|dy\Big]^qdx\right\}^{1/q}\\
      &\leq t^{-d/\alpha} \Big(1+\f{2^jr_B}{t^{1/\alpha}}\Big)^{-d-\beta}\left\{\int_{B}\Big[\int_{S_j(B)}  \Big(1+\f{t^{1/\alpha}}{|y|}\Big)^\theta|f(y)|dy\Big]^qdx\right\}^{1/q}\\
      &\leq t^{-d/\alpha}|B|^{1/q} \Big(1+\f{2^jr_B}{t^{1/\alpha}}\Big)^{-d-\beta} \int_{S_j(B)}  \Big(1+\f{t^{1/\alpha}}{|y|}\Big)^\theta|f(y)|dy.
    \end{aligned}
  \end{equation}
  We now write
  \[
    \begin{aligned}
      & \int_{S_j(B)}  \Big(1+\f{t^{1/\alpha}}{|y|}\Big)^\theta|f(y)|dy \\
      & \quad = \int_{S_j(B)\cap B(0,t^{1/\alpha})}   \Big(1+\f{t^{1/\alpha}}{|y|}\Big)^\theta|f(y)|dy  + \int_{S_j(B)\cap B(0,t^{1/\alpha})^c}   \Big(1+\f{t^{1/\alpha}}{|y|}\Big)^\theta|f(y)|dy \\
      & \quad =: II_1+II_2.
    \end{aligned}
  \]	
  Similarly to the term $I_1$, by H\"older's inequality and Lemma~\ref{lem1-Tt},
  \[
    \begin{aligned}
      II_1 & \le \Big[\int_{B(0,t^{1/\alpha})}   \Big(1+\f{t^{1/\alpha}}{|x|}\Big)^{\theta p'}dy\Big]^{1/p'}\|f\|_{L^p(S_j(B))} \\
           &\le Ct^{ \f{d}{\alpha p'}} |2^jB|^{1/p}\Big(\fint_{S_j(B)}|f|^p\Big)^{1/p}.
    \end{aligned}
  \]
  We proceed to consider $II_2$. Here, $\Big(1+\f{t^{1/\alpha}}{|y|}\Big)^\theta\sim 1$. Therefore,
  $$
  \begin{aligned}
    II_2&\lesi \|f\|_{L^1(S_j(B))}\lesi |2^jB| \Big(\fint_{S_j(B)}|f|^p\Big)^{1/p}.
  \end{aligned}
  $$
  The estimates for $II_1$ and $II_2$ yield
  \[
    \begin{aligned}
      \int_{S_j(B)}  \Big(1+\f{t^{1/\alpha}}{|y|}\Big)^\theta|f(y)|dy&\lesi \Big[t^{ \f{d}{\alpha p'}} |2^jB|^{1/p}+ |2^jB|\Big]\Big(\fint_{2^jB}|f|^p\Big)^{1/p}.
    \end{aligned}
  \]
  Inserting this into \eqref{eq2-proof of Tt estimates} gives
  \[
    \begin{aligned}
      \|T_tf\|_{L^q(B)}
      & \leq t^{-d/\alpha}|B|^{1/q} \Big(1+\f{2^jr_B}{t^{1/\alpha}}\Big)^{-d-\beta}\Big[t^{ \f{d}{\alpha p'}} |2^jB|^{1/p}+ |2^jB|\Big]\Big(\fint_{S_j(B)}|f|^p\Big)^{1/p}\\ 
      & \leq  |B|^{1/q} \Big(1+\f{2^jr_B}{t^{1/\alpha}}\Big)^{-d-\beta}\Big[t^{- \f{d}{\alpha p}} |2^jB|^{1/p}+ t^{-d/\alpha}|2^jB|\Big]\Big(\fint_{S_j(B)}|f|^p\Big)^{1/p}\\
      & \leq  |B|^{1/q} \Big(1+\f{2^jr_B}{t^{1/\alpha}}\Big)^{-d-\beta}\max\Big\{\Big(\f{2^jr_B}{t^{1/\alpha}}\Big)^{d/p},\Big(\f{2^jr_B}{t^{1/\alpha}}\Big)^{d} \Big\}\Big(\fint_{S_j(B)}|f|^p\Big)^{1/p}.
    \end{aligned}
  \]
  This concludes the proof.
\end{proof}

We now prove similar estimates for kernels of operators which are formally adjoint to those with integral kernels obeying the assumptions in Theorem~\ref{thm-Tt}.

\begin{theorem}
  \label{thm-St}
  Let $d\in\N$, $\alpha>0$, and $\beta>0$. Let $\{S_t\}_{t>0}$ be a family of linear operators on $L^2(\Rd)$ with their associated kernels $S_t(x,y)$. Assume that there exist $C>0$ and $\theta \in  {(-\vc,d)}$ such that for all $t>0$ and $x,y \in \Rd\backslash \{0\}$,
  \begin{equation}
    \label{kernelSt}
    |S_t(x,y)|\leq Ct^{-d/\alpha}\Big(\f{t^{1/\alpha}+|x-y|}{t^{1/\alpha}}\Big)^{-d-\beta} \Big(1+\f{t^{1/\alpha}}{|x|}\Big)^\theta.
  \end{equation}
  Assume that $1<p\leq q< \f{d}{\theta\vee 0}$ (with the convention $d/0 =\vc$). Then, for any ball $B$, every $t>0$, and $j\ge 2$, we have
  \begin{equation}\label{eq1-St}
    \Big(\fint_{S_j(B)}|S_tf|^q\Big)^{1/q}
    \lesssim \Big(\f{r_B}{t^{1/\alpha}}\Big)^d\Big(1+\f{t^{1/\alpha}}{2^jr_B}\Big)^{d/q} \Big(1+\f{2^jr_B}{t^{1/\alpha}}\Big)^{-d-\beta} \Big(\fint_B|f|^p\Big)^{1/p}
  \end{equation}
  for all $f\in L^p(\Rd)$ supported in $B$, and
  \begin{equation}\label{eq2-St}
    \Big(\fint_{B}|S_tf|^q\Big)^{1/q}
    \lesssim \Big(\f{2^jr_B}{t^{1/\alpha}}\Big)^d\Big(1+\f{t^{1/\alpha}}{r_B}\Big)^{d/q} \Big(1+\f{2^jr_B}{t^{1/\alpha}}\Big)^{-d-\beta} \Big(\fint_{S_j(B)}|f|^p\Big)^{1/p}
  \end{equation}
  for all $f\in L^p(\mathbb R^d)$ supported in $S_j(B)$.
\end{theorem}

\begin{proof}
Note that the estimate \eqref{kernelSt} boils down to 
\begin{equation*}
    |S_t(x,y)|\leq Ct^{-d/\alpha}\Big(\f{t^{1/\alpha}+|x-y|}{t^{1/\alpha}}\Big)^{-d-\beta}
  \end{equation*}
 as long as $\theta \in (-\vc, 0)$. Hence, it suffices to prove the theorem for $\theta \in [0,d)$. In this case, $\f{d}{\theta\vee 0} =\f{d}{\theta}$ with the convention $d/0=\vc$.

  Since $|x-y|\sim 2^jr_B$ whenever $x\in S_j(B)$ and $y\in B$, 
  \begin{equation}
    \label{eq1-proof of St estimates}
    \begin{aligned}
      \|S_tf\|_{L^q(S_j(B))}
      & \leq \left\{\int_{S_j(B)}\Big[\int_B  t^{-d/\alpha}\Big(\f{t^{1/\alpha}+|x-y|}{t^{1/\alpha}}\Big)^{-d-\beta}\Big(1+\f{t^{1/\alpha}}{|x|}\Big)^\theta|f(y)|dy\Big]^qdx\right\}^{1/q} \\
      & \leq t^{-d/\alpha} \Big(1+\f{2^jr_B}{t^{1/\alpha}}\Big)^{-d-\beta}\left\{\int_{S_j(B)}\Big[\int_B  \Big(1+\f{t^{1/\alpha}}{|x|}\Big)^\theta|f(y)|dy\Big]^qdx\right\}^{1/q}\\
      & \leq t^{-d/\alpha} \Big(1+\f{2^jr_B}{t^{1/\alpha}}\Big)^{-d-\beta} \Big[\int_{S_j(B)}   \Big(1+\f{t^{1/\alpha}}{|x|}\Big)^{\theta q}  dx\Big]^{1/q} \|f\|_{L^1(B)}.
    \end{aligned}
  \end{equation}
  On the other hand, by Lemma \ref{lem1-Tt},
  \begin{equation}
    \label{eq2- proof of St}
    \begin{aligned}
      \int_{S_j(B)} \Big(1+\f{t^{\frac1\alpha}}{|x|}\Big)^{\theta q}  dx
      & = \int_{S_j(B)\cap B(0,t^{\frac1\alpha})}\Big(1+\f{t^{\frac1\alpha}}{|x|}\Big)^{\theta q}  dx + \int_{S_j(B)\cap B(0,t^{1/\alpha})^c}\Big(1+\f{t^{\frac1\alpha}}{|x|}\Big)^{\theta q}  dx\\
      & \lesi  \int_{  B(0,t^{1/\alpha})}\Big(1+\f{t^{1/\alpha}}{|x|}\Big)^{\theta q}  dx + \int_{S_j(B)} dx \\
      & \lesi t^{d/\alpha} + |2^jB|
        \sim |2^jB|\Big(1+\f{t^{1/\alpha}}{2^jr_B}\Big)^d.
    \end{aligned}
  \end{equation}
  Inserting this into \eqref{eq1-proof of St estimates}, we get
  \[
    \begin{aligned}
      \|S_tf\|_{L^q(S_j(B))}
      & \lesi t^{-d/\alpha}|2^jB|^{1/q} \Big(1+\f{2^jr_B}{t^{1/\alpha}}\Big)^{-d-\beta} \Big(1+\f{t^{1/\alpha}}{2^jr_B}\Big)^{d/q} \|f\|_{L^1(B)} \\
      & \lesi t^{-d/\alpha}|2^jB| \Big(1+\f{2^jr_B}{t^{1/\alpha}}\Big)^{-d-\beta} \Big(1+\f{t^{1/\alpha}}{2^jr_B}\Big)^{d/q}|B| \Big(\fint_{B}|f|^p\Big)^{1/p} \\
      & \lesi |2^jB|^{1/q}\Big(\f{r_B}{t^{1/\alpha}}\Big)^d\Big(1+\f{t^{1/\alpha}}{2^jr_B}\Big)^{d/q} \Big(1+\f{2^jr_B}{t^{1/\alpha}}\Big)^{-d-\beta} \Big(\fint_{B}|f|^p\Big)^{1/p},
    \end{aligned}
  \]	
  where we used H\"older's inequality in the second inequality. This proves \eqref{eq1-St}.
  
  Similarly,
  \begin{equation}
    \label{eq2-proof of Tt estimates2}
    \begin{aligned}
      \|S_tf\|_{L^q(B)}
      & \leq \left\{\int_{B}\Big[\int_{S_j(B)}  t^{-d/\alpha}\Big(\f{t^{1/\alpha}+|x-y|}{t^{1/\alpha}}\Big)^{-d-\beta}\Big(1+\f{t^{1/\alpha}}{|x|}\Big)^\theta|f(y)|dy\Big]^qdx\right\}^{1/q} \\
      & \leq t^{-d/\alpha} \Big(1+\f{2^jr_B}{t^{1/\alpha}}\Big)^{-d-\beta}\left\{\int_{B}\Big[\int_{S_j(B)}  \Big(1+\f{t^{1/\alpha}}{|y|}\Big)^\theta|f(y)|dy\Big]^qdx\right\}^{1/q}\\
      & \leq t^{-d/\alpha} \Big(1+\f{2^jr_B}{t^{1/\alpha}}\Big)^{-d-\beta} \Big[\int_{B}   \Big(1+\f{t^{1/\alpha}}{|x|}\Big)^{\theta q}  dx\Big]^{1/q} \|f\|_{L^1(S_j(B))}.
    \end{aligned}
  \end{equation}
  Similarly to \eqref{eq2- proof of St},
  \[
    \int_{B} \Big(1+\f{t^{1/\alpha}}{|x|}\Big)^{\theta q}  dx \lesi |B| + t^{\f{d}{\alpha}}\sim |B|\Big(1+\f{t^{1/\alpha}}{r_B}\Big)^{d}.
  \]
  Therefore,
  \[
    \begin{aligned}
      \|S_tf\|_{L^q(B)}
      & \lesi t^{-d/\alpha}|B|^{1/q}\Big(1+\f{t^{1/\alpha}}{r_B}\Big)^{d/q} \Big(1+\f{2^jr_B}{t^{1/\alpha}}\Big)^{-d-\beta}  \|f\|_{L^1(S_j(B))} \\
      & \lesi t^{-d/\alpha}|B|^{1/q}\Big(1+\f{t^{1/\alpha}}{r_B}\Big)^{d/q} \Big(1+\f{2^jr_B}{t^{1/\alpha}}\Big)^{-d-\beta} |2^jB| \Big(\fint_{S_j(B)}|f|^p\Big)^{1/p} \\
      & \lesi |B|^{1/q}\Big(\f{2^jr_B}{t^{1/\alpha}}\Big)^d\Big(1+\f{t^{1/\alpha}}{r_B}\Big)^{d/q} \Big(1+\f{2^jr_B}{t^{1/\alpha}}\Big)^{-d-\beta}  \Big(\fint_{S_j(B)}|f|^p\Big)^{1/p}.
    \end{aligned}
  \]	
  This concludes the proof.
\end{proof}

\begin{theorem}
  \label{thm 2-Tt}
  Let $d\in\N$, $\alpha>0$, and $\beta>0$. 
  \begin{enumerate}[\rm (a)]
  \item Let $\{T_t\}_{t>0}$ be a family of linear operators as in Theorem \ref{thm-Tt} satisfying \eqref{kernelTt}. Then, for any $\f{d}{d-\theta}<p\leq q\le  \vc$, we have
    \[
      \|T_t\|_{p\to q} \lesi t^{-\f{d}{\alpha}(\f{1}{p}-\f{1}{q})}.
    \]
    
  \item Let $\{S_t\}_{t>0}$ be a family of linear operators as in Theorem \ref{thm-St} satisfying \eqref{kernelSt}. Then, for any $1\le p\leq q<  \f{d}{\theta}$, we have
    \[
      \|S_t\|_{p\to q} \lesi t^{-\f{d}{\alpha}(\f{1}{p}-\f{1}{q})}.
    \]
  \end{enumerate}
\end{theorem}

\begin{proof}
  The proof of (b) is similar to that of (a) with similar modifications as in the proof of Theorem \ref{thm-St}. Hence, we only provide the proof of (a). For $f\in L^p(\Rd)$, we have
  \[
    \begin{aligned}
      \|T_tf\|_{L^q(\Rd)}
      & \leq \left\{\int_{\Rd}\Big[\int_{\Rd}  t^{-d/\alpha}\Big(\f{t^{1/\alpha}+|x-y|}{t^{1/\alpha}}\Big)^{-d-\beta}\Big(1+\f{t^{1/\alpha}}{|y|}\Big)^\theta|f(y)|dy\Big]^qdx\right\}^{1/q} \\
      & \leq \left\{\int_{\Rd}\Big[\int_{B(0,t^{1/\alpha})}  t^{-d/\alpha}\Big(\f{t^{1/\alpha}+|x-y|}{t^{1/\alpha}}\Big)^{-d-\beta}\Big(1+\f{t^{1/\alpha}}{|y|}\Big)^\theta|f(y)|dy\Big]^qdx\right\}^{1/q} \\
      & \quad + \left\{\int_{\Rd}\Big[ \int_{B(0,t^{1/\alpha})^c}  t^{-\frac d\alpha}\Big(\f{t^{\frac1\alpha}+|x-y|}{t^{\frac1\alpha}}\Big)^{-d-\beta}\Big(1+\f{t^{1/\alpha}}{|y|}\Big)^\theta|f(y)|dy\Big]^qdx\right\}^{1/q} \\
      & =: I_1 +I_2.
    \end{aligned}
  \]
  We first consider $I_1$. By Minkowski's inequality,
  \[
    \begin{aligned}
      I_1 & \lesi  \int_{B(0,t^{1/\alpha})^c}  \left\{\int_{\Rd}\Big[t^{-d/\alpha}\Big(\f{t^{1/\alpha}+|x-y|}{t^{1/\alpha}}\Big)^{-d-\beta}\Big]^qdx\right\}^{1/q} \Big(1+\f{t^{1/\alpha}}{|y|}\Big)^\theta|f(y)|dy \\
          & \lesi t^{\f{d}{\alpha q}-\f{d}{\alpha}}\int_{B(0,t^{1/\alpha})^c}\Big(1+\f{t^{1/\alpha}}{|y|}\Big)^\theta|f(y)|dy.
    \end{aligned}
  \]
  Together with H\"older's inequality and Lemma~\ref{lem1-Tt} (since $\theta p'<d$), this estimate yields
  \[
    \begin{aligned}
      I_1 & \lesi t^{\f{d}{\alpha q}-\f{d}{\alpha}}\Big[\int_{B(0,t^{1/\alpha})}   \Big(1+\f{t^{1/\alpha}}{|x|}\Big)^{\theta p'} dy\Big]^{1/p'}\|f\|_{p} \\
          & \lesi t^{\f{d}{\alpha q}-\f{d}{\alpha}}t^{\f{d}{\alpha p'}}\|f\|_p
            \lesi t^{-\f{d}{\alpha}(\f{1}{p}-\f{1}{q})}\|f\|_p.
    \end{aligned}
  \]
  We proceed to consider $I_2$. Here, we have $\Big(1+\f{t^{1/\alpha}}{|y|}\Big)^\theta\sim 1$ whenever $y\in B(0,t^{1/\alpha})^c$. Hence,
  \[
    \begin{aligned}
      I_2 & \lesi \left\{\int_{\Rd}\Big[ \int_{\R^d}  t^{-d/\alpha}\Big(\f{t^{1/\alpha}+|x-y|}{t^{1/\alpha}}\Big)^{-d-\beta}|f(y)|dy\Big]^qdx\right\}^{1/q} \\
          & \sim \|\me{-t(-\Delta)^{\alpha/2}}|f|\|_{q}
            \lesi t^{-\f{d}{\alpha}(\f{1}{p}-\f{1}{q})}\|f\|_p.	
    \end{aligned}
  \]
  This concludes the proof.
\end{proof}

\subsection{Pointwise estimates}
\label{ss:kernelestimates}

In the following, $p_t(x,y)$ and $\tilde p_t(x,y)$ shall denote the kernels of $\me{-t\Lambda_\kappa}$ and $\me{-t\Ln}$, respectively, i.e.,
$$
p_t(x,y) := \me{-t\La}(x,y) \quad \text{and} \quad \tilde p_t(x,y):=\me{-t\Ln}(x,y).
$$
As indicated in Section~\ref{s:introduction}, we only consider $\kappa<\kappa_{\rm c}$, i.e., $\beta>(d+\alpha)/2$, since in that situation, $\me{-tL_\kappa}$ is a holomorphic semigroup in $L^2(\R^d)$ and a $C_0$ semigroup on $L^r(\R^d)$ for all $r\in[2,\infty)$.
As in \cite{Killipetal2018,Franketal2021,Merz2021,BuiDAncona2023,FrankMerz2023,BuiMerz2023}, a paramount role to prove the reversed and generalized Hardy inequalities is played by bounds for $p_t(x,y)$. According to \cite{Kinzebulatovetal2021,KinzebulatovSemenov2023F}, we have, 
\begin{align}
  \label{eq:heatkernel}
  p_t(x,y)
  \sim \tilde p_t(x,y) \cdot \left(1\wedge\frac{|y|}{t^{1/\alpha}}\right)^{\beta-d}
  \quad \text{for all $\alpha\in(1,2)$, $t>0$, and $x,y\in\R^d$}
\end{align}
with
\begin{align}
  \label{eq:heatkernelfree}
  \begin{split}
    \tilde p_t(x,y) & \sim t^{-d/\alpha}\left(\frac{t^{1/\alpha}}{t^{1/\alpha}+|x-y|}\right)^{d+\alpha} \quad \text{for all $\alpha\in(0,2)$, $t>0$, and $x,y\in\R^d$},
  \end{split}
\end{align}
see, e.g., \cite{BlumenthalGetoor1960}. 
When $\alpha=2$, corresponding estimates for $\me{-t\La}$ are unavailable to the best of our knowledge. However, in view of the factorization of the bounds in \eqref{eq:heatkernel} as product of the heat kernel with $\kappa=0$ times a singular weight, we expect
\begin{align}
  \label{eq:heatkernelalpha2}
  \begin{split}
    & t^{-d/2}\left(1\wedge\frac{|y|}{t^{1/2}}\right)^{\beta-d}\me{-c_1|x-y|^2/t}
      \lesssim \me{-t\La}(x,y)
      \lesssim t^{-d/2} \left(1\wedge\frac{|y|}{t^{1/2}}\right)^{\beta-d}\me{-c_2|x-y|^2/t} \\
    & \quad \text{for $\alpha=2$, all $t>0$, and all $x,y\in\R^d$},
  \end{split}
\end{align}
and certain $c_1\geq c_2>0$, depending only on $d$ and $\beta$.

\begin{remark}
  Metafune, Negro, Sobajima, and Spina \cite{Metafuneetal2017,Metafuneetal2018} proved matching upper and lower bounds for the heat kernel associated to
  \begin{align}
    \label{eq:metafuneoperator}
    -\Delta + (a-1)\sum_{j,\ell=1}^d\frac{x_j x_\ell}{|x|^2} \frac{\partial}{\partial x_j}\frac{\partial}{\partial x_\ell} + \kappa|x|^{-2}x\cdot\nabla + \frac{b}{|x|^2} \quad \text{in} \ L^2(\R^d,|x|^{(d-1-\kappa)/a-(d-1)}\,dx).
  \end{align}
  In particular, in this weighted $L^2$-space, the operator can be realized as a self-adjoint operator. Thus, unlike \eqref{eq:heatkernel} or \eqref{eq:heatkernelalpha2}, the corresponding heat kernel bounds are symmetric in $x$ and $y$.
  We refrain from attempting to prove an analog of Theorem~\ref{eqsob} for the operator in \eqref{eq:metafuneoperator} to keep the paper at a reasonable length. Such a proof likely requires a combination of techniques in \cite{Killipetal2018} and the present ones, together with more technical work due to the necessity to work in weighted Lebesgue spaces. Additionally, new ideas will be needed to deal with the case $a\neq1$.
\end{remark}

\begin{remark}
  \label{rem:heatkernelalphaleq1}
  In the critical case $\alpha=1$ and the supercritical case $\alpha\in(0,1)$, Kinzebulatov, Madou, and Semenov \cite[Theorem~1]{Kinzebulatovetal2024} showed $\me{-t\La}(x,y)\lesssim t^{-d/\alpha}(1\wedge|y|/t^{1/\alpha})^{\beta-d}$ for $t\lesssim1$, $\kappa<0$, $d\geq3$, and $\kappa=\Psi(\beta)$ defined in \eqref{eq:defbeta}\footnote{The terminologies "supercritical" and "critical" are motivated by the fact that a gradient perturbation $b\cdot\nabla$ is of the same weight as $\sqrt{-\Delta}$, while if $\alpha<1$, then, formally, $b\cdot\nabla$ dominates $(-\Delta)^{\alpha/2}$.}.
  This bound is insufficient for our purposes since it does not decay at all for $|x-y|\gtrsim t$ and is only stated for $t\lesssim1$. The decay in $|x-y|/t^{1/\alpha}$ is, however, crucial to estimate several integrals discussed below. We do not know if the estimate in \cite{Kinzebulatovetal2024} can be improved in this regard. Moreover, it remains to be explored if their estimate holds for $t\gtrsim1$, and if a corresponding lower bound holds. For Kato-class gradient perturbations, this possibility is ruled out by the observations in \cite[p.~181]{BogdanJakubowski2007}. There, the authors note that the heat kernel of $(-\Delta)^{\alpha/2}+|x|^{-1}x\cdot \nabla$ in $d=1$ equals $\me{-t\Ln}(y-(x+t))$ for all $x,y\in\R$, $t>0$, and $\alpha\in(0,2)$. Thus, taking $x=y$ and considering $t\searrow0$ and $\alpha\in(0,1]$, this kernel is bounded from above and below by $t^{-\alpha}$, whereas $\me{-t\Lambda_0}(t,x,x)$ is bounded from above and below by $t^{-1/\alpha}$. Thus, for $\alpha\in(0,1)$ the heat kernels of the fractional Laplacian with and without Kato-class gradient perturbations are not comparable with each other. This result, although proved for Kato-class gradient perturbations and not for the perturbation $|x|^{-\alpha}x\cdot\nabla$ considered in \cite{Kinzebulatovetal2024} and in the present paper, may be considered as a warning signal: There might not be a lower bound for $\me{-t\La}(x,y)$ which matches the upper bound obtained in \cite[Theorem~1]{Kinzebulatovetal2024} when $\alpha\in(0,1)$.
\end{remark}

\begin{remark}
  The heat kernels of $(-\Delta)^{\alpha/2}+b(x)\cdot\nabla$  
  with less singular, decaying drifts $b:\R^d\to\R^d$ are bounded from above and below by multiples of $\me{-t\Ln}(x,y)$, at least for small times; see \cite{BogdanJakubowski2007} for $\alpha\in(1,2)$ and, for instance, \cite{GarofaloLanconelli1990,Zhang1995,Zhang1997,KimSong2006} for $\alpha=2$.
\end{remark}

\smallskip
Given the above bounds for $\me{-t\La}$ and the estimates in Subsection~\ref{s:averagedestimates}, we now prove further estimates for $\me{-t\La}$. First, by Theorem \ref{thm 2-Tt}, we see that $\me{-t\Lambda_\kappa}$ is $L^p(\R^d)$-bounded whenever $p\in(d/\beta,\infty]$.
We now prove bounds for time derivatives of $\me{-t\La}$. To that end, we use that $\La$ generates, according to \cite[Proposition~8]{Kinzebulatovetal2021}, a holomorphic semigroup in $L^2(\R^d)$, and
prepare the following bounds for the holomorphic extension of $\me{-t\La}$ to the right complex half-plane. In the following, we write $\mathbb{C}_{\theta}:=\{z\in \mathbb{C}: |\arg z|<\theta\}$ for $\theta\in[0,\pi/2]$. Similarly to the proof of \cite[Proposition~8]{Kinzebulatovetal2021}, one can show that there exists $\mu\in(0,\pi/2)$ such that $\|(\Lambda_\kappa-\xi)^{-1}\|_{L^2(\R^d)\to L^2(\R^d)}\lesssim\frac{1}{|\xi|}$ for all $\xi\notin\C_\mu$. Together with the maximal accretivity of $\Lambda_\kappa$ and \cite[Section~8]{McIntosh1986}, this implies that $\Lambda_\kappa$ admits a bounded $H^\infty(\C_\omega\setminus\{0\})$ functional calculus on $L^2(\R^d)$ for any $\mu<\omega<\pi/2$. That is, if $f$ is a bounded holomorphic function on $\C_\omega\setminus\{0\}$, then $\|f(\Lambda_\kappa)\|_{L^2(\R^d)\to L^2(\R^d)}\lesssim_\omega \|f\|_\infty$. We fix $\theta\in(0,\pi/2-\mu)$ throughout the paper.

\begin{proposition}
  \label{heatkernelestimates-halfplane}
  Let $\alpha\in (1,2\wedge(d+2)/2)$, $\beta\in ((d+\alpha)/2,d+\alpha)$, and $\kappa=\Psi(\beta)$ be defined by~\eqref{eq:defbeta}. Then there exist $\epsilon_0>0$, $\epsilon\in (0,1)$, and a constant $C>0$ such that
  \begin{equation}
    \label{boundedp_t(x,y)-complex-alpha < 2}
    |p_z(x,y)| \le C |z|^{-d/\alpha} \Big(\f{|z|^{1/\alpha}}{|z|^{1/\alpha}+|x-y|}\Big)^{(d+\alpha)(1-\epsilon)}\Big(1+\f{|z|^{1/\alpha}}{|y|}\Big)^{d-\beta}
  \end{equation}
  for all $x,y \in \Rd$ and $z\in \mathbb{C}_{\epsilon\theta}$ with $\epsilon(0,\epsilon_0)$, where $p_z(x,y)$ is the kernel associated to the semigroup $\me{-z\Lambda_\kappa}$.
\end{proposition}

\begin{proof} 
  We use Davies' method \cite{Davies1990}, which relies on a Phragm\'en--Lindel\"of argument. For $z\in\mathbb{C}$, set
  \[
    w_z(x) = \Big(1+\f{z^{1/\alpha}}{|x|}\Big)^{-(d-\beta)}, \ \ \ x\in  \Rd.
  \]
  We first claim that there exists $C>0$ such that
  \begin{equation}
    \label{boundedp_t(x,y)-complex1'}
    \left|p_z(x,y)w_z(y)\right|\le \f{C}{|z|^{d/\alpha}}
  \end{equation}
  for all $x,y\in \Rd$ and $z\in\mathbb{C}_{\pi/4}$.
  To that end, let $w$ be a nonnegative locally integrable function on $\Rd$. We define  
  \[
    L_w^1(\Rd)=\Big\{f:  |f|_{L_w^1(\Rd)}:=\int_{\Rd}|f(x)|w(x)dx<\vc\Big\}.
  \]  
  Hence, the inequality \eqref{boundedp_t(x,y)-complex1'} is equivalent to
  $$
  \|\me{-z\Lambda_\kappa}\|_{L^1_{w_z^{-1}}(\Rd)\rightarrow L^\vc(\Rd)}\le \f{C}{|z|^{d/\alpha}}.
  $$
  For $z\in\mathbb{C}_{\epsilon\theta}$ we can write $z=3t+is$ for some $t\ge 0$ and $s\in \mathbb{R}$ such that $t\sim |z|$. Then, 
  \begin{align*}
    & \|\me{-z\Lambda_\kappa}\|_{L^1_{w_z^{-1}}(\Rd)\rightarrow L^\vc(\Rd)} \\
    & \quad \le \|\me{-t\Lambda_\kappa}\|_{L^2(\Rd)\rightarrow L^\vc(\Rd)}\|\me{-(t+is)\Lambda_\kappa}\|_{L^2(\Rd)\rightarrow L^2(\Rd)}\|\me{-t\Lambda_\kappa}\|_{L^1_{{w^{-1}_z}}(\Rd)\rightarrow L^2(\Rd)}.
  \end{align*}
  
  This, together with Theorem~\ref{thm 2-Tt} and the fact $\|\me{-is\Lambda_\kappa}\|_{L^2(\Rd)\rightarrow L^2(\Rd)}\lesi 1$ (since $\Lambda_\kappa$ has a bounded $H^\infty(\C_\theta\setminus\{0\})$ functional calculus on $L^2(\Rd)$ for $\beta>(d+\alpha)/2$), implies
  \begin{equation}
    \label{eq- L1vc norm}
    \|\me{-z\Lambda_\kappa}\|_{L^1_{w_z^{-1}}(\Rd)\rightarrow L^\vc(\Rd)}
    \lesi t^{-\f{d}{2\alpha}}\|\me{-t\Lambda_\kappa}\|_{L^1_{{w^{-1}_z}}(\Rd)\rightarrow L^2(\Rd)}.
  \end{equation}
  It remains to estimate $\|\me{-t\Lambda_\kappa}\|_{L^1_{{w^{-1}_z}}(\Rd)\rightarrow L^2(\Rd)}$. For $f\in L^1_{{w^{-1}_z}}(\Rd)$, we have, by \eqref{eq:heatkernel},
  $$
  \begin{aligned}
    \|\me{-t\Lambda_\kappa}f\|_{L^2(\Rd)}
    & \lesi \Big[\int_{\Rd}\Big|\int_{\Rd}   {t^{-d/\alpha}  \Big(\f{t^{1/\alpha}+|x-y|}{t^{1/\alpha}}\Big)^{-d-\alpha}w_z(y)^{-1}|f(y)|}dy\Big|^2dx\Big]^{1/2}.
  \end{aligned}
  $$
  By Minkowski's inequality,
  $$
  \begin{aligned}
    \|\me{-t\Lambda_\kappa}f\|_{L^2(\Rd)}
    & \lesi \int_{\mathbb{R}^d}\Big[\int_{\Rd}\Big| t^{-d/\alpha}   \Big(\f{t^{1/\alpha}+|x-y|}{t^{1/\alpha}}\Big)^{-d-\alpha}\Big|^2dx\Big]^{1/2} w_z(y)^{-1}|f(y)|dy \\
    & \lesi  {\sup_{y\in \Rd}}\Big[\int_{\mathbb{R}^d}\Big| t^{-d/\alpha}\Big(\f{t^{1/\alpha}+|x-y|}{t^{1/\alpha}}\Big)^{-d-\alpha}\Big|^2dx\Big]^{1/2}\|f\|_{L^1_{w_z^{-1}}}.
  \end{aligned}
  $$
  Using Lemma \ref{lem1-Tt}, we get
  \[
    \|\me{-t\Lambda_\kappa}f\|_{L^2(\Rd)}\lesi t^{-\f{d}{2\alpha}}\|f\|_{L^1_{w_z^{-1}}},
  \]
  which yields
  \[
    \|\me{-t\Lambda_\kappa}\|_{L^1_{{w^{-1}_z}}(\Rd)\rightarrow L^2(\Rd)}\lesi t^{-\f{d}{2\alpha}}.
  \]
  Inserting this into \eqref{eq- L1vc norm} gives
  $$
  \|\me{-z\Lambda_\kappa}\|_{L^1_{w_z^{-1}}(\Rd)\rightarrow L^\vc(\Rd)}\le t^{-\f{d}{\alpha}} \sim \f{1}{|z|^{d/\alpha}},
  $$
  which implies \eqref{boundedp_t(x,y)-complex1'}, i.e.,
  \[
    |p_z(x,y)|\lesi \f{1}{|z|^{d/\alpha}}\Big(1+\f{|z|^{1/\alpha}}{|y|}\Big)^{d-\beta}
  \]
  for all $x,y\in \Rd$ and $z\in\mathbb{C}_{\epsilon\theta}$. By the Phragm\'en--Lindel\"of argument in \cite[Proposition~3.3]{DuongRobinson1996} or \cite[Theorem~2.1]{Merz2022}, we get \eqref{boundedp_t(x,y)-complex-alpha < 2}.

  This completes our proof.
\end{proof}

For each $k\in \mathbb N$, we denote by $p_{t,k}(x,y)$ the kernel of $\Lambda_\kappa^k \me{-t\Lambda_\kappa}$. As a consequence of Proposition~\ref{heatkernelestimates-halfplane} and Cauchy's formula, we obtain the following result.

\begin{proposition}
  \label{thm-ptk}
  Let $\alpha\in (1,2\wedge(d+2)/2)$, $\beta\in ((d+\alpha)/2,d+\alpha)$, and $\kappa=\Psi(\beta)$ be defined by~\eqref{eq:defbeta}. Then for any $\epsilon\in(0,\epsilon_0)$ where $\epsilon_0$ is as in Proposition~\ref{heatkernelestimates-halfplane}, and $k\in \mathbb N$,
  \begin{equation*}
    |p_{t,k}(x,y)|\lesssim_{k,\epsilon} t^{-(k+d/\alpha)}\Big(\f{t^{1/\alpha}+|x-y|}{t^{1/\alpha}}\Big)^{-d-\alpha+\epsilon}\Big(1+\f{t^{1/\alpha}}{|y|}\Big)^{d-\beta}
  \end{equation*}
  for all $x,y \in \Rd$ and $t>0$.
\end{proposition}

\begin{proof}
  Applying Cauchy's formula, we get for every $t>0$ and $k\in \mathbb N$,
  \[
    \Lambda_\kappa^k\me{-t\Lambda_\kappa} = \f{(-1)^kk!}{2\pi i}\int_{|\xi-t|=\eta t} \me{-\xi \La}\f{d\xi}{(\xi-t)^{k+1}},
  \]
  where $\eta>0$ is small enough so that $\{\xi: |\xi-t|=\eta t\}\subset \mathbb C_{\tilde\epsilon\theta}$ with $\tilde \epsilon = \f{\epsilon}{d+\alpha}$, and the integral does not depend on the choice of $\eta$.
  We now apply Proposition \ref{heatkernelestimates-halfplane} and the fact that $|\xi|\sim |\xi-t|\sim t$ to deduce 
  \[
    \begin{aligned}
      |p_{t,k}(x,y)|
      & \le C_{k,\tilde\epsilon} t^{-(k+d/\alpha)} \Big(\f{t^{1/\alpha}+|x-y|}{t^{1/\alpha}}\Big)^{-(d+\alpha)(1-\tilde\epsilon)}\Big(1+\f{t^{1/\alpha}}{|y|}\Big)^{d-\beta}\\
      & \le C_{k,\epsilon} t^{-(k+d/\alpha)} \Big(\f{t^{1/\alpha}+|x-y|}{t^{1/\alpha}}\Big)^{-(d+\alpha-\epsilon)}\Big(1+\f{t^{1/\alpha}}{|y|}\Big)^{d-\beta}
    \end{aligned}
  \]
  for all $x,y\in \Rd$ and $t>0$.
\end{proof}

We also need estimates for $|\nabla_x\partial_t\me{-t\Lambda_0}(x,y)|$. To that end, we first record the following bounds for spatial derivatives, essentially contained in \cite[Lemma~5]{BogdanJakubowski2007}; see also \cite[(10), Lemma~1]{Kinzebulatovetal2021} for further estimates.

\begin{lemma}[{\cite{BogdanJakubowski2007,Kinzebulatovetal2021}}]
  \label{derivativeheatkernel}
  Let $d\in\N$ and $\alpha\in(0,2)$. Then, for all $t>0$ and $x,y\in\R^d$,
  \begin{align}
    \label{eq:derivativeheatkernel1}
    \begin{split}
      |\nabla_x \me{-t\Lambda_0}(x,y)|
      & \sim \frac{|x-y|}{t^{1/\alpha}} \cdot t^{-\frac{d+1}{\alpha}} \left(\frac{t^{1/\alpha}}{t^{1/\alpha}+|x-y|}\right)^{d+2+\alpha} \\
      & \lesssim t^{-\frac{d+1}{\alpha}} \left(\frac{t^{1/\alpha}}{t^{1/\alpha}+|x-y|}\right)^{d+1+\alpha}.
    \end{split}
  \end{align}
\end{lemma}

\begin{lemma}
  \label{thm-spacetimederivative}
  Let $d\in\N$ and $\alpha\in (0,2)$. Then, for any $\epsilon\in(0,\epsilon_0)$,
  \begin{align}
    \label{eq:spacetimederivative}
    |\nabla_x\partial_t\me{-t\Lambda_0}(x,y)|
    \lesssim_\epsilon t^{-\f{d+\alpha+1}{\alpha}}\Big(\f{t^{1/\alpha}}{t^{1/\alpha}+|x-y|}\Big)^{d+\alpha+1-\epsilon}.
  \end{align}
\end{lemma}

\begin{proof}
  Note that $\nabla_x\partial_t\me{-t\Lambda_0} = \partial_t\nabla_x\me{-t\Lambda_0}$. Hence, the proof is similar to that of Proposition~\ref{thm-ptk} by using Lemma~\ref{derivativeheatkernel} and hence we omit the details.
\end{proof}

\begin{lemma}
  \label{lem-gradient of heat kernel 2}
  Let $d\in\N$, $\alpha\in(0,2)$, and $0\le \gamma<(d+1)/\alpha$. Then,
  \begin{align}
    \label{eq:derivativerieszheat}
    |\nabla \Lambda_0^{-\gamma} \me{-t\Lambda_0}(x,y)|
    \lesi_\gamma t^{-\f{d+1-\gamma \alpha}{\alpha}}\Big(\f{t^{1/\alpha}}{t^{1/\alpha}+|x-y|}\Big)^{d+1-\alpha\gamma}.
  \end{align}
  Moreover, for $\alpha\in(1,2\wedge(d+2)/2)$, $0\leq\gamma<d/\alpha$, $\beta\in((d+\alpha)/2,d+\alpha)$, and $\kappa=\Psi(\beta)$ defined by~\eqref{eq:defbeta}, we have
  \begin{align}
    \label{eq:rieszheatkappa}
    |\La^{-\gamma}\me{-t\La}(x,y)| 
    \lesssim_\gamma t^{-\frac{d-\gamma\alpha}{\alpha}}\left(\frac{t^{1/\alpha}}{t^{1/\alpha}+|x-y|}\right)^{d-\alpha\gamma} \left(1\wedge\frac{|y|}{t^{1/\alpha}+|x-y|}\right)^{\beta-d}.
    \end{align}
\end{lemma}

Note that the statements in this lemma are consistent their limits as $t\to0$ in view of the Riesz kernel bounds for $\La$ in Lemma~\ref{riesz} below.

\begin{proof}
  We first prove~\eqref{eq:derivativerieszheat}.
  Using the formula
  \[
  \Lambda_0^{-\gamma} = c_\gamma \int_0^\vc s^\gamma \me{-s \Lambda_0} \f{ds}{s}, 
  \]
  we have
  \[
  \nabla \Lambda_0^{-\gamma} \me{-t\Lambda_0}(x,y)=c_\gamma \int_0^\vc s^\gamma \nabla \me{-(s+t) \Lambda_0} \f{ds}{s}.
  \]
  This, together with Lemma \ref{derivativeheatkernel}, implies
  \[
    \begin{aligned}
      | \nabla \Lambda_0^{-\gamma} \me{-t\Lambda_0}(x,y)|
      & \lesi  \int_0^\vc \f{s^\gamma}{(s+t)^{\f{d+1}{\alpha}}} \Big(\f{(s+t)^{1/\alpha}}{(s+t)^{1/\alpha}+|x-y|}\Big)^{d+1+\alpha} \f{ds}{s} \\
      & = \int_0^{t+|x-y|^\alpha}\ldots + \int^\vc_{t+|x-y|^\alpha}\ldots
        =: I_1 +I_2.
    \end{aligned}
  \]
  We have
  \[
    \begin{aligned}
      I_1 
      & \sim \int_0^{t+|x-y|^\alpha}  s^\gamma(t+s) \Big(\f{1}{(s+t)^{1/\alpha}+|x-y|}\Big)^{d+1+\alpha}\f{ds}{s} \\
      & \lesi (t^{1/\alpha}+|x-y|)^{\alpha(\gamma+1)}\Big(\f{1}{t^{1/\alpha}+|x-y|}\Big)^{d+1+\alpha}\\
      & \sim \Big(\f{1}{t^{1/\alpha}+|x-y|}\Big)^{d+1-\alpha\gamma}=\f{1}{t^{\f{d+1-\alpha\gamma}{\alpha}}}\Big(\f{t^{1/\alpha}}{t^{1/\alpha}+|x-y|}\Big)^{d+1-\alpha\gamma}.
    \end{aligned}
  \]
  Similarly, since $\alpha\gamma<d+1$, we have
  \[
    \begin{aligned}
      I_2 
      & \sim \int^\vc_{t+|x-y|^\alpha}  \f{s^\gamma}{s^{\f{d+1}{\alpha}}}\f{ds}{s}\\
      & \sim \Big(\f{1}{t^{1/\alpha}+|x-y|}\Big)^{d+1-\alpha\gamma}=\f{1}{t^{\f{d+1-\alpha\gamma}{\alpha}}}\Big(\f{t^{1/\alpha}}{t^{1/\alpha}+|x-y|}\Big)^{d+1-\alpha\gamma}.
    \end{aligned}
  \]
  This completes the proof of \eqref{eq:derivativerieszheat}.
  As the proof of \eqref{eq:rieszheatkappa} is similar, we omit the details.
\end{proof}

\section{Proof of square function estimates (Theorem~\ref{squarefunctions})}
\label{s:squarefunctions}

To prove Theorem~\ref{squarefunctions}, we recall two criteria for singular integrals to be bounded on Lebesgue spaces, which will play an important role in the proof of the boundedness of the square functions. The first theorem gives a criterion on the boundedness on $L^p(\Rd)$ spaces with $p\in (1,2)$, while the second one covers the range $p>2$.

\begin{theorem}
  \label{thm1-Auscher}
  Let $1\le p_0< 2$ and let $T$ be a sublinear operator which is bounded on $L^2(\Rd)$.  Assume that there exists a family of operators $\{\mathcal{A}_t\}_{t>0}$ satisfying that for every $j\ge 2$ and every ball $B$
  \begin{equation}
    \label{eq1-BZ}
    \Big(\fint_{S_j(B)}|T(I-\mathcal{A}_{r_B})f|^{2}\Big)^{1/2}\le
    \alpha(j)\Big(\fint_B |f|^{p_0} \Big)^{1/p_0}
  \end{equation}
  and
  \begin{equation}
    \label{eq1-BZ-bis}
    \Big(\fint_{S_j(B)}|\mathcal{A}_{r_B}f|^{2} \Big)^{1/2}\le
    \alpha(j)\Big(\fint_B |f|^{p_0} \Big)^{1/p_0}
  \end{equation}
  for all $f$ supported in $B$. If $\sum_j \alpha(j)2^{jd}<\vc$, then $T$ is bounded on $L^p(\mathbb{R}^d)$ for all $p\in (p_0,2)$.
\end{theorem}

\begin{theorem}
  \label{thm2-Auscher}
  Let $2< q_0\le \infty.$  Let $T$ be a bounded sublinear operator on $L^{2}(\mathbb{R}^d)$. Assume that there exists  a family of operators $\{\mathcal{A}_t\}_{t>0}$ satisfying
  \begin{eqnarray}
    \label{e1-Martell}
    \Big( \fint_{B} \big| T(I-\mathcal{A}_{r_B})f\big|^{2}dx\Big)^{1/2} \le
    C \mathcal{M}_{2}(f)(x)
  \end{eqnarray}
  and
  \begin{eqnarray}
    \label{e2-Martell}
    \Big( \fint_{B} \big| T\mathcal{A}_{r_B}f\big|^{q_0}dx\Big)^{1/q_0} \le
    C \mathcal{M}_{2}(Tf)(x)
  \end{eqnarray}
  \noindent
  for all balls $B$ with radius $r_B$, all $f \in C^{\infty}_c(\mathbb{R}^d) $ and all $x\in B$. Then $T$ is bounded on $L^p(\mathbb{R}^d)$ for all $2<p<q_0$.
\end{theorem}

For the proof of Theorems~\ref{thm1-Auscher}--\ref{thm2-Auscher}, see \cite{Auscher2007}.

\begin{proof}[Proof of Theorem~\ref{squarefunctions}] 
  As detailed below, the proof of the theorem only relies on the kernel estimates in Proposition \ref{thm-ptk}. It suffices to prove the theorem for $\beta \in ((d+\alpha)/2, d]$ since the proof for the case $\beta\in (d,d+M)$ is similar to the case $\beta=d$.
  Since the proof of \eqref{eq-square function} is similar to that of~\eqref{eq-square function duality} (in fact, even easier), we only prove~\eqref{eq-square function duality}.
  
  Fix $\delta \in (0,\alpha)$ such that $\delta/\alpha>\gamma$. 
  According to \cite[Proposition~8]{Kinzebulatovetal2021}, $\La$ generates a holomorphic semigroup in $L^2$. Therefore, $\Lambda_\kappa$ and $\La^*$ have a bounded functional calculus on $L^2(\Rd)$
  and $S_{\Lambda^*_\kappa,\gamma}$ is bounded on $L^2(\Rd)$ (see \cite{McIntosh1986}). We now prove the boundedness of $S_{\Lambda^*_\kappa,\gamma}$ and distinguish between whether $p<2$ or $p>2$.

  \medskip
  \textbf{Step 1: Proof of the $L^p$-boundedness for $1<p<2$}
  
  Fix $1<p\le 2$. Due to Theorem~\ref{thm1-Auscher}, it suffices to prove
  \begin{equation}
    \label{eq1-thmSquareFunctions}
    \Big(\fint_{S_j(B)}|S_{\Lambda^*_\kappa,\gamma}(I-\mathcal A_{r_B})f(x)|^{2}dx\Big)^{1/2}\lesi 2^{-(d+\delta)j}\Big(\fint_B|f(x)|^{p}dx\Big)^{1/p}
  \end{equation}
  and
  \begin{equation}
    \label{eq2-thmSquareFunctions}
    \Big(\fint_{S_j(B)}| \mathcal A_{r_B}f(x)|^{2}dx\Big)^{1/2}\lesi 2^{-(d+\delta)j}\Big(\fint_B|f(x)|^{p}dx\Big)^{1/p}
  \end{equation}
  for all $j\ge 2$ and for every function $f$ supported in $B$, where
  \[
    \mathcal A_{r_B} = I-(I-\me{-r_B^\alpha\Lambda^*_\kappa})^m, \ \ \  m>\f{d}{\alpha p'}+\f{\delta}{\alpha}+1.
  \]
  Since 
  \[
    \mathcal A_{r_B} =\sum_{k=1}^m C^m_k\me{-kr_B^\alpha\Lambda^*_\kappa},
  \]
  the estimate \eqref{eq2-thmSquareFunctions} follows directly from  \eqref{eq:heatkernel} and Theorem \ref{thm-St}.
  
  It remains to prove  \eqref{eq1-thmSquareFunctions}. To that end, we write
  \begin{equation}
    \label{eq1-squarefunctions}
    \begin{aligned}
      & \Big(\fint_{S_j(B)}|S_{\Lambda^*_\kappa,\gamma}(I-\me{-r_B^\alpha\Lambda^*_\kappa})^mf|^{2}dx\Big)^{1/2} \\
      & \quad \le \Big(\int_0^{r_B^\alpha}\left\|(t\Lambda^*_\kappa)^{\gamma}\me{-t\Lambda^*_\kappa}(I-\me{-r_B^\alpha\La^*})^mf\right\|_{L^{2}(S_j(B),\f{dx}{|S_j(B)|})}^2\f{dt}{t}\Big)^{1/2} \\
      & \qquad + \Big(\int_{r_B^\alpha}^\vc\left\|(t\Lambda^*_\kappa)^{\gamma}\me{-t\La^*}(I-\me{-r_B^\alpha\Lambda^*_\kappa})^mf\right\|_{L^{2}(S_j(B),\f{dx}{|S_j(B)|})}^2\f{dt}{t}\Big)^{1/2}\\
      & \quad := E_1+E_2.
    \end{aligned}
  \end{equation}
  We first take care of $I_1$. Note that
  \begin{align}
    \label{eq:operatorpowerheatkernel}
    (\Lambda^*_\kappa)^{\gamma}=\f{1}{\Gamma(1-\gamma)}\int_{0}^\vc u^{1-\gamma} \Lambda^*_\kappa \me{-u\Lambda^*_\kappa}\f{du}{u}.
  \end{align}  
  Using this and Minkowski's inequality, we have 
  $$
  \begin{aligned}
    E_1 & \lesi \Big(\int_0^{r_B^\alpha}\Big[\int_{0}^{r_B^\alpha}\Big(\f{u}{t}\Big)^{1-\gamma}\left\|t\Lambda^*_\kappa \me{-(t+u)\La}(I-\me{-r_B^\alpha\Lambda^*_\kappa})^mf\right\|_{L^{2}(S_j(B),\f{dx}{|S_j(B)|})}\f{du}{u}\Big]^2\f{dt}{t}\Big)^{1/2}\\
        & \quad + \Big(\int_0^{r_B^\alpha}\Big[\int_{r_B^\alpha}^\vc\Big(\f{u}{t}\Big)^{1-\gamma}\Big\|t\Lambda^*_\kappa \me{-(t+u)\Lambda^*_\kappa}(I-\me{-r_B^\alpha\Lambda^*_\kappa})^mf\Big\|_{L^{2}(S_j(B),\f{dx}{|S_j(B)|})}\f{du}{u}\Big]^2\f{dt}{t}\Big)^{1/2}\\
        & := E_{11} + E_{12}.
  \end{aligned}
  $$
  Using the identity 
  $$
  (I-\me{-r_B^\alpha\Lambda^*_\kappa})^m =\sum_{k=0}^m  (-1)^kC^m_k\me{-kr_B^\alpha\Lambda^*_\kappa}
  $$
  for certain numerical values $C^m_k$, we have
  \begin{align*}
    \footnotesize
    \begin{split}
      E_{11}&\lesi \Big(\int_0^{r_B^\alpha}\Big[\int_{0}^{r_B^\alpha}\Big(\f{u}{t}\Big)^{1-\gamma}\f{t}{t+u}\Big\|(t+u)\Lambda^*_\kappa \me{-(t+u)\Lambda^*_\kappa}f\Big\|_{L^{2}(S_j(B),\f{dx}{|S_j(B)|})}\f{du}{u}\Big]^2\f{dt}{t}\Big)^{1/2} \\
            & \quad + \sum_{k=1}^m\Big(\int_0^{r_B^\alpha}\Big[\int_{0}^{r_B^\alpha}\Big(\f{u}{t}\Big)^{1-\gamma}\f{t}{t+u+kr_B^\alpha}\Big\|(t+u+kr_B^\alpha)\Lambda^*_\kappa \me{-(t+u+kr_B^\alpha)\Lambda^*_\kappa}f\Big\|_{L^{2}(S_j(B),\f{dx}{|S_j(B)|})}\f{du}{u}\Big]^2\f{dt}{t}\Big)^{1/2}.
    \end{split}
  \end{align*}
  This, in combination with Proposition \ref{thm-ptk}, Theorem \ref{thm-St}, and the facts that $t+u\lesi r_B^{\alpha}$ and $(t+u+ kr_B^\alpha)^{1/\alpha}\sim r_B$ for $u, t\in (0,r_B^\alpha]$ and $k\ge 1$, yields
  $$
  \begin{aligned}
    E_{11}&\lesi  \Big(\int_0^{r_B^\alpha}\Big[\int_{0}^{r_B^\alpha}\Big(\f{u}{t}\Big)^{1-\gamma}\f{t}{t+u}\Big(\f{r_B}{(t+u)^{1/\alpha}}\Big)^d\Big(\f{2^jr_B}{(t+u)^{1/\alpha}}\Big)^{-d-\delta} \|f\|_{L^p(B,\f{dx}{|B|})}\f{du}{u}\Big]^2\f{dt}{t}\Big)^{\frac12} \\
          & \quad + \sum_{k=1}^m \Big(\int_0^{r_B^\alpha}\Big[\int_{0}^{r_B^\alpha}\Big(\f{u}{t}\Big)^{1-\gamma}\f{t}{r_B^\alpha} 2^{-j(d+\delta)}\|f\|_{L^p(B,\f{dx}{|B|})}\f{du}{u}\Big]^2\f{dt}{t}\Big)^{1/2}.
  \end{aligned}
  $$
  Hence, 
  $$
  \begin{aligned}
    E_{11} & \lesi 2^{-j(d+\delta)} \Big(\fint_B|f|^{p}dx\Big)^{1/p}.
  \end{aligned}
  $$
  
  To estimate $E_{12}$, we use
  \begin{equation}
    \label{eq2-squarefunction}
    (I-\me{-r_B^\alpha\Lambda_\kappa^*})^m = \int_0^{r_B^\alpha}\dots \int_0^{r_B^\alpha} (\Lambda_\kappa^*)^m\me{-(s_1+\dots+s_m)\Lambda_\kappa^*}d\vec{s},
  \end{equation}
  where $d\vec{s}:=ds_1\dots ds_m$, to write
  \begin{align*}
    \footnotesize
    \begin{split}
      E_{12}&\lesi \Big(\int_0^{r_B^\alpha}\Big[\int_{[0,r_B^\alpha]^m}\int^\vc_{r_B^\alpha}\Big(\f{u}{t}\Big)^{1-\gamma}\Big\|t(\Lambda_\kappa^*)^{m+1} \me{-(t+u+s_1+\ldots+s_m)\Lambda_\kappa^*}f\Big\|_{L^{2}(S_j(B),\f{dx}{|S_j(B)|})}\f{du}{u}d\vec{s}\Big]^2\f{dt}{t}\Big)^{1/2}.
    \end{split}
  \end{align*}
  In this case, $u\sim t+u+s_1+\ldots+s_m\ge r_B^\alpha$. Hence, by Proposition~\ref{thm-ptk} and Theorem~\ref{thm-St}, we obtain
  \begin{equation}
    \label{eq-I12}
    \begin{aligned}
      E_{12} & \lesi  \Big[\int^{r_B^\alpha}_0\Big(\int_{[0,r_B^\alpha]^m}\int^\vc_{r_B^\alpha}\Big(\f{u}{t}\Big)^{1-\gamma} \f{t}{u^{m+1}}\Big(\f{r_B}{u^{1/\alpha}}\Big)^{d}\Big(1+\f{u^{1/\alpha}}{2^jr_B}\Big)^{d/2} \\
             & \quad \Big(1+\f{2^jr_B}{u^{1/\alpha}}\Big)^{-(d+\delta)}\|f \|_{L^{p}(B,\f{dx}{|B|})}\f{du}{u}d\vec{s}\Big)^2\f{dt}{t}\Big]^{1/2}.
    \end{aligned}
  \end{equation}
  We see that
  $$
  \begin{aligned}
    & \int^\vc_{r_B^\alpha}\Big(\f{u}{t}\Big)^{1-\gamma} \f{t}{u^{m+1}}\Big(\f{r_B}{u^{1/\alpha}}\Big)^{d}\Big(1+\f{u^{1/\alpha}}{2^jr_B}\Big)^{d/2} \Big(1+\f{2^jr_B}{u^{1/\alpha}}\Big)^{-(d+\delta)}\f{du}{u} \\
    & \quad = \int_{r_B^\alpha}^{(2^jr_B)^\alpha}\ldots + \int^\vc_{(2^jr_B)^\alpha}\ldots \\
    & \quad \sim \int_{r_B^\alpha}^{(2^jr_B)^\alpha}\Big(\f{u}{t}\Big)^{1-\gamma} \f{t}{u^{m+1}}\Big(\f{r_B}{u^{1/\alpha}}\Big)^{d} \Big(\f{2^jr_B}{u^{1/\alpha}}\Big)^{-(d+\delta)}\f{du}{u} \\
    & \qquad + \int^\vc_{(2^jr_B)^\alpha}\Big(\f{u}{t}\Big)^{1-\gamma} \f{t}{u^{m+1}}\Big(\f{r_B}{u^{1/\alpha}}\Big)^{d}\Big(\f{u^{1/\alpha}}{2^jr_B}\Big)^{d/2} \f{du}{u} \\
    & \quad \lesi 2^{-j(d+\delta)}\Big(\f{r_B^\alpha}{t}\Big)^{1-\gamma}\f{t}{r_B^{\alpha(m+1)}} + 2^{-jd}\Big(\f{(2^jr_B)^\alpha}{t}\Big)^{1-\gamma}\f{t}{(2^jr_B)^{\alpha(m+1)}} \\
    & \quad \lesi 2^{-j(d+\delta)}\Big(\f{r_B^\alpha}{t}\Big)^{1-\gamma}\f{t}{r_B^{\alpha(m+1)}},
  \end{aligned}
  $$
  for all $m\ge 2$. Inserting this into \eqref{eq-I12} yields
  \[
    \begin{aligned}
      E_{12} & \lesi  2^{-j(d+\delta)}\|f\|_{L^{p}(B,\f{dx}{|B|})}\Big(\int^{r_B^\alpha}_0\Big[\int_{[0,r_B^\alpha]^m}\Big(\f{r_B^\alpha}{t}\Big)^{1-\gamma}\f{t}{r_B^{\alpha(m+1)}} d\vec{s}\Big]^2\f{dt}{t}\Big)^{1/2} \\
             &\lesi  2^{-j(d+\delta)} \Big(\fint_B|f|^{p}dx\Big)^{1/p}.
    \end{aligned}
  \]
  Collecting the estimates for $E_{11}$ and $E_{12}$ gives
  $$
  E_1\lesi 2^{-j(d+\delta)} \Big(\fint_B|f|^{p}dx\Big)^{1/p}.
  $$

  \smallskip
  We now take care of $E_2$. To that end, we write
  $$
  \begin{aligned}
    E_2&\lesi \Big(\int^\vc_{r_B^\alpha}\Big[\int_{0}^{r_B^\alpha}\Big(\f{u}{t}\Big)^{1-\gamma}\Big\|t\Lambda_\kappa^*  \me{-(t+u)\Lambda_\kappa^*}(I-\me{-r_B^\alpha\Lambda_\kappa^*})^mf\Big\|_{L^{2}(S_j(B),\f{d}{|S_j(B)|})}\f{du}{u}\Big]^2\f{dt}{t}\Big)^{1/2} \\
       & \quad + \Big(\int^\vc_{r_B^\alpha}\Big[\int_{r_B^\alpha}^\vc\Big(\f{u}{t}\Big)^{1-\gamma}\Big\|t\Lambda_\kappa^* \me{-(t+u)\Lambda_\kappa^*}(I-\me{-r_B^\alpha\Lambda_\kappa^*})^mf\Big\|_{L^{2}(S_j(B),\f{dx}{|S_j(B)|})}\f{du}{u}\Big]^2\f{dt}{t}\Big)^{1/2} \\
       & =: E_{21}+E_{22}.
  \end{aligned}
  $$
  Similarly to \eqref{eq-I12}, 
  \begin{align*}
  \footnotesize
    \begin{split}
      E_{12}&\lesi  \Big(\int_{r_B^\alpha}^\vc\Big[\int_{[0,r_B^\alpha]^m}\int_0^{r_B^\alpha}\Big(\f{u}{t}\Big)^{1-\gamma} \f{t}{t^{m+1}}\Big(\f{r_B}{t^{1/\alpha}}\Big)^{d}\Big(1+\f{t^{1/\alpha}}{2^jr_B}\Big)^{d/2}   \Big(1+\f{2^jr_B}{t^{1/\alpha}}\Big)^{-(d+\delta)} \|f \|_{L^{p}(B,\f{dx}{|B|})}\f{du}{u}d\vec{s}\Big]^2\f{dt}{t}\Big)^{\frac12} \\
            & \lesi  \Big(\int_{r_B^\alpha}^\vc\Big[ \Big(\f{r_B^{\alpha}}{t}\Big)^{1-\gamma} \f{r_B^{\alpha m}}{t^{m}}\Big(\f{r_B}{t^{1/\alpha}}\Big)^{d}\Big(1+\f{t^{1/\alpha}}{2^jr_B}\Big)^{d/2}   \Big(1+\f{2^jr_B}{t^{1/\alpha}}\Big)^{-(d+\delta)} \|f \|_{L^{p}(B,\f{dx}{|B|})}  \Big]^2\f{dt}{t}\Big)^{1/2}\\
            & \lesi 2^{-j(d+\delta)} \Big(\fint_B|f|^{p}dx\Big)^{1/p},
    \end{split}
  \end{align*}
  for all $m\ge 2$. Similarly as before,
  \begin{align*}
  \footnotesize
    \begin{split}
      E_{22} & \lesi \Big(\int_{r_B^\alpha}^\vc\Big[\int_{[0,r_B^\alpha]^m}\int^\vc_{r_B^\alpha}\Big(\f{u}{t}\Big)^{1-\gamma} \f{t}{(t+u)^{m+1}}\Big(\f{r_B}{(t+u)^{1/\alpha}}\Big)^{d}\Big(1+\f{(t+u)^{1/\alpha}}{2^jr_B}\Big)^{d/2} \\ 
             & \quad \times  \Big(1+\f{2^jr_B}{(t+u)^{1/\alpha}}\Big)^{-(d+\delta)} \|f \|_{L^{p}(B,\f{dx}{|B|})}\f{du}{u}d\vec{s}\Big]^2\f{dt}{t}\Big)^{1/2} \\
             & \lesi  \Big(\int_{r_B^\alpha}^\vc\Big[\int_{[0,r_B^\alpha]^m}\int^\vc_{r_B^\alpha}\Big(\f{u}{t}\Big)^{1-\gamma} \f{1}{t^{m-1}u}\Big(\f{r_B}{t^{1/\alpha}}\Big)^{d} \Big(1+\f{t^{1/\alpha}}{2^jr_B}\Big)^{d/2}\Big(1+\f{2^jr_B}{t^{1/\alpha}}\Big)^{-(d+\delta)} \|f \|_{L^{p}(B,\f{dx}{|B|})}\f{du}{u}d\vec{s}\Big]^2\f{dt}{t}\Big)^{\frac12} \\
             & \lesi \Big(\int_{r_B^\alpha}^\vc\Big[ \Big(\f{r_B^\alpha}{t}\Big)^{1-\gamma} \f{r_B^{\alpha m}}{t^{m-1}r_B^\alpha}\Big(\f{r_B}{t^{1/\alpha}}\Big)^{d} \Big(1+\f{t^{1/\alpha}}{2^jr_B}\Big)^{d/2} \Big(\f{2^jr_B}{t^{1/\alpha}}\Big)^{-(d+\delta)} \|f \|_{L^{p}(B,\f{dx}{|B|})}\Big]^2\f{dt}{t}\Big)^{1/2} \\
             & \lesi 2^{-j(d+\delta)} \Big(\fint_B|f|^{p}dx\Big)^{1/p},
    \end{split}
  \end{align*}
  for all $m\ge 2$. Taking all the estimates of $E_{21}, E_{22}$ and $E_1$ into account, we conclude that  
  $$
  \Big(\int_{S_j(B)}|S_{\Lambda^*_\kappa,\gamma}(I-\me{-r_B^\alpha\Lambda^*_\kappa})^mf|^{2}dx\Big)^{1/2}
  \lesi 2^{-j(d+\delta)}|2^jB|^{1/2}\Big(\fint_B|f|^{2}dx\Big)^{1/2}.
  $$
  This completes the proof of  \eqref{eq1-thmSquareFunctions}.

  \bigskip

  \textbf{Step 2: Proof of the $L^p$-boundedness for $2<p<\f{d}{d-\beta}$}
  
  By Theorem \ref{thm2-Auscher}, for any $q\in (2,\f{d}{d-\beta})$ it suffices to prove that
  \begin{eqnarray}
    \label{e1-SLa}
    \Big( \fint_{B} \left| S_{\Lambda^*_\kappa,\gamma}(I-\mathcal{A}_{r_B})f\right|^{2}dx\Big)^{1/2}
    \le C \mathcal{M}_{2}(f)(x),
  \end{eqnarray}
  and
  \begin{eqnarray}
    \label{e2-SLa}
    \Big( \fint_{B} \big| S_{\Lambda^*_\kappa,\gamma}\mathcal{A}_{r_B}f\big|^{q}dx\Big)^{1/q}
    \le	C \mathcal{M}_{2}(|S_{\Lambda_\kappa^*,\gamma}f|)(x)
  \end{eqnarray}
  for all balls $B$ with radius $r_B$, all $f \in C^{\infty}_c(\mathbb{R}^d) $ and all $x\in B$ with $\mathcal{A}_{r_B}=I-(I-\me{-r_B^\alpha\Lambda_\kappa^*})^m$, and $m\ge 2$.
  
  To prove \eqref{e1-SLa}, we write
  $$
  \begin{aligned}
    \begin{aligned}
      \Big(\fint_{B}|S_{\Lambda^*_\kappa,\gamma}(I-\me{-r_B^\alpha\Lambda^*_\kappa})^mf|^{2}dx\Big)^{1/2}
      \le \sum_{j\ge 0}\Big(\fint_{B}|S_{\Lambda^*_\kappa,\gamma}(I-\me{-r_B^\alpha\Lambda^*_\kappa})^mf_j|^{2}dx\Big)^{1/2}  
      =: \sum_{j=0}^\vc F_j,
    \end{aligned}
  \end{aligned}
  $$
  where $f_j=f\chi_{S_j(B)}$.
  
  For $j=0,1$, using the $L^2$-boundedness of $S_{\La,\gamma}$ and $\mathcal{A}_{r_B}$, we have
  $$
  F_j\lesi \mathcal{M}_{2}(f)(x).
  $$  
  Hence, it suffices to prove that 
  \begin{equation}\label{eq- Fj}
    F_j\lesi 2^{-j\beta}\Big(\fint_{S_j(B)}|f|^{2}dx\Big)^{1/2}
  \end{equation}
  for all $j\ge 2$. To do this, for $j\ge 2$, we write
  \[
    \begin{aligned}
      & \Big(\fint_{B}|S_{\Lambda^*_\kappa,\gamma}(I-\me{-r_B^\alpha\Lambda^*_\kappa})^mf_j|^{2}dx\Big)^{1/2} \\
      & \quad \le \Big(\int_0^{r_B^\alpha}\left\|(t\Lambda_\kappa^*)^{\gamma}\me{-t\Lambda_\kappa^*}(I-\me{-r_B^\alpha\Lambda_\kappa^*})^mf_j\right\|_{L^{2}(B,\f{dx}{|B|})}^2\f{dt}{t}\Big)^{1/2} \\
      & \qquad + \Big(\int_{r_B^\alpha}^\vc\left\|(t\Lambda^*_\kappa)^{\gamma}\me{-t\Lambda^*_\kappa}(I-\me{-r_B^\alpha\Lambda^*_\kappa})^mf_j\right\|_{L^{2}(B,\f{dx}{|B|})}^2\f{dt}{t}\Big)^{1/2}.
    \end{aligned}
  \]
  At this stage, we can argue as in the proof of~\eqref{eq1-thmSquareFunctions} in Step~1. However, in this case, we will utilize \eqref{eq2-St} instead of \eqref{eq1-St}. By doing so, we arrive at the expression \eqref{eq- Fj}. As the proof follows a similar structure, we omit the details.
	
  It remains to prove \eqref{e2-SLa}. We first write
  $$
  \begin{aligned}
    & \Big(\int_B |S_{\Lambda^*_\kappa,\gamma}\mathcal A_{r_B}f(x)|^{q}dx\Big)^{\frac1q}
      = \Big[\int_B \Big( \int_0^\vc \Big|\sum_{k=1}^m C^m_k\me{-kr_B^\alpha\Lambda^*_\kappa}(t\Lambda^*_\kappa)^{\gamma}\me{-t\La^*}f(x)\Big|^2\f{dt}{t} \Big)^{q/2}dx\Big]^{\frac1q} \\
    & \quad \lesi\sum_{j\ge 0}\Big[\int_B \Big( \int_0^\vc \Big|\sum_{k=1}^m C^m_k\me{-kr_B^\alpha\La^*}[(t\La^*)^{\gamma}\me{-t\Lambda^*_\kappa}f\chi_{S_j(B)}](x)\Big|^2\f{dt}{t} \Big)^{q/2}dx\Big]^{1/q}
  \end{aligned}
  $$
  which, along with Minkowski's inequality, Proposition \ref{thm-ptk}, and \eqref{eq2-St} in Theorem~\ref{thm-St}, gives
  $$
  \begin{aligned}
    & \Big(\fint_B |S_{\Lambda^*_\kappa,\gamma}\mathcal A_{r_B} f(x)|^{q}dx\Big)^{1/q}\\
    & \quad \lesi \sum_{j\ge 0}  \Big( \int_0^\vc \Big\|\sum_{k=1}^m\me{-kr_B^\alpha\Lambda^*_\kappa}[(t\Lambda^*_\kappa)^{\gamma}\me{-t\Lambda^*_\kappa}f\chi_{S_j(B)}]\Big\|_{L^{q}(B,\f{dx}{|B|})}^2\f{dt}{t} \Big)^{1/2} \\
    & \quad \lesi \sum_{j\ge 0} 2^{-j\beta} \left( \int_0^\vc \left\|(t\Lambda^*_\kappa)^{\gamma}\me{-t\Lambda^*_\kappa}f\right\|_{L^2(S_j(B),\f{dx}{|S_j(B)|})}^2\f{dt}{t} \right)^{1/2} \\
    & \quad \lesi \sum_{j\ge 0} 2^{-j\beta} \Big(\fint_{2^jB} |S_{\Lambda^*_\kappa,\gamma}f(x)|^{2}dx\Big)^{1/2}.
  \end{aligned}
  $$
  This implies \eqref{e2-SLa}. Hence the proof of Step 2 is completed.

  \medskip
  
  Thus, we proved that the square function $S_{\La^*,\gamma}$ is bounded on $L^p(\Rd)$ for all $1<p<\vc$, i.e.,
  \[
    \|S_{\Lambda^*_\kappa,\gamma}f\|_{L^p(\Rd)}\lesi \|f\|_{L^p(\Rd)}.
  \]
  As we remarked at the beginning of this proof, the $L^p$-boundedness of $S_{\La,\gamma}$ is proved analogously. We new prove the reversed square function inequalities using duality. For the sake of brevity, we only prove $\|f\|_p\lesssim\|S_{\La,\gamma}f\|_p$; the proof for $\La^*$ is again similar.
  
  By functional calculus, for any $g\in L^{p'}(\Rd)$, we have
  $$
  \begin{aligned}
    \int_{\Rd} f(x)g(x)dx&=c(\gamma)\int_{\Rd} \int_0^\vc(t\Lambda_\kappa)^{2\gamma}\me{-2t\Lambda_\kappa}f(x)g(x)\f{dt}{t}dx,
  \end{aligned}
  $$
  where $c(\gamma)= \int_0^\vc t^{2\gamma}\me{-2t}\f{dt}{t}$. Using H\"older's inequality, we obtain
  $$
  \begin{aligned}
    \int_{\Rd} f(x)g(x)dx
    & = c(\gamma)\int_{\Rd} \int_0^\vc(t\Lambda_\kappa)^{\gamma}\me{-t\Lambda_\kappa}f(x)(t\Lambda_\kappa^*)^{\gamma}\me{-t\Lambda_\kappa^*}g(x)\f{dt}{t}dx\\
    & \lesi \int_{\Rd}S_{\Lambda_\kappa,\gamma}f(x)S_{\Lambda_\kappa^*,\gamma}g(x) dx\\
    & \lesi\|S_{\Lambda_\kappa,\gamma}f\|_{L^p(\Rd)}\|S_{\Lambda_\kappa^*,\gamma}g\|_{L^{p'}(\Rd)}.
  \end{aligned}
  $$
  By \eqref{eq-square function duality}, i.e., $\|S_{\Lambda_\kappa^*,\gamma}g\|_{L^{p'}(\Rd)}\lesi \|g\|_{L^{p'}(\Rd)}$, we get
  $$
  \int_{\Rd} f(x)g(x)dx\lesi \|S_{\Lambda_\kappa,\gamma}f\|_{L^p(\Rd)}\|g\|_{L^{p'}(\Rd)}.
  $$
  As a consequence,
  $$
  \|f\|_{L^p(\Rd)}\lesi \|S_{\Lambda_\kappa,\gamma}f\|_{L^p(\Rd)},
  $$
  which completes the proof.
\end{proof}

\section{Bounds for differences of kernels}
\label{s:newboundsdifferenceskernels}

The goal of this section is to prove pointwise estimates for the difference
\begin{align}
  \label{eq:defqt}
  Q_t(x,y) := t\Lambda_0 \me{-t\Lambda_0}(x,y) - t\Lambda_\kappa \me{-t\Lambda_\kappa}(x,y)
\end{align}
and $|(Q_t f)(x)|$ for any $f\in L^p(\R^d)$, $p>1\vee d/\beta$. These will be instrumental to prove the reversed Hardy inequality (Theorem~\ref{thm-difference}) via Schur tests in the ensuing section.
For the sake of concreteness, we always assume $\alpha\in(1,2)$, and $\beta\in((d+\alpha)/2,d+\alpha)$ from now on.

To that end, we use Duhamel's formula and integral bounds involving, for $t>0$, $x,y\in\R^d$, and $\gamma,\delta>0$ the functions
\begin{align}
  \label{eq:defT}
  T^{\gamma}_t(x,y) := \left(1\wedge\f{|y|}{t^{1/\alpha}}\right)^{\beta-d} \, t^{-d/\alpha}\left(\f{t^{1/\alpha}}{t^{1/\alpha}+|x-y|}\right)^{d+\gamma}
\end{align}
and
\begin{align}
  \label{eq:defH}
  H^{\gamma,\delta}_t(x,y) := t^{-\f{d+\delta}{\alpha}}\left(\f{t^{1/\alpha}}{t^{1/\alpha}+|x-y|}\right)^{d+\gamma}.
\end{align}
For $\delta=1$ and appropriate values of $\gamma$, these functions arise from estimating $t|\partial_t\me{-t\Lambda_\kappa}|$ and $t|\nabla_x\partial_t\me{-t\Lambda_0}(x,y)|$, respectively; see Propositions~\ref{thm-ptk} and~\ref{thm-spacetimederivative}.
We will formulate our bounds on $|Q_t(x,y)|$ and $|(Q_tf)(x)|$ in terms of the functions
\begin{align}
  \label{eq:defl}
  \begin{split}
    L^{\gamma,\delta}_t(x,y)
    & := \textbf{1}_{\{|y|\le t^{1/\alpha}\}} \Big(\f{|y|}{t^{1/\alpha}}\Big)^{\beta-d-(\alpha-\gamma)}\,t^{-\frac{d}{\alpha}}\Big( \f{t^{1/\alpha}}{t^{1/\alpha}+|x-y|}\Big)^{d+\gamma} \cdot t^{(1-\delta)/\alpha} \\
    & \quad + \textbf{1}_{\{|y|\le t^{1/\alpha}, |x|\sim |y|\}} \Big(\f{|y|}{t^{1/\alpha}}\Big)^{1+\beta-d-\alpha}\,t^{-\frac{d}{\alpha}} \cdot t^{(1-\delta)/\alpha} \\
    & \quad + \textbf{1}_{\{|y|\ge t^{1/\alpha}\}}\one_{|x-y|\geq(|x|\wedge|y|)/2} t^{-\frac{d}{\alpha}}\left(\frac{t^{1/\alpha}}{t^{1/\alpha}+|x-y|}\right)^{d+\gamma} |y|^{1-\delta}
  \end{split}
\end{align}
and
\begin{align}
  \label{eq:defm}
  M^{\gamma,\delta}_t(x,y) := \one_{|y|\ge t^{1/\alpha}} \one_{|x-y|\le (|x|\wedge |y|)/2} \cdot \Big(\f{|x|\vee |y|}{t^{1/\alpha}}\Big)^{\delta-\alpha} \f{1}{t^{d/\alpha}} \Big( \f{t^{1/\alpha}}{t^{1/\alpha}+|x-y|}\Big)^{d+\gamma}.
\end{align}
These two functions are similar to those in \cite[Lemma~3.1]{Franketal2021}. One major difference between the function $M^{\gamma,\delta}_t(x,y)$ here and its analog in \cite[p.~2295]{Franketal2021} is that the power $\delta-\alpha$ of the factor $|y|/t^{1/\alpha}$ is only $-\alpha$ in \cite{Franketal2021}.
In parts of our proofs below we will have $\delta=0$, while in other parts $\delta>0$ is necessary, which is essentially due to the gradient perturbation.

\smallskip
In the following subsection, we estimate integrals involving the functions $T_t^\gamma(x,y)$ and $H_t^{\gamma,\delta}$ in terms of the functions $L_t^\gamma(x,y)$ and $M_t^{\gamma,\delta}(x,y)$.
In the ensuing two subsections, we prove pointwise bounds for $|Q_t(x,y)|$ and, afterwards, for $|(Q_t f)(x)|$. The former bounds will be used to prove part~(1), while the latter will be used to prove part~(2) in Theorem~\ref{thm-difference}.

\subsection{Integral bounds}
\label{ss:integralbounds}

In this section, we estimate integrals involving the functions $T_t^\gamma(x,y)$ and $H_t^{\gamma,\delta}$. 
The estimates for $\delta=1$ will be crucial in Subsection~\ref{ss:qtpointwise}, while those for $\delta=\gamma$ will be important in Subsection~\ref{ss:qtintegrated}.
Our techniques are similar to those in \cite{FrankMerz2023,BuiMerz2023}.

The following lemma, whose proof we defer to Appendix~\ref{a:proofcomptwokernels}, will be useful for our endeavors.
\begin{lemma}
  \label{lem- composition of two kernels}
  Let $d\in\N$. Then for all $N>0$, $0<s<t$, and $x,y\in\R^d$, we have
  \begin{equation}
    \label{eq- last inequality}
    \int_{\R^d} dz\, (t-s)^{-d}\left(\frac{(t-s)}{(t-s)+|x-z|}\right)^{d+N}\, s^{-d}\left(\frac{s}{s+|z-y|}\right)^{d+N}\sim \f{1}{t^d}\left(\frac{t}{t+|x-y|}\right)^{d+N}.
  \end{equation}
\end{lemma}

In the following lemma, we consider integrals arising when studying $Q_t$ in the region $|x-y|<(|x|\wedge|y|)/2$, where we expect cancellation effects.

\begin{lemma}
  \label{lem- difference for alpha < 2}
  Let $\alpha\in (1,2]$, $\beta\in((d+\alpha)/2,d+\alpha)$, and $\gamma,\delta\in (0,\alpha)$, $t>0$, and $x,y\in\R^d$.
  Then,
  \begin{align}
    \label{eq:lem- difference for alpha < 2spacederivative W}
    \begin{split}
      & \one_{ |y|\ge t^{1/\alpha}} \one_{|x-y|\le (|x|\wedge |y|)/2}\int_{\R^d}\int_{0}^t T^{\gamma}_{t-s}(x,z)|z|^{-\alpha+1} H^{\gamma,\delta}_{s}(z,y)\,ds\,dz 
      \lesssim |y|^{1-\delta} M^{\gamma,\delta}_t(x,y).
    \end{split}
  \end{align}
\end{lemma}

Despite the cancellations introduced by taking the gradient of the heat kernel, we do not expect that \eqref{eq:lem- difference for alpha < 2spacederivative W} holds with $H^{\gamma,1}_s(z,y)$ and $M^{\gamma,0}_t$ instead of $H^{\gamma,\delta}_s(z,y)$ and $M^{\gamma,\delta}_t$. See Appendix~\ref{a:optimallowerbound} for an argument.

\begin{proof}
  By a scaling argument, it suffices to consider $t=1$. Set 
  \[
    S := \{(x,y)\in\R^d\times\R^d: |y|\ge 1, |x-y|\le (|x|\wedge |y|)/2\}.
  \]
  We write
  \[
    \begin{aligned}
      & \one_{ |y|\ge 1} \one_{|x-y|\le (|x|\wedge |y|)/2}\int\limits_{\R^d}\int_{0}^1 T^{\gamma}_{t-s}(x,z)|z|^{-\alpha+1} H^{\gamma,\delta}_{s}(z,y)\,ds\,dz \\
      & \quad = \one_{S}(x,y) \int_{|z|\leq|y|/8}\int_0^1 T^{\gamma}_{1-s}(x,z)|z|^{1-\alpha}  H^{\gamma,\delta}_{s}(z,y)\,ds\,dz \\
      & \qquad + \one_{S}(x,y) \int_{|z|\geq|y|/8} \int_0^1 T^{\gamma}_{1-s}(x,z)|z|^{1-\alpha} H^{\gamma,\delta}_{s}(z,y)\,ds\,dz \\
      & \quad =: F_1 + F_2.
    \end{aligned}
  \]
  To bound $F_2$, we use $|z|\ge |y|/8$ and $ |y| \ge 1$, Lemma~\ref{lem- composition of two kernels}, and obtain
  \begin{align}
    \label{eq:boundf2}
    \begin{split}
      F_2 & \lesi \f{1}{|y|^{\alpha-1}} \int_{0}^1 \int_{\Rd}(1-s)^{-\frac d\alpha}\Big(\f{(1-s)^{1/\alpha}}{(1-s)^{1/\alpha}+|x-z|}\Big)^{d+\gamma}s^{-\frac{d+\delta}{\alpha}}\Big(\f{s^{1/\alpha}}{s^{1/\alpha}+|z-y|}\Big)^{d+\gamma}\,dz\, ds \\
          & \lesi \f{1}{|y|^{\alpha-1}} \Big(\f{1}{1+|x-y|}\Big)^{d+\gamma} 
            \lesi |y|^{1-\delta}\Big(\f{1}{|y|}\Big)^{\alpha-\delta} \Big(\f{1}{1+|x-y|}\Big)^{d+\gamma}
            \lesi |y|^{1-\delta} M_1^{\gamma,\delta}(x,y).
    \end{split}
  \end{align}
	
  We now consider $F_1$. In this situation, we have $|x|\sim |y|$, $|x-z|\sim |x|$ and $|y-z|\sim |y|$ as $|z|\le |y|/8$ and $|x-y|\le (|x|\wedge |y|)/2$. Hence,
  \[
    \begin{aligned}
      F_1 & \sim \int_{B(0,|y|/8)}\int_{0}^1 s^{-\f{d+\delta}{\alpha}}(1-s)^{-\frac d\alpha} |z|^{1-\alpha} \Big(\f{|x|}{(1-s)^{1/\alpha}}\Big)^{-d-\gamma}\Big(\f{|y|}{s^{1/\alpha}}\Big)^{-d-\gamma} \Big(1+\f{s^{1/\alpha}}{|z|}\Big)^{d-\beta} ds dz \\
          & \sim \f{1}{|x|^{d+\gamma}|y|^{d+\gamma}}\int_{0}^1 (1-s) ^{\gamma/\alpha}s^{(\gamma-\delta)/\alpha}\int_{B(0,|y|/8)}|z|^{1-\alpha}\Big(1+\f{s^{1/\alpha}}{|z|}\Big)^{d-\beta} dz\,ds.
    \end{aligned}
  \]
  By Lemma \ref{lem1-Tt},
  \[
    \begin{aligned}
      \int_{B(0,|y|/8)}|z|^{-\alpha+1}\Big(1+\f{s^{1/\alpha}}{|z|}\Big)^{d-\beta} dz
      & \le \int_{B(0,1)} |z|^{-\alpha+1}\Big(1+\f{1}{|z|}\Big)^{d-\beta} dz \\
      & \quad + \int_{1\le |z|\le |y|/8}|z|^{-\alpha+1}\Big(1+\f{1}{|z|}\Big)^{d-\beta} dz \\
      & \sim \int_{B(0,1)}\f{1}{|z|^{d-\beta+\alpha-1}} dz  + \int_{1\le |z|\le |y|/8}|z|^{-\alpha+1} dz \\
      & \lesi 1+|y|^{d-\alpha+1}
        \sim |y|^{d-\alpha+1},
    \end{aligned}
  \]
  where in the last inequality we used $|y|\gtrsim 1$. Plugging this into the bound of $F_1$ and using $|x|\sim |y|\gtrsim 1$, we obtain
  \[
    \begin{aligned}
      F_1 & \lesi \f{|y|^{d-\alpha+1}}{|x|^{d+\gamma}|y|^{d+\gamma}}\int_0^1  (1-s)^{\gamma/\alpha}s^{(\gamma-\delta)/\alpha}\,ds
            \lesi |y|^{1-\delta} \cdot |y|^{\delta-\gamma-\alpha} \cdot \f{1}{|x|^{d+\gamma}} 
            \lesi |y|^{1-\delta} M^{\gamma,\delta-\gamma}_1(x,y).
    \end{aligned}
  \]
  Since $M^{\gamma,\delta-\gamma}_1(x,y)\leq M^{\gamma,\delta}_1(x,y)$, this completes the proof of~\eqref{eq:lem- difference for alpha < 2spacederivative W}.
\end{proof}

\begin{remark}
  \label{rem:boundf2}
  Note that in \eqref{eq:boundf2}, the integral over $|z|\in[(1-\epsilon_1)|y|,(1+\epsilon_2)|y|]^c$ for arbitrary but fixed $\epsilon_1,\epsilon_2>0$ is actually bounded by $M_1^{\gamma,0}(x,y)$ since in that case we can use
  \begin{align*}
    & |z| s^{-\frac{d+1}{\alpha}} \left(\frac{s^{1/\alpha}}{s^{\frac1\alpha}+|z-y|}\right)^{d+\gamma} 
      \lesssim s^{-\frac{d}{\alpha}} \left(\frac{s^{1/\alpha}}{s^{\frac1\alpha}+|z-y|}\right)^{d+\gamma-1} \ \text{for}\ |z|\in[(1-\epsilon_1)|y|,(1+\epsilon_2)|y|]^c.
  \end{align*}
  The remaining integrals are convergent whenever $\gamma\in(1,\alpha)$.
\end{remark}

The following lemma concerns the region where no cancellation effects are expected anymore. It will be important in Subsection~\ref{ss:qtintegrated}.

\begin{lemma}
  \label{lem- TH x equiv z}
  Let $\alpha\in (1,2]$, $\beta\in((d+\alpha)/2,d+\alpha)$, $\gamma\in (0,1)$, $t>0$, and $x,y\in\R^d$. Then,
  \begin{align}
    \begin{split}
      & \one_{ |y|\le t^{1/\alpha}} \int\limits_{|x|/16\le |z|\le 4|x|}\int_{0}^t T^{\gamma}_{t-s}(x,z)|z|^{-\alpha+1} H^{\gamma,\delta}_{s}(z,y)\,ds\,dz \\
      & \qquad + \one_{ |y|\ge t^{1/\alpha}} \one_{|x-y|\ge (|x|\wedge |y|)/2}\int\limits_{|x|/16\le |z|\le 4|x|}\int_{0}^t T^{\gamma}_{t-s}(x,z)|z|^{-\alpha+1} H^{\gamma,\delta}_{s}(z,y)\,ds\,dz \\
      & \quad \lesssim L^{\gamma,\delta}_t(x,y).
    \end{split}
  \end{align}
\end{lemma}

\begin{proof}
  By a scaling argument, it suffices to consider $t=1$. We first prove
  $$
  \one_{ |y|\ge 1} \one_{|x-y|\ge (|x|\wedge |y|)/2}\int\limits_{|x|/16\le |z|\le 4|x|}\int_{0}^1 T^{\gamma}_{1-s}(x,z)|z|^{-\alpha+1} H^{\gamma,\delta}_{s}(z,y)\,ds\,dz 
  \lesi L^{\gamma,\delta}_1(x,y).
  $$
  To do this, we break the integral with respect to $dz$ as
  \[
    \int\limits_{|x|/16\le |z|\le 4|x|} \ldots dz =\int\limits_{|x|/16\le |z|\le 4|x| \atop |z| \ge |y|/100} \ldots dz + \int\limits_{|x|/16\le |z|\le 4|x| \atop |z| < |y|/100} \ldots dz.
  \]
	
  \textbf{Subcase 1.1:} $|x|\sim |z| \ge |y|/100$. By Lemma~\ref{lem- composition of two kernels} and $|z|^{-\alpha+1}\lesi|y|^{-\alpha+1}\lesi |y|^{1-\gamma}$, we have
  \[
    \begin{aligned}
      & \one_{ |y|\ge 1} \one_{|x-y|\ge (|x|\wedge |y|)/2}\int\limits_{|x|/16\le |z|\le 4|x| \atop |z| \ge |y|/100}\int_{0}^1 T^{\gamma}_{1-s}(x,z)|z|^{-\alpha+1}  H^{\gamma,\delta}_{s}(z,y)\,ds\,dz \\
      & \quad \lesi |y|^{1-\alpha}\Big(\f{1}{1+|x-y|}\Big)^{d+\gamma}\int_0^1 s^{-\delta/\alpha}\, ds \\
       & \quad \lesi |y|^{1-\delta}\Big(\f{1}{1+|x-y|}\Big)^{d+\gamma}\\
       & \quad \lesi L^{\gamma,\delta}_1(x,y).
    \end{aligned}
  \]
	
  \textbf{Subcase 1.2:} $|x|\sim |z| < |y|/100$. In this case $|x-y|\sim |z-y|\sim |y|$. Hence,
  \[
    \begin{aligned}
    & \one_{ |y|\ge 1} \one_{|x-y|\ge (|x|\wedge |y|)/2}\int\limits_{|x|/16\le |z|\le 4|x| \atop |z| < |y|/100}\int_{0}^1 T^{\gamma}_{1-s}(x,z)|z|^{-\alpha+1} H^{\gamma,\delta}_{s}(z,y)\,ds\,dz \\
    & \quad \lesi \int\limits_{|x|/16\le |z|\le 4|x| \atop |z| < |y|/100}\int_{0}^1 T^{\gamma}_{1-s}(x,z)|z|^{-\alpha+1} \f{1}{s^{\f{d+\delta}{\alpha}}} \Big(\f{s^{1/\alpha}}{|y|}\Big)^{d+\gamma} ds\,dz.
    \end{aligned}
  \]
  By Lemma~\ref{lem1-Tt},
  \begin{equation}
    \label{eq1-appendix C}
    \begin{aligned}
      \int_{\Rd} T^{\gamma}_{1-s}(x,z)|z|^{-\alpha+1}dz
      & \lesi \int_{|z|\le (1-s)^{1/\alpha}}\f{1}{(1-s)^{d/\alpha}}\Big(\f{(1-s)^{1/\alpha}}{|z|}\Big)^{d-\beta}|z|^{1-\alpha}\,dz \\
      & \quad + \int_{|z|> (1-s)^{1/\alpha}}\f{1}{(1-s)^{d/\alpha}}\Big(\f{(1-s)^{1/\alpha}}{(1-s)^{1/\alpha}+|x-z|}\Big)^{d+\gamma}|z|^{1-\alpha}dz \\
      & \lesi (1-s)^{-1+1/\alpha}.
    \end{aligned}
  \end{equation}
  Hence,
  \[
    \begin{aligned}
      & \one_{ |y|\ge 1} \one_{|x-y|\ge (|x|\wedge |y|)/2}\int\limits_{|x|/16\le |z|\le 4|x| \atop |z| < |y|/100}\int_{0}^1 T^{\gamma}_{1-s}(x,z)|z|^{-\alpha+1} H^{\gamma,\delta}_{s}(z,y)\,ds\,dz \\
      & \quad \lesi \int_{0}^1 (1-s)^{-1+1/\alpha}\f{s^{(\gamma-1)/\alpha}}{|y|^{d+\gamma}} ds
        \lesi \f{1}{|y|^{d+\gamma}}
        \lesi L^{\gamma,\delta}_1(x,y).
    \end{aligned}
  \]
  We have proved that 
  $$
  \one_{ |y|\ge 1} \one_{|x-y|\ge (|x|\wedge |y|)/2}\int\limits_{|x|/16\le |z|\le 4|x|}\int_{0}^1 T^{\gamma}_{1-s}(x,z)|z|^{-\alpha+1} H^{\gamma,\delta}_{s}(z,y)\,ds\,dz
  \lesi L^{\gamma,\delta}_1(x,y).
  $$
	
  \medskip
  It remains to prove
  $$
  \one_{ |y|< 1} \int\limits_{|x|/16\le |z|\le 4|x|}\int_{0}^1 T^{\gamma}_{1-s}(x,z)|z|^{-\alpha+1} H^{\gamma,\delta}_{s}(z,y)\,ds\,dz 
  \lesi L^{\gamma,\delta}_1(x,y).
  $$
  We now break the integral with respect to $dz$ into four integrals as
  \[
    \begin{aligned}
      \int\limits_{|x|/16\le |z|\le 4|x|} \ldots dz
      & = \int\limits_{|x|/16\le |z|\le 4|x| \atop |y|/100<|z| <2|y|} \ldots dz + \int\limits_{|x|/16\le |z|\le 4|x| \atop |z|\ge t^{1/\alpha}} \ldots dz \\
      & \quad + \int\limits_{|x|/16\le |z|\le 4|x| \atop 2|y|\le |z| <t^{1/\alpha}} \ldots dz + \int\limits_{|x|/16\le |z|\le 4|x| \atop |z| \le  |y|/100} \ldots dz.
    \end{aligned}
  \]
	
  \textbf{Subcase 2.1: $|y|/100<|z|<2|y|$}. In this case, by Lemma~\ref{lem- composition of two kernels} and $|x|\sim |z|\sim |y|$, we have 
  \[
    \begin{aligned}
      \one_{ |y|< 1} \int\limits_{|x|/16\le |z|\le 4|x|\atop |y|/100<|z|<2|y|}\int_{0}^1 T^{\gamma}_{1-s}(x,z)|z|^{1-\alpha} H^{\gamma,\delta}_{s}(z,y)\,ds\,dz
      & \lesi L^{\gamma,\delta}_1(x,y).
    \end{aligned}
  \]
	
  \textbf{Subcase 2.2: $|z|\ge 2|y|$}. Similarly to Subcase 2.1, we have
  $$
    \begin{aligned}
      & \one_{ |y|< 1} \int\limits_{|x|/16\le |z|\le 4|x|\atop |z| \ge 1}\int_{0}^1 T^{\gamma}_{1-s}(x,z)|z|^{-\alpha+1} H^{\gamma,\delta}_{s}(z,y)\,ds\,dz \\
      & \quad \lesi |y|^{-\alpha+1}\Big(\f{1}{1+|x-y|}\Big)^{d+1+\gamma} 
      \lesi L^{\gamma,\delta}_1(x,y).
    \end{aligned}
  $$

  \textbf{Subcase 2.3: $|z|\le |y|/100$}. In this case, $|z-y|\sim |y|\gtrsim|x-y|$. Hence,
  \[
    \begin{aligned}
      & \one_{ |y|< 1} \int\limits_{|x|/16\le |z|\le 4|x| \atop |z|\le |y|/100}\int_{0}^1 T^{\gamma}_{1-s}(x,z)|z|^{-\alpha+1} H^{\gamma,\delta}_{s}(z,y)\,ds\,dz \\
      & \quad \lesi \int\limits_{|x|/16\le |z|\le 4|x|\atop |z|\le |y|/100}\int_{0}^1 T^{\gamma}_{1-s}(x,z)|z|^{-\alpha+1}\f{1}{s^{\f{d+\delta}{\alpha}}}\Big(\f{s^{1/\alpha}}{s^{1/\alpha}+|y|}\Big)^{d+\gamma} \left(\one_{s\in(0,1/2)} + \one_{s\in(1/2,1)}\right) dsdz \\
      & \quad =: E_1 + E_2.
    \end{aligned}
  \]
  To bound $E_2$, we use \eqref{eq1-appendix C}, and obtain
  \[
    \begin{aligned}
      E_2 & \lesi \Big(\f{1}{1+|x-y|}\Big)^{d+\gamma} \int_{1/2}^1 (1-s)^{-1+1/\alpha} ds 
          \lesi L^{\gamma,\delta}(x,y).
     \end{aligned}
  \]
  For the term $E_1$, we have
  \[
    \begin{aligned}
      E_1 & \lesi \int_0^{1/2}\int_{|x|/16\le |z|\le 4|x|} \Big(\f{1}{|z|}\Big)^{d-\beta}|z|^{-\alpha+1}\f{1}{s^{\f{d+\delta}{\alpha}}}\Big(\f{s^{1/\alpha}}{s^{1/\alpha}+|y|}\Big)^{d+\gamma} dz ds \\
          & \lesi \int_0^{1/2} \Big(\f{1}{|x|}\Big)^{d-\beta}|x|^{d-\alpha+1}\f{1}{s^{\f{d+\delta}{\alpha}}}\Big(\f{s^{1/\alpha}}{s^{1/\alpha}+|y|}\Big)^{d+\gamma} dz ds \\
          & \lesi |x|^{\beta+1-\alpha}\left[\int_0^{|y|^\alpha} \ldots +  \int_{|y|^\alpha}^{1/2} \ldots\right] \\
          & \lesi |y|^{\beta+1-d-\alpha}
            \lesi L^{\gamma,\delta}_1(x,y).
    \end{aligned}
  \]
  Therefore,
  $$
  \one_{ |y|< 1} \int\limits_{|x|/16\le |z|\le 4|x|}\int_{0}^1 T^{\gamma}_{1-s}(x,z)|z|^{-\alpha+1} H^{\gamma,\delta}_{s}(z,y)\,ds\,dz 
  \lesi L^{\gamma,\delta}_1(x,y).
  $$
  This completes our proof.
\end{proof}

\subsection{An estimate for $|Q_t(x,y)|$}
\label{ss:qtpointwise}

The estimate in this subsection will be used to prove part~(1) in Theorem~\ref{thm-difference} for $\alpha s<\alpha-1$.

\begin{proposition}
  \label{prop-difference}
  Let $\alpha\in(1,2)$, $\beta\in ((d+\alpha)/2,d+\alpha)$, and $\kappa=\Psi(\beta)$ be defined by \eqref{eq:defbeta}. Then, for any $\gamma\in(0,\alpha)$, $x,y\in\R^d$, and $t>0$,
  \begin{align}
    \label{eq1-DifferenceKernels}
    |Q_t(x,y)| \lesssim L^{\gamma,1}_t(x,y)+M^{\gamma,1}_t(x,y).
  \end{align}
\end{proposition}

\begin{proof}
  By Proposition \ref{thm-ptk},
  \begin{align*}
    |\widetilde{p}_{t}(x,y)|+ t|\widetilde{p}_{t,1}(x,y)| & \lesi H^{\gamma,1}_t(x,y) \quad \text{and} \\
    |p_{t}(x,y)| + t|p_{t,1}(x,y)| & \lesi T^{\gamma}_t(x,y)  
  \end{align*}
  for all $t>0$ and $x,y\in \Rd$. We now consider two cases.
  
  \medskip
  \noindent{ {\textbf{Case 1:} [$|y|\le t^{1/\alpha}$] OR [$|y|\ge t^{1/\alpha}$ and $|x-y|\ge (|x|\wedge |y|)/2$].}}

  Since in this case 
  \[
    H^{\gamma,1}_t(x,y)+T^{\gamma}_t(x,y) \lesi L^{\gamma,1}_t(x,y),
  \]
  we get $|Q_t(x,y)|\lesi L^{\gamma,1}_t(x,y)$ as desired.
  
  \medskip
  \noindent{ {\textbf{Case 2:} $ |y|\ge t^{1/\alpha}$ and $|x-y|< (|x|\wedge |y|)/2$.}}

  By Duhamel's formula,
  \begin{align}
    \label{eq:duhamel}
    \begin{split}
      \tilde{p}_{t}(x,y) - {p}_{t}(x,y)
      & = \kappa \int_0^t\int_{\Rd}{p}_{t-s}(x,z){|z|^{-\alpha}z\cdot\nabla_z}\tilde p_{s}(z,y)\,dz\,ds \\
      & = \kappa \int_0^{t/2}\int_{\Rd}{p}_{t-s}(x,z){|z|^{-\alpha}z\cdot\nabla_z}\tilde p_{s}(z,y)\,dz\,ds \\
      & \quad + \kappa \int_0^{t/2}\int_{\Rd}{p}_{s}(x,z){|z|^{-\alpha}z\cdot\nabla_z}\tilde{p}_{t-s}(z,y)\,dz\,ds.
    \end{split}
  \end{align}
  Differentiating both sides with respect to $t$ and multiplying by $t$ gives
  \begin{equation}
    \label{eq-Kato}
    \begin{aligned}
      {Q}_t(x,y)
      & = \kappa  t\int_{\mathbb{R}^d}{p}_{t/2}(x,z) {|z|^{-\alpha}z\cdot\nabla_z}\tilde p_{t/2}(z,y)\,dz \\
      & \quad + \kappa  t\int_0^{t/2}\int_{\mathbb{R}^d}{p}_{t-s,1}(x,z) {|z|^{-\alpha}z\cdot\nabla_z} \tilde p_{s}(z,y)\,dz\,ds \\
      & \quad + \kappa t\int_{t/2}^t\int_{\mathbb{R}^d}{p}_{t-s}(x,z) {|z|^{-\alpha}z\cdot\nabla_z}\tilde p_{s,1}(z,y)\,dz\,ds \\
      & =: I_1+I_2+I_3.
    \end{aligned}
  \end{equation}
  Without loss of generality, we now assume $t=1$.
  The term $I_1$ can be written as
  \[
    \begin{aligned}
      \kappa \int_{\mathbb{R}^d} p_{1/2}(x,z) {|z|^{-\alpha}z\cdot\nabla_z}\tilde p_{1/2}(z,y)dz
      & = 6\kappa \int_{1/3}^{1/2}\int_{\mathbb{R}^d}{p}_{1/2}(x,z) {|z|^{-\alpha}z\cdot\nabla_z}\tilde p_{1/2}(z,y)dzds.
    \end{aligned}
  \]
  By Lemma~\ref{lem- difference for alpha < 2} and \eqref{eq:derivativeheatkernel1}, we have $H^{\gamma,1}_{1-s}(z,y) \sim H^{\gamma,1}_{1/2}(z,y) \gtrsim |\nabla_z \tilde{p}_{1/2}(z,y)|$ and $T^{\gamma}_{s}(\cdot,\cdot)\sim T^{\gamma}_{1/2}(\cdot,\cdot)\gtrsim p_{1/2}(\cdot,\cdot)$ for $s\in [1/3,1/2]$.
  Thus,
  \[
    \begin{aligned}
      I_1
      & \lesi \int_{1/3}^{1/2}\int_{\mathbb{R}^d} T^{\gamma}_{1-s}(x,z) {|z|^{1-\alpha}}H^{\gamma,1}_s(z,y)\,dz\,ds
      \lesi \int_{0}^{1}\int_{\mathbb{R}^d} T^{\gamma}_{1-s}(x,z) {|z|^{1-\alpha}} H^{\gamma,1}_s(z,y)dz\,ds \\
      & \lesi M^{\gamma,1}_1(x,y).
    \end{aligned}
  \]
  
  We now estimate $I_2$. Since $1-s\sim t$ for $s\in (0,1/2)$,
  \[
    |{p}_{1-s,1}(x,z)|
    \sim (1-s)|{p}_{1-s,1}(x,z)|
    \lesi T^{\gamma}_{1-s}(x,z)
  \]
  and 
  $$
  |\nabla_z \tilde{p}_{s}(z,y)|\lesi H^{\gamma,1}_{s}(z,y)
  $$
  due to \eqref{eq:derivativeheatkernel1}.  
  Therefore, by Lemma \ref{lem- difference for alpha < 2},
  \[
    \begin{aligned}
      I_2 \lesi \int_0^{1/2} \int_{\mathbb{R}^d}T^{\gamma}_{1-s}(x,z) {|z|^{1-\alpha}}H^{\gamma,1}_s(z,y)dzds
      \lesi M^{\gamma,1}_1(x,y).
    \end{aligned}
  \]
  For the last term $I_3$, by Proposition \ref{thm-spacetimederivative},
  \[
    |\nabla_z \tilde p_{s,1}(z,y)|
    \lesi s^{-1}H^{\gamma,1}_s(z,y)
    \sim H^{\gamma,1}_s(z,y)
  \]
  for $s\in [1/2,1]$. This, together with $|p_{1-s}(x,z)|\lesi T_{1-s}^{\gamma}(x,z)$ and \eqref{eq:lem- difference for alpha < 2spacederivative W} in Lemma~\ref{lem- difference for alpha < 2}, implies
  \[
    \begin{aligned}
      I_3 \lesi \int_{1/2}^1 \int_{\mathbb{R}^d}T^{\gamma}_{1-s}(x,z) |z|^{-\alpha+1}H^{\gamma,1}_s(z,y)\,dz\,ds
      \lesi M^{\gamma,1}_1(x,y).
    \end{aligned}
  \]
  This completes our proof.
\end{proof}

\subsection{An estimate for $|(Q_t f)(x)|$}
\label{ss:prop-difference v2}
\label{ss:qtintegrated}

In this subsection, we prove an estimate, which we use to show part (2) in Theorem~\ref{thm-difference}.

\begin{proposition}
  \label{prop-difference v2}
  Let $\alpha\in(1,2)$, $\beta\in ((d+\alpha)/2,d+\alpha)$, and $\kappa=\Psi(\beta)$ be defined by \eqref{eq:defbeta}. Then, for any $\gamma\in(0,1)$, $x\in\R^d$, $t>0$, and $f\in L^p(\R^d)$ with $p>1\vee d/\beta$,
  \begin{align}
    \label{eq1-DifferenceKernels v2}
    \begin{split}
      |(Q_t f)(x)|
      & \lesi \int_{\Rd} \big[L^{\gamma,1}_t(x,y)+ M^{\gamma,0}_t(x,y)\big] |f(y)|dy \\
      & \quad + \int_{\Rd} L^{\gamma,\gamma}_t(x,y)\big|(\Lambda_0^{\f{1-\gamma}{\alpha}}f)(y)\big|dy \\
      & \quad + \int_{\Rd} M^{\gamma,\gamma}_t(x,y)\big||y|^{1-\gamma}(\Lambda_0^{\f{1-\gamma}{\alpha}}f)(y)\big|dy.
    \end{split}
  \end{align}
\end{proposition}

\begin{proof}  
  By Duhamel's formula,
  \[
    \begin{aligned}
      \tilde {p}_{t}(x,y) - {p}_{t}(x,y)
      & = \kappa \int_0^t\int_{\Rd}{p}_{t-s}(x,z){|z|^{-\alpha}z\cdot\nabla_z}\tilde p_{s}(z,y)\,dz\,ds \\
      & = \kappa \int_0^{t/2}\int_{\Rd}{p}_{t-s}(x,z){|z|^{-\alpha}z\cdot\nabla_z}\tilde p_{s}(z,y)\,dz\,ds \\
      & \quad + \kappa \int_0^{t/2}\int_{\Rd}{p}_{s}(x,z){|z|^{-\alpha}z\cdot\nabla_z}\tilde{p}_{t-s}(z,y)\,dz\,ds.
    \end{aligned}
  \]
  Differentiating both sides with respect to $t$ and multiplying by $t$ gives
  \begin{equation*}
    \begin{aligned}
      {Q}_t(x,y)
      & = \kappa  t\int_{\mathbb{R}^d}{p}_{t/2}(x,z) {|z|^{-\alpha}z\cdot\nabla_z}\tilde p_{t/2}(z,y)\,dz \\
      & \quad + \kappa  t\int_0^{t/2}\int_{\mathbb{R}^d}{p}_{t-s,1}(x,z) {|z|^{-\alpha}z\cdot\nabla_z} \tilde p_{s}(z,y)\,dz\,ds \\
      & \quad + \kappa t\int_{t/2}^t\int_{\mathbb{R}^d}{p}_{t-s}(x,z) {|z|^{-\alpha}z\cdot\nabla_z}\tilde p_{s,1}(z,y)\,dz\,ds,
    \end{aligned}
  \end{equation*}
  which implies
  \begin{equation*}
    \begin{aligned}
      {Q}_tf(x)
      & = \kappa  t\int_{\mathbb{R}^d}\int_{\mathbb{R}^d}{p}_{t/2}(x,z) {|z|^{-\alpha}z\cdot\nabla_z}\tilde p_{t/2}(z,y)f(y)\,dz\,dy \\
      & \quad + \kappa  t\int_{\mathbb{R}^d}\int_0^{t/2}\int_{\mathbb{R}^d}{p}_{t-s,1}(x,z) {|z|^{-\alpha}z\cdot\nabla_z} \tilde p_{s}(z,y)f(y)\,dz\,ds\,dy \\
      & \quad + \kappa t\int_{\mathbb{R}^d}\int_{t/2}^t\int_{\mathbb{R}^d}{p}_{t-s}(x,z) {|z|^{-\alpha}z\cdot\nabla_z}\tilde p_{s,1}(z,y)f(y)\,dz\,ds\,dy.
    \end{aligned}
  \end{equation*}
  Without loss of generality, we now assume $t=1$.
  Set $S_{x,1}=\{z: |x|/16<|z|<4|x|\}$, $S_{x,2}=\Rd\setminus S_{x,1}$, $R_{x}=\{y: |y|\ge 1, |x-y|<(|x|\wedge|y|)/2\}$. Then we can write
  \[
    \begin{aligned}
      (Q_1f)(x) = \int_{\Rd\setminus R_{x}}{Q}_1(x,y)f(y)dy + \int_{R_{x}} {Q}_1(x,y)f(y)dy. 
    \end{aligned}
  \]
  The first term can be done similarly to Case 1 in the proof of Proposition \ref{prop-difference}. We have
  \[
    \left|\int_{\Rd\setminus R_{x}}{Q}_1(x,y)f(y)dy\right|
    \lesi \int_{\Rd }L_1^{\gamma,1}(x,y)|f(y)|dy.
  \]
  
  For the second term, we write
  \begin{equation*}
    \begin{aligned}
      \int_{R_{x}}{Q}_1(x,y)f(y)dy
      & = \kappa \sum_{i=1,2} \int_{R_{x}}\int_{S_{x,i}}{p}_{1/2}(x,z) {|z|^{-\alpha}z\cdot\nabla_z}\tilde p_{1/2}(z,y)f(y)\,dz\,dy \\
      & \quad + \kappa  \sum_{i=1,2}\int_{R_{x}}\int_0^{1/2}\int_{S_{x,i}}{p}_{1-s,1}(x,z) {|z|^{-\alpha}z\cdot\nabla_z} \tilde p_{s}(z,y)f(y)\,dz\,ds\,dy \\
      & \quad + \kappa \sum_{i=1,2}\int_{R_{x}}\int_{1/2}^1\int_{S_{x,i}}{p}_{1-s}(x,z) {|z|^{-\alpha}z\cdot\nabla_z}\tilde p_{s,1}(z,y)f(y)\,dz\,ds\,dy \\
      & =: Q_{1,1}f(x) + Q_{1,2}f(x).
    \end{aligned}
  \end{equation*}
  
  For $Q_{1,2}f$, by the kernel bounds in Proposition~\ref{thm-ptk}, Lemmas~\ref{derivativeheatkernel} and~\ref{thm-spacetimederivative} and Lemma~\ref{lem- difference for alpha < 2},
  (using $H_t^{\gamma+1,\gamma}\leq H_t^{\gamma,\gamma}$)
  \[
    \begin{aligned}
      |Q_{1,2} f|
      \lesi \int_{\Rd} L^{\gamma,1}_1(x,y) |f(y)|dy + \int_{\Rd} M^{\gamma,0}(x,y)|f(y)|dy.
    \end{aligned}
  \]
  To study $Q_{1,1}f$, we use $\nabla \me{-\Lambda_0}=\nabla \Lambda_0^{-\f{1-\gamma}{\alpha}} \me{-\Lambda_0}\Lambda_0^{\f{1-\gamma}{\alpha}}$. Thus,
  \begin{equation}
    \label{eq-KatoNew}
    \begin{aligned}
      Q_{1,1}f
      & = \kappa \int_{\mathbb{R}^d}\int_{S_{x,1}}{p}_{1/2}(x,z) |z|^{-\alpha}z \cdot \nabla_z \Lambda_0^{-\f{1-\gamma}{\alpha}} \me{-\f{1}{2}\Lambda_0}(z,y)\Lambda_0^{\f{1-\gamma}{\alpha}} f(y)\,dz\,dy \\
      & \quad + \kappa \int_{\mathbb{R}^d}\int_0^{1/2}\int_{S_{x,1}}{p}_{1-s,1}(x,z) {|z|^{-\alpha}z\cdot} \nabla_z \Lambda_0^{-\f{1-\gamma}{\alpha}} \me{-s\Lambda_0}(z,y)\Lambda_0^{\f{1-\gamma}{\alpha}} f(y)\,dz\,ds\,dy \\
      & \quad + \kappa \int_{\mathbb{R}^d}\int_{1/2}^1 \int_{S_{x,1}}{p}_{1-s}(x,z) {|z|^{-\alpha}z\cdot}\nabla_z \Lambda_0^{-\f{1-\gamma}{\alpha}+1} \me{-s\Lambda_0}(z,y)\Lambda_0^{\f{1-\gamma}{\alpha}} f(y)\,dz\,ds\, dy \\
      & \quad -\kappa \int_{\mathbb{R}^d\setminus R_{x}}\int_{S_{x,1}}{p}_{1/2}(x,z) |z|^{-\alpha}z \cdot \nabla_z \tilde p_{1/2}(z,y)  f(y)\,dz\,dy \\
      & \quad - \kappa \int_{\mathbb{R}^d\setminus R_{x}}\int_0^{1/2}\int_{S_{x,1}}{p}_{1-s,1}(x,z) {|z|^{-\alpha}z\cdot} \nabla_z \tilde p_{s}(z,y)  f(y)\,dz\,ds\,dy \\
      & \quad - \kappa \int_{\mathbb{R}^d\setminus R_{x}}\int_{1/2}^1 \int_{S_{x,1}}{p}_{1-s}(x,z) {|z|^{-\alpha}z\cdot}\nabla_z \tilde p_{s,1}(z,y) f(y)\,dz\,ds\, dy.
    \end{aligned}
  \end{equation}
  We now set
  \begin{align}
    \label{eq:qt11}
    \begin{split}
      Q_{1,1}^1f(x)
      & = \kappa \int_{\mathbb{R}^d}\int_{S_{x,1}}{p}_{1/2}(x,z) |z|^{-\alpha}z \cdot \nabla_z \Lambda_0^{-\f{1-\gamma}{\alpha}} \me{-\f{1}{2}\Lambda_0}(z,y)\Lambda_0^{\f{1-\gamma}{\alpha}} f(y)\,dz\,dy \\
      & \quad + \kappa \int_{\mathbb{R}^d}\int_0^{1/2}\int_{S_{x,1}}{p}_{1-s,1}(x,z) {|z|^{-\alpha}z\cdot} \nabla_z \Lambda_0^{-\f{1-\gamma}{\alpha}} \me{-s\Lambda_0}(z,y)\Lambda_0^{\f{1-\gamma}{\alpha}} f(y)\,dz\,ds\,dy \\
      & \quad + \kappa \int_{\mathbb{R}^d}\int_{1/2}^1 \int_{S_{x,1}}{p}_{1-s}(x,z) {|z|^{-\alpha}z\cdot}\nabla_z \Lambda_0^{-\f{1-\gamma}{\alpha}+1} \me{-s\Lambda_0}(z,y)\Lambda_0^{\f{1-\gamma}{\alpha}} f(y)\,dz\,ds\, dy
    \end{split}
  \end{align}
  and
  \begin{align}
    \begin{split}
      Q_{1,1}^2f(x)
      & = -\kappa \int_{\mathbb{R}^d\setminus R_{x}}\int_{S_{x,1}}{p}_{1/2}(x,z) |z|^{-\alpha}z \cdot \nabla_z \tilde p_{1/2}(z,y)  f(y)\,dz\,dy \\
      & \quad - \kappa \int_{\mathbb{R}^d\setminus R_{x}}\int_0^{1/2}\int_{S_{x,1}}{p}_{1-s,1}(x,z) {|z|^{-\alpha}z\cdot} \nabla_z \tilde p_{s}(z,y)  f(y)\,dz\,ds\,dy \\
      & \quad - \kappa \int_{\mathbb{R}^d\setminus R_{x}}\int_{1/2}^1 \int_{S_{x,1}}{p}_{1-s}(x,z) {|z|^{-\alpha}z\cdot}\nabla_z \tilde p_{s,1}(z,y) f(y)\,dz\,ds\, dy.
    \end{split}
  \end{align}
  From the kernel bounds in Proposition \ref{thm-ptk} and Lemma \ref{lem-gradient of heat kernel 2}, by using Lemmas~\ref{lem- TH x equiv z} and~\ref{lem- difference for alpha < 2}, we obtain
  \begin{align}
    \label{eq:boundq111}
    \begin{split}
      Q_{1,1}^1f(x)
      & \lesi \int_{\Rd}  L^{\gamma,\gamma}_1(x,y) \big| (\Lambda_0^{\f{1-\gamma}{\alpha}}f)(y)\big|dy + \int_{\Rd} M^{\gamma,\gamma}_1(x,y)\big||y|^{1-\gamma}(\Lambda_0^{\f{1-\gamma}{\alpha}}f)(y)\big|dy.
    \end{split}
  \end{align}
		
  To bound $Q_{1,1}^2f(x)$, we use the kernel bounds in Proposition~\ref{thm-ptk}, Lemmas~\ref{derivativeheatkernel} and~\ref{thm-spacetimederivative}, and Lemma~\ref{lem- TH x equiv z}. We obtain
  \[
    \begin{aligned}
      |Q_{1,1}^2 f(x)| \lesi \int_{\Rd} L^{\gamma,1}_1(x,y) |f(y)|dy.
    \end{aligned}
  \]
  This completes our proof.
\end{proof}

\section{Proof of the reversed Hardy inequality (Theorem~\ref{thm-difference})}
\label{s:reversedhardy}

We now use the previous bounds for the difference of kernels to prove Theorem~\ref{thm-difference}, i.e., the reversed Hardy inequality, expressed in terms of our square functions.

\smallskip
\textbf{Proof of part~(1) in Theorem~\ref{thm-difference}.}
Let $\alpha s<\alpha-1$.
By Proposition \ref{prop-difference},
\begin{align}
  \label{eq:prop-differenceaux0}
  \begin{split}
    & \Big(\int_0^\vc t^{-2s}\left|\left(t\Ln \me{-t\Ln} -t\La \me{-t\La}\right)f(x)\right|^2\f{dt}{t}\Big)^{1/2} \\
    & \quad = \Big[\sum_{j\in \mathbb{Z}}\int_{2^{\alpha j}}^{2^{\alpha(j+1)}} t^{-2s}\left|\left(t\Ln \me{-t\Ln} - t \La \me{-t\La}\right)f(x)\right|^2\f{dt}{t}\Big]^{1/2} \\
    & \quad \lesssim \left[\sum_{j\in \mathbb{Z}}\int_{2^{\alpha j}}^{2^{\alpha(j+1)}} t^{-2s}\left(\int_{\mathbb{R}^d} [L^{\gamma,1}_t(x,y)+M^{\gamma,1}_t(x,y)]|f(y)|dy\right)^2\f{dt}{t}\right]^{1/2} \\
    & \quad \lesssim \sum_{j\in \mathbb{Z}} 2^{-js \alpha} \int_{\mathbb{R}^d} \left[L^{\gamma,1}_{2^{\alpha j}}(x,y) + M^{\gamma,1}_{2^{j\alpha}}(x,y)\right] |y|^{\alpha s} \,\frac{|f(y)|}{|y|^{\alpha s}} dy
  \end{split}
\end{align}
where in the last inequality we used the embedding $\ell_1\hookrightarrow\ell_2$. Thus, it suffices to show the $L^p(\R^d)$-boundedness of the operator with kernel
\begin{align}
  \label{eq:thm-differenceaux1}
  \sum_{j\in\Z}2^{-j\alpha s}\left[L_{2^{\alpha j}}^{\gamma,1}(x,y) + M_{2^{j\alpha}}^{\gamma,1}(x,y)\right] |y|^{\alpha s}.
\end{align}
To that end, we use Schur tests similar to those in \cite{Merz2021}.
For the sake of completeness, we give the details.
We begin with the Schur test involving $M_{2^{j\alpha}}^{\gamma,1}$. In particular, we will see that these Schur tests require $\alpha s < \alpha-1$. In the following, let $N:=2^{-j}\in 2^\Z$. Using $\one_{x\in\R^d}\one_{|y|>t^{1/\alpha}} \leq \one_{|x|\vee|y|>t^{1/\alpha}}$ and noting that on the support of $M_{N^{-\alpha}}^{\gamma,1}(x,y)$, we have $|y|/2\leq|x|\leq2|y|$, we may replace the kernel with a symmetric kernel, i.e., it suffices to carry out a single Schur test. We estimate
\begin{align}
  \label{eq:schurtest2m0}
  \begin{split}
  & \sup_{y\in\R^d} \int_{\R^d} dx\, \sum_{N\in2^\Z} N^{\alpha s}\, M_{N^{-\alpha}}^{\gamma,1}(x,y)  (|x||y|)^{\frac{\alpha s}{2}} \\
  & \quad \sim \sup_{y\in\R^d} \int\limits_{\frac12 |y|\leq |x|\leq 2|y|} dx\, \sum_{N\geq(|x|\vee|y|)^{-1}} N^{\alpha s}\, \frac{N^{1-\alpha+d}}{(|x|\vee |y|)^{\alpha-1}} \left( 1 \wedge \frac{N^{-\gamma-d}}{|x-y|^{d+\gamma}} \right)  (|x||y|)^\frac{\alpha s}{2} \\
  & \quad \lesssim \sup_{y\in\R^d} |y|^{\alpha s-\alpha+1} \int\limits_{\frac12 |y|\leq |x|\leq 2|y|} \,dx \sum_{N\geq(2|y|)^{-1}} N^{\alpha s +1-\alpha +d} \, \left( 1 \wedge \frac{N^{-\gamma-d}}{|x-y|^{d+\gamma}} \right).
  \end{split}
\end{align}
Interchanging the order of integration and summation shows that the right-hand side is bounded by
\begin{align}
  \label{eq:schurtest2m}
  \begin{split}
    & \sup_{y\in\R^d} |y|^{\alpha s+1-\alpha} \sum_{N\geq(2|y|)^{-1}} N^{\alpha s +1-\alpha +d} \int\limits_{\frac12 |y|\leq |x|\leq 2|y|} dx\, \left(1\wedge\frac{N^{-\gamma-d}}{|x-y|^{d+\gamma}}\right) \\
    & \quad \leq \sup_{y\in\R^d} |y|^{\alpha s+1-\alpha} \sum_{N\geq(2|y|)^{-1}} N^{\alpha s+1 -\alpha +d} \int_{\R^d} dx\, \left(1\wedge\frac{N^{-\gamma-d}}{|x-y|^{d+\gamma}}\right) \\
    & \quad \sim \sup_{y\in\R^d} |y|^{\alpha s+1-\alpha} \sum_{N\geq(2|y|)^{-1}} N^{\alpha s+1 -\alpha} \sim 1
  \end{split}
\end{align}
where we used $\gamma>0$ and $\alpha s<\alpha-1$. This concludes the Schur test involving the kernel $M_{2^{j\alpha}}^{\gamma,1}$.

\smallskip
It remains to carry out the Schur tests involving $L_{2^{\alpha j}}^{\gamma,1}$. The $L^p$-boundedness of the second summand of $L^{\gamma,1}_t$ follows from
\[
  \begin{aligned}
    \int_0^\vc\f{dt}{t} \int_{\Rd}\textbf{1}_{\{|y|\le t^{1/\alpha}, |x|\sim |y|\}} \frac{|y|^{\alpha s}}{t^{s}} \Big(\f{|y|}{t^{1/\alpha}}\Big)^{1+\beta-d-\alpha}\,t^{-\frac{d}{\alpha}}|g(y)|\,dy
    \lesi (\mathcal{M}_1g)(x)
  \end{aligned}
\]
and the $L^p$-boundedness of the Hardy--Littlewood maximal operator.
To treat the other two summands, let $N:=2^{-j}\in 2^\Z$ as before.
The tests for the regions $|x|\vee|y|<t^{1/\alpha}$ and $|x|\wedge|y|>t^{1/\alpha}$ are similar to those in \cite{Merz2021}; however, unlike in \cite{Merz2021}, where $p\in(d/\beta,d/(d-\beta))$ was required, we only need $p>d/\beta$ here because of the absent singular weight $(|x|/t^{1/\alpha})^{\beta-d}$ in the region $|x|\vee|y|<t^{1/\alpha}$, which is due to the non-symmetry of $\La$. Let us now give the details. First, we bound
\begin{align}
  \begin{split}
    & |y|^{\beta-d-(\alpha-\gamma)} \sum_{N\leq |y|^{-1}} N^{\alpha s + \beta-(\alpha-\gamma)} \left(\frac{N^{-1}}{N^{-1}+|x-y|}\right)^{d+\gamma} \\
    & \quad + \sum_{N\geq |y|^{-1}} N^{\alpha s+d} \, \left(\frac{N^{-1}}{N^{-1}+|x-y|}\right)^{d+\gamma} \one_{|x-y|>(|x|\wedge|y|)/2} \\
    & \lesssim |y|^{\beta-d} |x-y|^{-\alpha s-\beta} + |x-y|^{-d-\alpha s} \one_{|x-y|>\frac{|x|\wedge|y|}{2}},
  \end{split}
\end{align}
where the summability relied on $\alpha s + \beta>0$ (which follows from $\beta\geq(d+\alpha)/2$) and $\alpha s-\gamma<0$ (which follow from $s<1$ and that $\gamma<\alpha$ may be chosen arbitrarily close to $\alpha$). Thus,
\begin{align}
  \label{eq:schurlfinal}
  \begin{split}
    & \left\|\int_{\R^d} dy\ \sum_{N\in 2^\Z} N^{\alpha s}L_{N^{-\alpha}}^{\gamma,1}(x,y) |y|^{\alpha s} g(y) \right\|_{L^p(\Rd)} \\
    & \quad \lesssim \|g\|_{L^p(\Rd)} + \left\|\int_{\R^d} dy\,  |y|^{\alpha s+\beta-d} |x-y|^{-\alpha s-\beta}g(y)\right\|_{L^p(\Rd)} \\ 
    & \qquad + \left\| \int_{\R^d}dy\, \frac{\one_{|x-y|>(|x|\wedge|y|)/2}}{|x-y|^{d+\alpha s}}|y|^{\alpha s} g(y)\right\|_{L^p(\Rd)}.
  \end{split}
\end{align}
Now, we estimate the last two summands by multiples of $\|g\|_p$ using Schur tests with weights being powers of $|x|/|y|$.
By
\begin{align*}
  \int\limits_{|x-y|>(|x|\wedge|y|)/2}\frac{dy}{|x-y|^{d+\alpha s}}|y|^{\alpha s} \left(\frac{|x|}{|y|}\right)^{\delta/p} + \int\limits_{|x-y|>(|x|\wedge|y|)/2}\frac{dx}{|x-y|^{d+\alpha s}}|y|^{\alpha s} \left(\frac{|y|}{|x|}\right)^{\delta/p'} < \infty
\end{align*}
with $0<\delta/p<d+\alpha s$ and $-\alpha s<\delta/p'<d$ (such $\delta$ exist for all $1<p<\infty$), the second summand on the right-hand side of \eqref{eq:schurlfinal} is bounded by $\|g\|_{L^p(\Rd)}$.
We now consider the first summand on the right-hand side of \eqref{eq:schurlfinal}. We have
\begin{align*}
  \int_{\R^d} |y|^{\alpha s+\beta-d} |x-y|^{-\alpha s-\beta} \left(\frac{|x|}{|y|}\right)^{\delta/p}\,dy + \int_{\R^d} |y|^{\alpha s+\beta-d} |x-y|^{-\alpha s-\beta} \left(\frac{|y|}{|x|}\right)^{\delta/p'}\,dx < \infty
\end{align*}
if $0<\delta/p<\alpha s + \beta$ and $d-\alpha s-\beta<\delta/p'<d$, or, equivalently, $p>\delta/(\alpha s+\beta)$. Since $p>d/\beta$ and $d/\beta>\delta/(\alpha s+\beta)$ (since $d(\alpha s+\beta)>\alpha d>\delta\beta$ because $d>\beta>\alpha$ and $\delta<\alpha$), there are $\delta$ such that the condition $\delta/p<\alpha s+\beta$ is fulfilled. Thus, the first summand on the right-hand side of \eqref{eq:schurlfinal} is bounded by $\|g\|_{L^p(\Rd)}$, too.
This concludes the proof of part~(1) in Theorem~\ref{thm-difference}.

\medskip
\textbf{Proof of part~(2) in Theorem~\ref{thm-difference}.}
We argue as before, but use Proposition~\ref{prop-difference v2} instead of Proposition~\ref{prop-difference}. More precisely, fix $s\in (0,1)$ and take $\gamma\in (0,1)$ such that $1-\gamma \leq  \alpha s < \alpha -\gamma$.
Then, by Proposition~\ref{prop-difference v2},
\begin{align}
  \begin{split}
    & \Big(\int_0^\vc t^{-2s}\left|\left(t\Ln \me{-t\Ln} -t\La \me{-t\La}\right)f(x)\right|^2\f{dt}{t}\Big)^{1/2} \\
    & \quad = \Big[\sum_{j\in \mathbb{Z}}\int_{2^{\alpha j}}^{2^{\alpha(j+1)}} t^{-2s}\left|\left(t\Ln \me{-t\Ln} - t \La \me{-t\La}\right)f(x)\right|^2\f{dt}{t}\Big]^{1/2} \\
    & \quad \le \sum_{j\in \mathbb{Z}} 2^{-js \alpha} \int_{\mathbb{R}^d} \left[L^{\gamma,\gamma}_{2^{\alpha(j+1)}}(x,y)|y|^{\gamma-1} + M^{\gamma,\gamma}_{2^{j\alpha}}(x,y)\right] |y|^{\alpha s} \,\frac{|\Lambda_0^{\f{1-\gamma}{\alpha}}f(y)|}{|y|^{\alpha s+\gamma-1}} dy \\
   & \qquad + \sum_{j\in \mathbb{Z}} 2^{-js \alpha} \int_{\mathbb{R}^d} \left[ L^{\gamma,1}_{2^{\alpha(j+1)}}(x,y) + M^{\gamma,0}_{2^{j\alpha}}(x,y)\right] |y|^{\alpha s} \,\frac{|f(y)|}{|y|^{\alpha s}}\, dy.
  \end{split}
\end{align}
Since $s\alpha< \alpha -\gamma $, we have, using similar Schur tests as above,
\begin{align}
  \label{eq:reversehardyauxfinal}
  \begin{split}
    & \Big\|\Big(\int_0^\vc t^{-2s}\left|\left(t\Ln \me{-t\Ln} -t\La \me{-t\La}\right)f(x)\right|^2\f{dt}{t}\Big)^{1/2}\Big\|_{L^p(\Rd)} \\
    & \quad \lesi \Big\|\frac{|\Lambda_0^{\f{1-\gamma}{\alpha}}f(x)|}{|x|^{\alpha s+\gamma-1}}\Big\|_{L^p(\Rd)} + \||x|^{-\alpha s}f\|_{L^p(\R^d)}.
  \end{split}
\end{align}
This concludes also the proof of part~(2) in Theorem~\ref{thm-difference}.
\qed

\begin{remark}
  \label{rem:schurm}
  Note that the power $d+\gamma$ in the Schur test for $M_{2^{j\alpha}}^{\gamma,1}$ (see the factor $\left(1\wedge\frac{N^{-\gamma-d}}{|x-y|^{d+\gamma}}\right)$ in the second to last line of \eqref{eq:schurtest2m}) does not affect the range of admissible $s$, as long as $\gamma>0$.
\end{remark}

\section{Proof of the generalized Hardy inequality (Theorem~\ref{thm-HardyIneq})}
\label{s:generalizedhardy}

In this section, we prove the generalized Hardy inequality (Theorem~\ref{thm-HardyIneq}). To that end, we prove the following Riesz kernel bounds, which we obtain by integrating the heat kernel bounds \eqref{eq:heatkernel} against monomials in time.

\begin{lemma}
  \label{riesz}
  Let $\alpha\in(1,2\wedge (d+2)/2)$, $\beta\in ((d+\alpha)/2,d+\alpha)$, 
  $s\in(0,1]$,
  and $\kappa=\Psi(\beta)$ be defined by \eqref{eq:defbeta}. Then,
  \begin{align}
    \label{eq:riesz}
    \Lambda_\kappa^{-s}(x,y) \sim |x-y|^{\alpha s-d} \left(1\wedge\frac{|y|}{|x-y|}\right)^{\beta-d}.
  \end{align}
\end{lemma}

\begin{proof}
  By the functional calculus, we have
  \begin{align*}
    \Lambda_\kappa^{-s}(x,y) = \frac{1}{\Gamma(s)} \int_0^\infty \frac{dt}{t} t^s \me{-t\Lambda_\kappa}(x,y).
  \end{align*}
  By scaling and the heat kernel bounds \eqref{eq:heatkernel},
  \begin{align*}
    \Lambda_\kappa^{-s}(x,y) \sim |x-y|^{\alpha s-d} \int_0^\infty \frac{dt}{t} t^s \frac{t}{(t^{1/\alpha}+1)^{d+\alpha}} \cdot \left(1\wedge\frac{|y|/|x-y|}{t^{1/\alpha}}\right)^{\beta-d}.
  \end{align*}
  In the following, we distinguish between $|x-y|>|y|/2$ and $|x-y|<|y|/2$.

  \textbf{Case $|x-y|>|y|/2$.} We get
  \begin{align*}
    \Lambda_\kappa^{-s}(x,y)
    & \sim |x-y|^{\alpha s-d} \left[\int_0^{(|y|/|x-y|)^{\alpha}}\frac{dt}{t}\, t^{s+1} + \int_{(|y|/|x-y|)^{\alpha}}^1 \frac{dt}{t} t^{s+1} \cdot \left(\frac{|y|/|x-y|}{t^{1/\alpha}}\right)^{\beta-d} \right. \\ & \quad \left. + \int_1^\infty \frac{dt}{t}\, t^{s+1-1-d/\alpha} \left(\frac{|y|/|x-y|}{t^{1/\alpha}}\right)^{\beta-d} \right] \\
    & \sim |x-y|^{\alpha s-d} \left[ \left(\frac{|y|}{|x-y|}\right)^{(s+1)\alpha} + \left(\frac{|y|}{|x-y|}\right)^{\beta-d} \right] \\
    & \sim |x-y|^{\alpha s-d} \cdot \left(\frac{|y|}{|x-y|}\right)^{\beta-d}.
  \end{align*}

  \textbf{Case $|x-y|<|y|/2$.} We get
  \begin{align*}
    \Lambda_\kappa^{-s}(x,y)
    & \sim |x-y|^{\alpha s-d} \left[\int_0^1 \frac{dt}{t} t^{s+1} + \int_1^{(|y|/|x-y|)^{\alpha}} \frac{dt}{t} t^{s+1-1-d/\alpha} \right. \\ & \quad \left. + \int_{(|y|/|x-y|)^{\alpha}}^\infty \frac{dt}{t}\, t^{s+1-1-d/\alpha} \left(\frac{|y|/|x-y|}{t^{1/\alpha}}\right)^{\beta-d} \right] \\
    & \sim |x-y|^{\alpha s-d}\left[1+\left(\frac{|y|}{|x-y|}\right)^{\alpha s-d}\right]
      \sim |x-y|^{\alpha s-d}.
  \end{align*}
  Combining the above two estimates concludes the proof.
\end{proof}

We are now ready to prove the generalized Hardy inequality.

\begin{proof}[Proof of Theorem~\ref{thm-HardyIneq}]
  It suffices to show the $L^p$-boundedness of the operator with integral kernel $|x|^{-\alpha s}\Lambda_\kappa^{-s}(x,y)$. To that end, we use weighted Schur tests and distinguish between the following regions.

  \textbf{Case $|x-y|<4(|x|\wedge|y|)$.}
  In that case, $|x|\sim|y|$ and the integral kernel in question is bounded from above and below by constants times $|x|^{-\alpha s}|x-y|^{\alpha s-d}$. By a Schur test,
  \begin{align*}
    \int_{|x-y|<4|x|}|x|^{-\alpha s}|x-y|^{\alpha s-d}\,dy + \int_{|x-y|<4|y|}|x|^{-\alpha s}|x-y|^{\alpha s-d}\,dx
    \lesssim 1.
  \end{align*}

  \smallskip
  \textbf{Case $4|x|<|x-y|<4|y|$.}
  In that case, $|x-y|\sim|y|\geq |x|$ and the integral kernel in question is bounded from above and below by constants times $|x|^{-\alpha s}|y|^{\alpha s-d}$. A Schur test with weight being a power of $|x|/|y|$ gives
  \begin{align*}
    \int_{|x|<|y|}|x|^{-\alpha s}|y|^{\alpha s-d}(|y|/|x|)^{\delta/p'}\,dx + \int_{|x|<|y|} |x|^{-\alpha s}|y|^{\alpha s-d}(|x|/|y|)^{\delta/p}\,dy \lesssim1,
  \end{align*}
  whenever $p\alpha s<\delta<p'(d-\alpha s)$. Such $\delta$ exist since $\alpha s-\frac{\alpha s}{p}<\frac{d-\alpha s}{p}$ is true under the assumption $p<d/(\alpha s)$. 

  \smallskip
  \textbf{Case $4|y|<|x-y|<4|x|$.}
  In that case, $|x-y|\sim|x|>|y|$ and the integral kernel in question is bounded from above and below by constants times
  \begin{align*}
    |x|^{-\alpha s}|y|^{\beta-d}|x-y|^{\alpha s-\beta} \sim |x|^{-\beta}|y|^{\beta-d}.
  \end{align*}
  This integrand is similar to that in the previous case but with $\alpha s$ replaced with $\beta$ and with the regions of integration interchanged. Hence, arguing similarly as before yields~\eqref{eq:thm-HardyIneq} under the assumption $p>d/\beta$.

  Finally, \eqref{eq:newgenhardynew2} follows from \eqref{eq:thm-HardyIneq} since
  \begin{align*}
    \||y|^{\alpha(\frac{1-\gamma}{\alpha}-s)} |\La^{\frac{1-\gamma}{\alpha}} f| \|_{L^p(\Rd)} 
    = \||y|^{\alpha(\frac{1-\gamma}{\alpha}-s)} \La^{\frac{1-\gamma}{\alpha}} f \|_{L^p(\Rd)}
    \lesssim \|\La^{s-\frac{1-\gamma}{\alpha}}\La^{\frac{1-\gamma}{\alpha}}f\|_{{L^p(\Rd)}}
    = \|\La^s f\|_{{L^p(\Rd)}}.
  \end{align*}
  This concludes the proof.
\end{proof}

\section{Applications of Theorem~\ref{eqsob}}
\label{s:applications}

In this section, we provide two concrete applications of the main result, Theorem~\ref{eqsob}, in the contexts of nonlinear PDE and perturbation theory.

\subsection{Application to the nonlinear heat equation associated to $\Lambda_\kappa$}

As a first concrete illustration of the usefulness of Theorem~\ref{eqsob}, we consider the Cauchy problem
\begin{equation}\label{eq:6.4}
  \begin{cases}
    &\partial_t u + \Lambda_{\kappa} u = F(u), \ \ (t,x)\in \mathbb{R}^{d+1}_+=[0,\vc)\times \Rd,\\ 
    &\qquad u(0,\cdot) = u_0,
  \end{cases}
\end{equation}
where $F: \mathbb R\to \mathbb R$ is smooth and satisfies $F(0)=F'(0)=0$ and 
\begin{equation}\label{eq-F}
  |F'(x)| + |xF''(x)|\lesi |x|^{\beta -1}
\end{equation}
for all $x{\in\R}$ and some $\beta>2$. Typical examples for $F(u)$ include the cases $F(u)=|u|^\beta$ or $F(u)=u|u|^{\beta-1}$ with $\beta>2$. 

In what follows, for $s>0$, $p\in (1,\vc)$, and functions $f=f(x)$ and $u=u(t,x)$, we use the notations
\[
  \|f\|_{\dot{W}^p_s(\Rd)}:=\|(-\Delta)^{s/2}f\|_{L^p(\Rd)}
\]
and
$$
\|u\|_{C\big(I; \dot{W}^p_s(\mathbb R^d)\big)} := \sup_{t\in I}\|u\|_{ \dot{W}^p_s(\R^d) }.
$$

The following chain rule is standard; see, for example, \cite[Chapter~2, Proposition~5.1]{Taylor2000}.

\begin{lemma}\label{lem-chain rule}
  Assume that $F\in C^1$ with
  \[
    F(0)=0, \quad |F'(x)|\le C|x|^{\sigma-1}, \ \sigma>1.
  \]
  Let $s\in[0,1]$ and $1<p<\vc$. Then,
  \[
    \|F(u)\|_{{\dot{W}_s^p(\Rd)}}
    \le C\|u\|^{\sigma-1}_{L^\vc(\Rd)}\|u\|_{{\dot{W}_s^p(\Rd)}}.
  \]
\end{lemma}
 
In the following, we use the equivalence Sobolev norms in Theorem~\ref{eqsob}, together with estimates for $\me{-t\Lambda_\kappa}$, to show existence, uniqueness, and a priori estimates for solutions of \eqref{eq:6.4}. Thus, Theorem~\ref{eqsob} is not only of theoretical interest: it provides a useful tool for solving evolution equations involving $\Lambda_\kappa$ in widely used function spaces of practical relevance.

\begin{theorem}
  Let $\beta>1$ and $0<s<\alpha-1$ and $\displaystyle (d_\beta)' < p < \frac{d}{s} \wedge d_\beta$, where $d_\beta$ is as in \eqref{eq-d beta}. Then for any $\tau > 0$, there exists $\epsilon_0>0$ such that if 
  $\|u_{0}\|_{_{\dot{W}^p_s(\mathbb R^d)\cap L^\vc(\Rd)}}  \le \epsilon_0$, there exists a unique solution 
  $u \in C\big(I; \dot{W}^{p}_{s}(\Rd)\big)\cap L^\infty\big(I; L^\infty(\Rd)\big)$,  where $I=[0,\tau)$, of the Cauchy problem~\eqref{eq:6.4} such that 
  \[
    \|u\|_{C\big(I; \dot{W}^{p}_{s}(\Rd)\big)\cap L^\infty\big(I; L^\infty(\Rd)\big)}
    \lesssim  \|u_0\|_{\dot{W}^p_s(\mathbb R^d)\cap L^\vc(\Rd)}.
  \]
 
\end{theorem}

\begin{proof}
  We fix $s$ and $p$ as above. By \eqref{eq:heatkernelalpha2} and Theorem \ref{thm 2-Tt}, 
  \begin{equation}\label{eq2-parabolic}
    \|\me{-t\Lambda_\kappa}\|_{L^p(\Rd)\to L^p(\R^d)}+\|\me{-t\Lambda_\kappa}\|_{L^\vc(\Rd)\to L^\vc(\R^d)}\lesi 1
  \end{equation}
  holds uniformly for all $t>0$.

  Let $\tau>0$, to be fixed later. By Duhamel's formula, a solution $u$ to \eqref{eq:6.4} satisfies
  \[
    u(t,x) = \me{-t\Lambda_{\kappa}}u_{0}(x) + \int_{0}^{t} \me{-(t-\sigma)\Lambda_{\kappa}} F(u(\sigma,x))\,d\sigma.
  \]
  By Theorem~\ref{eqsob} and~\eqref{eq2-parabolic},
  \begin{equation}\label{eq-uo}
    \begin{aligned}
      \|\me{-t\Lambda_{\kappa}}u_{0}\|_{\dot{W}^p_s(\mathbb R^d)}
      & \sim \|\me{-t\Lambda_{\kappa}}\Lambda_\kappa^{s/\alpha} u_{0}\|_{L^p(\mathbb R^d)} \\
      & \lesi \| \Lambda_\kappa^{s/\alpha} u_{0}\|_{L^p(\mathbb R^d)}\\
      & \lesi \|u_{0}\|_{\dot{W}^p_s(\mathbb R^d)}.
    \end{aligned}
  \end{equation}
  Moreover, whenever $\tau>t_1>t_2\ge 0$, by Theorem~\ref{eqsob} and~\eqref{eq2-parabolic},
  \[
     \begin{aligned}
        \|\me{-t_1\Lambda_{\kappa}}u_{0}-\me{-t_2\Lambda_{\kappa}}u_{0}\|_{ \dot{W}^p_s(\mathbb R^d)} 
	& \sim \|\me{-t_1\Lambda_{\kappa}}\Lambda_\kappa^{s/\alpha} u_{0}-\me{-t_2\Lambda_{\kappa}}\Lambda_\kappa^{s/\alpha}u_{0}\|_{ L^p(\mathbb R^d)} \\
	& \lesi \|\me{-(t_1-t_2)\Lambda_\kappa}\Lambda_\kappa^{s/\alpha}u_{0}\|_{ L^p(\mathbb R^d)}\\
	& \lesi \|\me{-(t_1-t_2)\Lambda_\kappa}\|_{L^p(\mathbb R^d)\to L^p(\mathbb R^d)} \|\Lambda_\kappa^{s/\alpha}u_{0}\|_{ L^p(\mathbb R^d)} \\
	& \sim \|\me{-(t_1-t_2)\Lambda_\kappa}\|_{L^p(\mathbb R^d)\to L^p(\mathbb R^d)}\|u_0\|_{\dot{W}^p_s(\mathbb R^d)}.
    \end{aligned}
  \]
  By an argument similar to that in the proof of estimate~(i) on \cite[p.~2458]{BuiDuongYan2012},
  \[
    \|\me{-(t_1-t_2)\Lambda_\kappa}\|_{L^p(\mathbb R^d)\to L^p(\mathbb R^d)} \to 0 \ \ \text{as} \ \ t_2\to t_1.
  \]
  Therefore, $\me{-t\Lambda_{\kappa}}u_{0}\in C\big(I; \dot{W}^p_s(\mathbb R^d)\big)$. Moreover, by \eqref{eq-uo},
  \[
    \|\me{-t\Lambda_{\kappa}}u_{0}\|_{C\big(I; \dot{W}^p_s(\mathbb R^d)\big)}
    \lesi \|u_{0}\|_{\dot{W}^p_s(\mathbb R^d)}.
  \]
  Similarly,
  \[
    \|\me{-t\Lambda_{\kappa}}u_{0}\|_{L^\vc\big(I; L^\vc(\mathbb R^d)\big)} \lesi \|u_{0}\|_{L^\vc(\mathbb R^d)}.
  \]
  Consequently,
  \begin{equation}\label{eq-u0 estimate}
    \|\me{-t\Lambda_{\kappa}}u_{0}\|_{C\big(I; \dot{W}^{p}_{s})\big)\cap L^\infty\big(I; L^\infty(\Rd)\big)}
    \lesssim \|u_0\|_{\dot{W}^p_s(\mathbb R^d)\cap L^\vc(\Rd)}.
  \end{equation}

  To prove the existence of a solution to \eqref{eq:6.4}, we consider the linear operator
  \[
    u\mapsto Sf(t,x): = \int_{0}^{t} \me{-(t-\sigma)\Lambda_{\kappa}} f(\sigma,x) \, d\sigma.
  \]
  We first show that $S$ is bounded on $C\big(I; {\dot{W}^{p}_{s}(\Rd)}\big)\cap L^\infty\big(I; L^\infty(\Rd)\big)$.
  Indeed, by Theorem~\ref{eqsob}, \eqref{eq2-parabolic}, and Lemma ~\ref{lem-chain rule}, we have for all $t\in I$,
  \begin{equation}\label{eq-parabolic eq}
    \begin{aligned}
      \|Sf(t,\cdot)\|_{{\dot W^p_s(\mathbb R^d)}}
      & \sim \| \Lambda_{\kappa}^{s/\alpha}Sf(t,\cdot)\|_{L^{p}(\Rd)} \\
      & \lesi \int_{0}^{t} \|\Lambda_{\kappa}^{s/\alpha} \me{-(t-\sigma)\Lambda_{\kappa}} f(\sigma,\cdot) \|_{L^{p}(\Rd)}\,d\sigma\\
      & \lesi t \, \sup_{\sigma\in I} \|\Lambda_{\kappa}^{s/\alpha}f(\sigma,\cdot)\|_{L^{p}(\Rd)} \\
      &\lesi |I| \sup_{\sigma\in I}\|f(\sigma,\cdot)\|_{\dot{W}^p_s(\Rd)}
    \end{aligned}
  \end{equation}
  which implies that $Sf(t,\cdot)\in \dot{W}^p_s(\Rd)$ for all $t\in I$.

  Next, for $\tau>t_1>t_2\ge 0$, we estimate
  \[
    \begin{aligned}
      \|Sf(t_1,\cdot)-Sf(t_2,\cdot)\|_{\dot{W}^p_s(\mathbb R^d)}
      &\le \int_{0}^{t_2} \|\Lambda_{\kappa}^{s/\alpha} \big[\me{-(t_1-t_2)\Lambda_\kappa}-I\big]\me{-(t_2-\sigma)\Lambda_{\kappa}} f(\sigma,\cdot)\|_{L^{p}(\Rd)}\,d\sigma\\
      & \quad + \int^{t_1}_{t_2} \|\Lambda_{\kappa}^{s/\alpha} \me{-(t_1-\sigma)\Lambda_{\kappa}} f(\sigma,\cdot)\|_{L^{p}(\Rd)}\,d\sigma.
    \end{aligned}
  \]
  Arguing as in \eqref{eq-parabolic eq}, we obtain
  \[
    \begin{aligned}
      \int_{0}^{t_2} \|\Lambda_{\kappa}^{s/\alpha} &\big[\me{-(t_1-t_2)\Lambda_\kappa}-1\big]\me{-(t_2-\sigma)\Lambda_{\kappa}} f(\sigma,\cdot)\|_{L^{p}(\Rd)}\,d\sigma \\
      &\lesi \|\me{-(t_1-t_2)\Lambda_\kappa}-1\|_{L^p(\Rd)}|I|\sup_{\sigma\in I} \|f(\sigma,\cdot)\|_{\dot{W}^p_s(\Rd)}
    \end{aligned}
  \]
  and
  \[
    \int^{t_1}_{t_2} \|\Lambda_{\kappa}^{s/\alpha} \me{-(t_1-\sigma)\Lambda_{\kappa}} f(\sigma,\cdot)\|_{L^{p}(\Rd)}\,d\sigma\lesi (t_1-t_2)|I|\sup_{\sigma\in I} \|f(s,\cdot)\|_{\dot{W}^p_s(\Rd)}.
  \]
  Hence,
  \[
    \|Sf(t_1,\cdot)-Sf(t_2,\cdot)\|_{\dot{W}^p_s(\mathbb R^d)}\to 0 \ \text{as} \ \ t_2\to t_1.
  \]
  Consequently, $Sf \in C\big(I; {\dot{W}^{p}_{s}}(\Rd)\big)$ and
  \[
    \|Sf\|_{C\big(I; {\dot{W}^{p}_{s}}(\Rd)\big)}
    \lesi |I|\|f\|_{C\big(I;  {\dot{W}^{p}_{s}}(\Rd)\big)}.
  \]
  A similar argument also implies that $Su\in L^\infty\big(I; L^\infty(\Rd)\big)$; moreover, from \eqref{eq-parabolic eq},
  \[
     \|Sf\|_{L^\infty\big(I; L^\infty(\Rd)\big)}\lesi|I|\|f\|_{L^\infty\big(I; L^\infty(\Rd)\big)}.
  \]
  Therefore, there exists $A_0>0$ such that
  \begin{equation}\label{eq-A0}
    \|Sf\|_{C\big(I; {\dot{W}^{p}_{s}}(\Rd)\big)\cap L^\infty\big(I; L^\infty(\Rd)\big)}
    \le A_0 \|f\|_{C\big(I; {\dot{W}^{p}_{s}}(\Rd)\big)\cap L^\infty\big(I; L^\infty(\Rd)\big)}.
  \end{equation}

  Let $u,v\in C\big(I; \dot{W}^{p}_{{s}}(\Rd)\big)\cap L^\infty\big(I; L^\infty(\Rd)\big)$ satisfy
  \begin{equation}\label{eq-uv}
    \|u\|_{C\big(I; \dot{W}^{p}_{s}(\Rd)\big)\cap L^\infty\big(I; L^\infty(\Rd)\big)}\le \epsilon, \quad  \|v\|_{C\big(I; \dot{W}^{p}_{s}(\Rd)\big)\cap L^\infty\big(I; L^\infty(\Rd)\big)}\le \epsilon.
  \end{equation}
  By \eqref{eq-F},
  \[
    |F(x)-F(y)|\le C|x-y|(|x|^{\beta-1}+|y|^{\beta-1}).
  \]
  Using this, we obtain from \eqref{eq2-parabolic}, for all $t\in I$,
  \[
  \begin{aligned}
    \|F(u(t,\cdot))-F(v(t,\cdot))\|_{L^\vc(\mathbb R^d)}
    & \lesi \|u(t,\cdot)-v(t,\cdot)\|_{L^\vc(\mathbb R^d)}\big[\|u(t,\cdot)\|^{\beta-1}_{L^\vc(\Rd)}+\|v(t,\cdot)\|^{\beta-1}_{L^\vc(\Rd)}\big] \\
    &\lesi \epsilon^{\beta-1}\|u(t,\cdot)-v(t,\cdot)\|_{L^\vc(\mathbb R^d)}, 
    \end{aligned}
  \]
  which implies
  \[
    \|F(u)-F(v)\|_{L^\infty\big(I; L^\infty(\Rd)\big)}
    \lesi \epsilon^{\beta-1} \|u-v\|_{C\big(I; \dot{W}^{p}_{s}(\Rd)\big)\cap L^\infty\big(I; L^\infty(\Rd)\big)}.
  \]
  Next, we use
  \[
  F(u)-F(v) = (u-v)\int_0^1 F'(u+\eta(v-u))d\eta,
  \]
  apply Lemma~\ref{lem-chain rule} and \eqref{eq-uv}, and obtain, for $t\in (0,\tau)$,
  \[
  \begin{aligned}
    \|F(u(t,\cdot))&-F(v(t,\cdot))\|_{ \dot{W}^{p}_{s}(\Rd) }\\
    & \lesi \|u(t,\cdot)-v(t,\cdot)\|_{L^{\vc}(\Rd)}	\int_0^1 \|F'(u(t,\cdot)+\eta(v(t,\cdot)-u(t,\cdot)))\|_{\dot{W}^{p}_{s}(\Rd)}d\eta \\
    & \quad + \|u(t,\cdot)-v(t,\cdot)\|_{\dot{W}^{p}_{s}(\Rd)}	\int_0^1 \|F'(u(t,\cdot)+\eta(v(t,\cdot)-u(t,\cdot)))\|_{L^{\vc}(\Rd)}d\eta\\
    & \lesi \epsilon^{\beta-1}\big[\|u(t,\cdot)-v(t,\cdot)\|_{L^{\vc}(\Rd)}+\|u(t,\cdot)-v(t,\cdot)\|_{\dot{W}^{p}_{s}(\Rd)}\big]\\
    & \lesi \epsilon^{\beta-1}\|u-v\|_{C\big(I; \dot{W}^{p}_{s}(\Rd)\big)\cap L^\infty\big(I; L^\infty(\Rd)\big)}.
    \end{aligned}
  \]
  Hence, we can choose $\epsilon$ sufficiently small so that
  \[
    \|F(u)-F(v)\|_{C\big(I; \dot{W}^{p}_{s}(\Rd)\big)\cap L^\infty\big(I; L^\infty(\Rd)\big)}\le \f{1}{2A_0}\|u-v\|_{C\big(I; \dot{W}^{p}_{s}(\Rd)\big)\cap L^\infty\big(I; L^\infty(\Rd)\big)}.
  \]
  From this, together with \eqref{eq-A0} and \eqref{eq-u0 estimate}, we apply Proposition~1.38 in \cite{Tao} to find $\epsilon_0>0$ such that if $\|u_{0}\|_{_{\dot{W}^p_s(\mathbb R^d)\cap L^\vc(\Rd)}}  \le \epsilon_0$, then there exists a unique solution 
  \[
    u \in C\big(I; \dot{W}^{p}_{s}(\Rd)\big)\cap L^\infty\big(I; L^\infty(\Rd)\big)
  \]
  to the Cauchy problem~\eqref{eq:6.4}, satisfying
  \[
    \|u\|_{C\big(I; \dot{W}^{p}_{s}(\Rd)\big)\cap L^\infty\big(I; L^\infty(\Rd)\big)}
    \lesssim \|u_0\|_{\dot{W}^p_s(\mathbb R^d)\cap L^\vc(\Rd)}.
  \]
  This completes the proof.
\end{proof}

\subsection{Application in perturbation theory}

Theorem~\ref{eqsob} is useful whenever perturbation theoretic arguments are involved. For instance, since $\Lambda_\kappa$ generates a holomorphic semigroup, $(\Lambda_\kappa)^s$, $s\in(0,1]$, is a closed operator in $L^p(\R^d)$ for all $\beta\in((d+\alpha)/2,d+M)$, i.e., the range in Theorem~\ref{eqsob}.
Let $U:\R^d\to\R$ be such that $\|U(\Lambda_0)^{-s}\|_{L^p\to L^p}<\infty$. Then, by perturbation theory (see, e.g., \cite[Chapter~IV, Theorem~1.1]{Kato1966}), $(\Lambda_\kappa)^s+\epsilon U$ is also closed whenever there are $a\in[0,1)$ and $b\geq0$ such that $\|\epsilon Uf\|_{L^p}\leq a\|(\Lambda_\kappa)^s f\|_p + b\|f\|_p$. By Theorem~\ref{eqsob}, whenever applicable, this is indeed the case if $\epsilon$ is sufficiently small, depending on $\|(\Lambda_0)^s(\Lambda_\kappa)^{-s}\|_{L^p\to L^p}\cdot\|U(\Lambda_0)^{-s}\|_{L^p\to L^p}$. A similar argument (cf.~\cite[Chapter~IV, Theorem~1.16]{Kato1966}) can be used to show invertibility of $(\Lambda_\kappa)^s+\epsilon U-z$, whenever $z\in\C$ belongs to the resolvent set of $(\Lambda_\kappa)^s$.

\smallskip
Let us now outline another scenario, which could arise in a many-particle problem involving $\Lambda_\kappa$. As indicated in the introduction (specifically Subsection~\ref{ss:impact}), Schatten bounds of external perturbations relative to the operator describing a physical system are often crucial to study its stationary states. In the context of quantum mechanics, this has been demonstrated at the hand of the one-particle ground state density in \cite{Franketal2020P}, where the effective operator describing the electrons close to the nucleus of a relativistically described atom is $(-\Delta)^{1/2}+\kappa/|x|$ in $L^2(\R^3)$.

To make the following discussion precise, we introduce some notation. For $p\in[1,\infty)$, we denote by $\cs^p=\cs^p(L^2(\R^d))$ the $p$-th Schatten ideal, i.e., the space of all compact operators on $L^2(\R^d)$, denoted by $\cs^\infty$, whose singular values belong to $\ell^p$. For every $p\in[1,\infty)$, $\cs^p$ equipped with the Schatten norm $\|T\|_{\cs^p}=\|(\mu_n(T))_{n\in\N}\|_{\ell^p}$, is a Banach space. Here, $(\mu_n(T))_{n\in\N}$ are the singular values of $T$ in non-increasing order, appearing according to their multiplicities. When $p=\infty$, $\|T\|_{\cs^\infty}$ denotes the operator norm of $T$.
One also often considers the $p$-th weak Schatten ideal $\cs^{p,\infty}$, i.e., the space of all compact operators $T$ satisfying $\|T\|_{\cs^{p,\infty}}:=\sup_n \mu_n(T)n^{1/p}<\infty$.
We record the inclusions $\cs^p\subseteq\cs^{p,\infty}$ and $\cs^p\subseteq\cs^{\tilde p}$ whenever $\tilde p\geq p$. Moreover, $\|T\|_{\cs^p}\lesssim_{p,q} \|T\|_{\cs^{q,\infty}}$ whenever $q<p$.
For further details on Schatten ideals, we refer, e.g., to \cite[Chapters~1 and 2]{Simon2005}.

We now give an application of Theorem~\ref{eqsob} on Schatten bounds relative to powers of $\Lambda_\kappa$. 

\begin{theorem}
  \label{schattenapplication}
  Assume $s\in(0,1]$, $\alpha\in(1,2]$, $\alpha<(d+2)/2$, $\beta\in((d+\alpha)/2,d+M)$, $(d_\beta)'<2<\frac{d}{\alpha s}\wedge d_\beta$, $p=d/(2\alpha s)$, and $U\in L^p(\R^d)$. If $\alpha=2$, assume that the upper heat kernel bound \eqref{eq:heatkernelalpha2} for $\me{-t\Lambda_\kappa}$ holds.
  Then
  \begin{align}
    \label{eq:schattenapplication}
    \|(\Lambda_\kappa)^{-s}U(\Lambda_\kappa)^{-s}\|_{\cs^{p,\infty}} 
    \lesssim_{d,\alpha,s,\beta,p} \|U\|_{L^p(\R^d)}
  \end{align}
  and for all $q>p=d/(2\alpha s)$, 
  \begin{align}
    \label{eq:schattenapplication2}
    \|(\Lambda_\kappa)^{-s}U(\Lambda_\kappa)^{-s}\|_{\cs^{q}}
    \lesssim_{d,\alpha,s,\beta,p,q} \|U\|_{L^p(\R^d)}.
  \end{align}
\end{theorem}

The proof involves the weak $L^p$-space, denoted by $L^{p,\infty}(\R^d)$ and consisting of all $f\in L_{\rm loc}^1(\R^d)$ which satisfy $\|f\|_{p,\infty}:=\sup_{\gamma>0} \gamma \cdot d_f(\gamma)^{1/p}<\infty$; here $[0,\infty)\ni\alpha\mapsto d_f(\alpha):={\rm Leb}(\{x\in\R^d:\, |f(x)|\geq\alpha\})$ is the distribution function of $f$ at height $\alpha\in[0,\alpha)$. Note that $L^p(\R^d)\subseteq L^{p,\infty}(\R^d)$ and $|x|^{-d/p}\in L^{p,\infty}(\R^d)\setminus L^p(\Rd)$.

\begin{proof}[Proof of Theorem~\ref{schattenapplication}]
  By Part~(1) in Theorem~\ref{eqsob},
  \begin{align*}
    \|(\Lambda_\kappa)^{-s} U (\Lambda_\kappa)^{-s}\|_{\cs^{p,\infty}}
    & \leq \|(\Lambda_0)^{s}(\Lambda_\kappa)^{-s}\|_{L^2\to L^2}^2 \cdot \|(\Lambda_0)^{-s}U(\Lambda_0)^{-s}\|_{\cs^{p,\infty}} \\
    & \lesssim_{d,\alpha,s,\beta,p} \|(\Lambda_0)^{-s}U(\Lambda_0)^{-s}\|_{\cs^{p,\infty}}.
  \end{align*}
  By the Cauchy--Schwarz inequality for Schatten ideals (see, e.g., \cite[Theorem~2.8]{Simon2005}) and Cwikel's inequality (see, e.g., \cite[Theorem~4.2]{Simon2005}),
  \begin{align*}
    \|(\Lambda_0)^{-s}U(\Lambda_0)^{-s}\|_{\cs^{p,\infty}}
    & \leq \||U(x)|^{1/2} (-\Delta)^{-\alpha s/2}\|_{\cs^{2p,\infty}}^2 
    \leq \|U\|_{L^{p}} \||\xi|^{-2\alpha s}\|_{L^{p,\infty}}
    \lesssim \|U\|_p,
  \end{align*}
  if $p=d/(2\alpha s)$. Thus, \eqref{eq:schattenapplication} follows. Formula~\eqref{eq:schattenapplication2} follows from~\eqref{eq:schattenapplication} and the inclusion properties of the Schatten ideals.
\end{proof}

\appendix

\section{Proofs of auxiliary statements}
\label{a:proofauxstatements}

\subsection{Proof of Lemma~\ref{monotonicitykappabeta}}
\label{a:proofmonotonicity}

For $\alpha=2$, the claim is obvious.
Thus, let $\alpha\in(1,2)$ and $\beta\in((d+\alpha)/2,d+\alpha)$.
We compute
\begin{align*}
  \begin{split}
    \frac{d\Psi(\beta)}{d\beta}
    & = \frac{2^{\alpha -1} \Gamma \left(\frac{\beta}{2}\right) \Gamma \left(\frac{1}{2} (d+\alpha -\beta )\right)}{(\alpha -\beta )^2 \Gamma \left(\frac{\beta -\alpha }{2}\right) \Gamma \left(\frac{d-\beta }{2}\right)} \\
    & \quad \times \left[(\beta-\alpha) \left(\psi\left(\frac{\beta }{2}\right) -\psi \left(\frac{\beta -\alpha }{2}\right) + \psi\left(\frac{d-\beta}{2}\right) -\psi\left(\frac{1}{2} (d+\alpha -\beta )\right) \right)-2\right]
  \end{split}
\end{align*}
with the Digamma function $\psi(z)=\Gamma'(z)/\Gamma(z)$. Thus, the claim would follow from
\begin{align*}
  \psi\left(\frac{d-\beta}{2}\right) - \psi\left(\frac{d-\beta+\alpha}{2}\right) + \psi\left(\frac{\beta}{2}\right) - \psi\left(\frac{\beta-\alpha}{2}\right) < 0, \quad \beta\in((d+\alpha)/2,d+\alpha).
\end{align*}
To show this inequality, we use \cite[(5.7.6)]{NIST:DLMF}, i.e.,
\begin{align*}
  \psi(z) = -\gamma_{\rm E} + \sum_{k\geq0}\left(\frac{1}{k+1} - \frac{1}{k+z}\right)
\end{align*}
with the Euler--Mascheroni constant $\gamma_{\rm E}$. Thus,
\begin{align*}
  & \psi\left(\frac{d-\beta}{2}\right) - \psi\left(\frac{d-\beta+\alpha}{2}\right) + \psi\left(\frac{\beta}{2}\right) - \psi\left(\frac{\beta-\alpha}{2}\right) \\
  & \quad = \frac{\alpha}{2} \sum_{k\geq0}\left[\frac{1}{(k+\beta/2)(k+(\beta-\alpha)/2)} - \frac{1}{(k+(d-\beta)/2)(k+(d+\alpha-\beta)/2)}\right]
\end{align*}
which is negative for all $\beta>(d+\alpha)/2$. Hence, the strict monotonicity of $\Psi(\beta)$ for $\beta\in((d+\alpha)/2,d+\alpha)$ follows.
\qed

\subsection{Proof of Lemma~\ref{lem- composition of two kernels}}
\label{a:proofcomptwokernels}

The upper bound was proved in \cite[Lemma~22]{FrankMerz2023}.
The lower bound for all $N\in(0,2]$ was proved in \cite[Lemma~4.1]{BuiMerz2023}. We now give another proof of the lower bound which covers all $N>0$.
We consider two cases $s\in (0,t/2)$ and $s\in [t/2,t)$. Since these two cases are similar, we only give the proof for the first case $s\in (0,t/2)$. In this situation $t-s \sim t$, and hence
\begin{equation*}
  \begin{aligned}
    & \int_{\R^d} dz\, (t-s)^{-d}
    \left(\frac{(t-s)}{(t-s)+|x-z|}\right)^{d+N}\, s^{-d}\left(\frac{s}{s+|z-y|}\right)^{d+N} \\
    & \quad \sim \int_{\R^d} dz\, t^{-d}\left(\frac{t}{t+|x-z|}\right)^{d+N}\, s^{-d}\left(\frac{s}{s+|z-y|}\right)^{d+N}.
  \end{aligned}
\end{equation*}
If $|x-y|<2t$, then
\begin{equation*}
  \begin{aligned}
    & \int_{\R^d} dz\, t^{-d} \left(\frac{t}{t+|x-z|}\right)^{d+N}\, s^{-d}\left(\frac{s}{s+|z-y|}\right)^{d+N} \\
    & \quad \ge \int_{B(x,4t)} dz\, t^{-d} s^{-d}\left(\frac{s}{s+|z-y|}\right)^{d+N} 
      \ge \int_{B(y,s)} dz\, t^{-d} s^{-d}\left(\frac{s}{s+|z-y|}\right)^{d+N} \\
    & \quad \gtrsim t^{-d}
      \sim \f{1}{t^d} \left(\frac{t}{t+|x-y|}\right)^{d+N}.
  \end{aligned}
\end{equation*}

If $|x-y|\ge 2t$, then we have $|x-y|\sim |x-y|$ for $z\in B(y,|x-y|/2)$. Hence,
\begin{equation*}
  \begin{aligned}
    & \int_{\R^d} dz\, t^{-d} \left(\frac{t}{t+|x-z|}\right)^{d+N}\, s^{-d}\left(\frac{s}{s+|z-y|}\right)^{d+N} \\
    & \quad \ge \int_{B(y,|x-y|/2)} dz\, t^{-d}\left(\frac{t}{t+|x-y|}\right)^{d+N} s^{-d}\left(\frac{s}{s+|z-y|}\right)^{d+N} \\
    & \quad \ge t^{-d}\left(\frac{t}{t+|x-y|}\right)^{d+N}\int_{B(y,s)} dz\, s^{-d}\left(\frac{s}{s+|z-y|}\right)^{d+N}  \gtrsim \f{1}{t^d} \left(\frac{t}{t+|x-y|}\right)^{d+N}.
  \end{aligned}
\end{equation*}
This concludes the proof of Lemma~\ref{lem- composition of two kernels}.
\qed

\section{On the necessity of $\alpha s < \alpha-1$ for \eqref{eq:equivalencesobolev1}}
\label{a:newrestriction}

In this appendix, we discuss the restriction $\alpha s < \alpha-1$ for \eqref{eq:equivalencesobolev1} in our main result, Theorem~\ref{eqsob}, and the reversed Hardy inequality \eqref{eq:thm-difference} in Theorem~\ref{thm-difference}. Moreover, we discuss whether one can expect a variant of Proposition~\ref{prop-difference v2} where $\Ln^{\frac{1-\gamma}{\alpha}}$ is replaced with $\La^{\frac{1-\gamma}{\alpha}}$.

\subsection{On the restriction on $s$ in \eqref{eq:equivalencesobolev1}}
\label{a:restrictedrange}

Let us discuss the restriction $\alpha s<\alpha-1$ in \eqref{eq:equivalencesobolev1}. To prove
$$
\|\Lambda_0^s \La^{-s} f\|_{L^p(\Rd)} \lesssim \|f\|_{L^p(\Rd)}
$$
it suffices, by the identity
$$
\La^{-s} = \frac{1}{\Gamma(s)} \int_0^\infty \frac{dt}{t}\, t^s \me{-t\La},
$$
to estimate the kernel of $\Lambda_0^s \me{-t\Lambda_\kappa}$. By Duhamel's formula
$$
\me{-t\La}(x,y) = \me{-t\Ln}(x,y) - \kappa\int_0^t ds \int_{\R^d}dz\, \me{-(t-s)\Ln}(x,z)|z|^{-\alpha} z\cdot\nabla_z \me{-s\La}(z,y),
$$
and integration by parts (since gradient bounds for $\me{-t\La}$ are unavailable at the time of this writing and likely difficult to obtain), this requires to treat, among others,
\begin{equation}
  \label{eq1ss}
  \int_0^t \int_{\mathbb R^d}  \nabla_z(\Lambda_0^s \me{-(t-\tau)\Lambda_0}(x,z)) \frac{z}{|z|^\alpha} \me{-\tau\Lambda_\kappa}(z,y)\,dz\,d\tau.
\end{equation}
By scaling,
\begin{align}
  \label{eq:pospowersheatkernel}
  |\nabla \Lambda_0^s \me{-(t-\tau)\Lambda_0} (x,z)|
  \lesi (t-\tau)^{-\frac{1}{\alpha}-s} \f{1}{(t-\tau)^{d/\alpha}}\Big(\f{(t-\tau)^{1/\alpha}}{(t-\tau)^{1/\alpha}+|x-z|}\Big)^{d+\gamma}
\end{align}
for some $\gamma>0$.
(One can deduce this bound, e.g., using \eqref{eq:operatorpowerheatkernel} and Lemma~\ref{derivativeheatkernel}.)
By scaling, the time-integral in \eqref{eq1ss} converges at $\tau=t$, if and only if $\alpha s < \alpha -1$. 

The above argument indicates that \eqref{eq:equivalencesobolev1} may only hold for $\alpha s<\alpha-1$.
However, \emph{if} the following gradient bound
\begin{align}
  \label{eq:heatkernelgradientconjecture0}
  |\nabla_x \me{-t\La}(x,y)|
  & \lesssim t^{-\frac{d+1}{\alpha}}\left(\frac{t^{1/\alpha}}{t^{1/\alpha}+|x-y|}\right)^{d+\gamma} \left(1\wedge\frac{|y|}{t^{1/\alpha}}\right)^{\beta-d}
\end{align}
held for some $\gamma>0$, then we may circumvent the restriction $\alpha s<\alpha-1$. Indeed, given \eqref{eq:heatkernelgradientconjecture0}, we would not need to integrate by parts in the above Duhamel formula. In particular, in that case, the factor $(t-\tau)^{-1/\alpha}$ appearing in \eqref{eq:pospowersheatkernel} would be replaced by a factor $s^{-1/\alpha}$, which, in turn, would make the $\tau$-integration converge for $s<1$.
In view of scaling and Lemma~\ref{derivativeheatkernel} one may wonder if \eqref{eq:heatkernelgradientconjecture0} holds.

\smallskip
Let us also remark that if the gradient perturbation $|z|^{-\alpha}z\cdot\nabla$ was replaced with a scalar perturbation, we would have to consider 
\begin{align}
  \Lambda_0^s \me{-(t-\tau)\Lambda_0} (x,z)
  \lesi (t-\tau)^{-s} \f{1}{(t-\tau)^{d/\alpha}}\Big(\f{(t-\tau)^{1/\alpha}}{(t-\tau)^{1/\alpha}+|x-z|}\Big)^{d+\gamma}
\end{align}
instead of \eqref{eq:pospowersheatkernel}. In this case, the remaining time-integral would converge whenever $s<1$.

\subsection{On the restriction on $s$ in \eqref{eq:thm-difference}}
\label{a:optimallowerbound}

Here, we argue that one can prove \eqref{eq:thm-difference} in Theorem~\ref{thm-difference} using only pointwise bounds for the kernel $Q_t(x,y)$ (defined in \eqref{eq:defqt}) only if $\alpha s<\alpha-1$. This restriction arose when performing the Schur tests \eqref{eq:schurtest2m0}--\eqref{eq:schurtest2m} involving the function $M_t^{\gamma,1}(x,y)$, defined in \eqref{eq:defm}. As we noted in Remark~\ref{rem:schurm}, the reason why the Schur test involving $M_t^{\gamma,1}$ is positive only for $\alpha s<\alpha-1$ is due to the exponent $1-\alpha$ of the factor $\Big(\f{|x|\vee|y|}{t^{1/\alpha}}\Big)^{1-\alpha}$, not because of the decay of the term depending on $|x-y|$ in $M_t^{\gamma,1}(x,y)$. The proof of Proposition~\ref{prop-difference} reveals that the function $M_t^{\gamma,1}$ arose when estimating the right-hand side of \eqref{eq-Kato} on the set $\{(x,y)\in\R^{2d}:\, |y|\geq t^{1/\alpha},\, |x-y|\leq(|x|\wedge|y|)/2\}$. These estimates were carried out in Lemma~\ref{lem- difference for alpha < 2} and were based on Duhamel's formula. In particular, the factor $\Big(\f{|x|\vee |y|}{t^{1/\alpha}}\Big)^{1-\alpha}$ only appears in the integral over $|z|\in[(1-\epsilon_1)|y|,(1+\epsilon_2)|y|]$ for arbitrary but fixed $\epsilon_1,\epsilon_2>0$; see also Remark~\ref{rem:boundf2}.

\smallskip
In the following, we argue that the Duhamel integrals in \eqref{eq-Kato} are bounded from below by a constant times
\begin{align*}
  \left(\frac{|y|}{t^{1/\alpha}}\right)^{1-\alpha}\, t^{-d/\alpha}\left(\frac{t^{1/\alpha}}{t^{1/\alpha}+|x-y|}\right)^{d+\gamma}
\end{align*}
for some $\gamma>0$ and $x,y\in\R^d$ such that $|y|>t^{1/\alpha}$ and $|x-y|<(|x|\wedge|y|)/2$.
A strong argument in that favor would be the lower bound
\begin{align}
  \label{eq:lowerboundoptimalassumptionconj}
  \begin{split}
    & -\int_0^t ds \int_{|z|\in[(1-\epsilon_1)|y|,(1+\epsilon_2)|y|]}dz\, (t-s)^{-d/\alpha}\left(\frac{(t-s)^{1/\alpha}}{(t-s)^{1/\alpha}+|x-z|}\right)^{d+\gamma_1}|z|^{-\alpha} \\
    & \qquad \times z\cdot\frac{(z-y)/|z-y|}{s^{1/\alpha}}s^{-d/\alpha}\left(\frac{s^{1/\alpha}}{s^{1/\alpha}+|z-y|}\right)^{d+\gamma_2} \\
    & \quad \gtrsim \left(\frac{|y|}{t^{1/\alpha}}\right)^{1-\alpha} t^{-d/\alpha}\left(\frac{t^{1/\alpha}}{t^{1/\alpha}+|x-y|}\right)^{d+\gamma}
  \end{split}
\end{align}
for some $\gamma_1,\gamma_2>0$, arbitrary but fixed $\epsilon_1,\epsilon_2>0$, and some $\gamma>0$ and $x,y\in\R^d$ such that $|y|>t^{1/\alpha}$ and $|x-y|\leq(|x|\wedge|y|)/2$. Let us motivate the left-hand side of \eqref{eq:lowerboundoptimalassumptionconj}. The term
\begin{align*}
  (t-s)^{-d/\alpha}\left(\frac{(t-s)^{1/\alpha}}{(t-s)^{1/\alpha}+|x-z|}\right)^{d+\gamma_1}
\end{align*}
comes from the bounds for $\me{-t\La}(x,z)$ and its time-derivative (see \eqref{eq:heatkernel} and Proposition~\ref{thm-ptk}) when $|y|>t^{1/\alpha}$, i.e., when the singular weight $(1\wedge|y|/t^{1/\alpha})^{\beta-d}\sim1$. The term
\begin{align*}
  -\frac{(z-y)/|z-y|}{s^{1/\alpha}}s^{-d/\alpha}\left(\frac{s^{1/\alpha}}{s^{1/\alpha}+|z-y|}\right)^{d+\gamma_2}
\end{align*}
comes from computing the spatial, and possibly additional temporal, derivatives of the right-hand side of the bounds for $\me{-t\Ln}$ in \eqref{eq:heatkernelfree}; see Propositions~\ref{derivativeheatkernel} and \ref{thm-spacetimederivative}. 
Here, the prefactor $s^{-1/\alpha}(z-y)/|z-y|$ comes from the fact that $t^{d/\alpha}\me{-t\Ln}(x,y)$ is a function depending only on $|x-y|/t^{1/\alpha}$.

\smallskip
In the following, we establish the lower bound~\eqref{eq:lowerboundoptimalassumptionconj} in the technically simpler (but artificial) one-dimensional case.

\begin{proposition}
  \label{prop:lowerboundoptimalassumption}
  Let $d=1$, $\alpha\in(0,2)$, $\gamma_1,\gamma_2>0$, and $t>0$. Then there are $\gamma>0$, $y>t^{1/\alpha}$, and $x>0$ such that $|x-y|\leq(x\wedge y)/2$ and \eqref{eq:lowerboundoptimalassumptionconj} hold.
\end{proposition}

\begin{proof}
  By a scaling argument, it suffices to consider $t=1$.
  Let $\epsilon_1<1/6$ and $x=(1-2\epsilon_1)y<y$. Then, $|x-y|=2\epsilon_1 y \leq (1/2-\epsilon_1)y=x/2$. Thus, it suffices to accomplish the following two tasks for some $0<\epsilon_2<1$.

  \begin{enumerate}[(a)]
  \item Show
    \begin{align}
    \label{eq:lowerboundoptimalassumptionaux1}
      \begin{split}
        & \int_0^1 ds \int_y^{(1+\epsilon_2)y} dz\, (1-s)^{-\frac d\alpha}\left(\frac{(1-s)^{1/\alpha}}{(1-s)^{1/\alpha}+|x-z|}\right)^{d+\gamma_1}z^{1-\alpha} s^{-\frac{d+1}{\alpha}}\left(\frac{s^{1/\alpha}}{s^{1/\alpha}+|z-y|}\right)^{d+\gamma_2} \\
        & \quad \leq \int_0^1 ds \int^y_{(1-\epsilon_1)y} dz\, (1-s)^{-\frac d\alpha}\left(\frac{(1-s)^{1/\alpha}}{(1-s)^{1/\alpha}+|x-z|}\right)^{d+\gamma_1}z^{1-\alpha} s^{-\frac{d+1}{\alpha}}\left(\frac{s^{1/\alpha}}{s^{1/\alpha}+|z-y|}\right)^{d+\gamma_2}.
      \end{split}
    \end{align}

  \item Show
    \begin{align}
      \label{eq:lowerboundoptimalassumptionaux2}
      \begin{split}
        & \int_0^1 ds \int^y_{(1-\epsilon_1)y} dz\, (1-s)^{-\frac d\alpha}\left(\frac{(1-s)^{1/\alpha}}{(1-s)^{1/\alpha}+|x-z|}\right)^{d+\gamma_1}z^{1-\alpha} s^{-\frac{d+1}{\alpha}}\left(\frac{s^{1/\alpha}}{s^{1/\alpha}+|z-y|}\right)^{d+\gamma_2} \\
        & \quad \gtrsim y^{1-\alpha} \left(\frac{1}{1+|x-y|}\right)^{d+\gamma}.
      \end{split}
    \end{align}
  \end{enumerate}

  We now accomplish these two tasks:
  \begin{enumerate}[(a)]
  \item We show that \eqref{eq:lowerboundoptimalassumptionaux1} even holds pointwise for all $s\in[0,1]$. To see this, we use that
  \begin{align*}
    & z^{1-\alpha} \left(\frac{(1-s)^{1/\alpha}}{(1-s)^{1/\alpha}+|x-z|}\right)^{d+\gamma_1}\Big|_{z\in[y,(1+\epsilon_2)y]} \\
    & \quad \leq y^{1-\alpha}\left(\frac{(1-s)^{1/\alpha}}{(1-s)^{1/\alpha}+y-x}\right)^{d+\gamma_1} \\
    & \quad \leq z^{1-\alpha} \left(\frac{(1-s)^{1/\alpha}}{(1-s)^{1/\alpha}+|x-z|}\right)^{d+\gamma_1}\Big|_{z\in[(1-\epsilon_1)y,y]},
  \end{align*}
  where we used $|x-z|=z-x$ when $z\in[(1-\epsilon_1)y,y]$ since $x=(1-2\epsilon_1)y$. Thus, it suffices to show
  \begin{align}
    \begin{split}
      & \int_y^{(1+\epsilon_2)y} dz\, \left(\frac{s^{1/\alpha}}{s^{1/\alpha}+z-y}\right)^{d+\gamma_2}
        \leq \int_{(1-\epsilon_1)y}^y dz\, \left(\frac{s^{1/\alpha}}{s^{1/\alpha}+y-z}\right)^{d+\gamma_2}.
    \end{split}
  \end{align}
  But this just follows from shifting $z\mapsto z+y$ in both integrals and replacing $z\mapsto-z$ in one of the integrals. In fact, we could get a strict inequality by taking $\epsilon_2<\epsilon_1$. This concludes the first task.

  \item We now show \eqref{eq:lowerboundoptimalassumptionaux2}. As $z\sim y$, it suffices to show
  \begin{align}
    \begin{split}
      & y^{1-\alpha} \int_0^1 ds \int^y_{(1-\epsilon_1)y} dz\, (1-s)^{-\frac d\alpha}\left(\frac{(1-s)^{1/\alpha}}{(1-s)^{1/\alpha}+|x-z|}\right)^{d+\gamma_1} s^{-\frac{d+1}{\alpha}}\left(\frac{s^{1/\alpha}}{s^{1/\alpha}+|z-y|}\right)^{d+\gamma_2} \\
      & \quad \gtrsim y^{1-\alpha}\left(\frac{1}{1+|x-y|}\right)^{d+\gamma}.
    \end{split}
  \end{align}
  Let
  $$
  F(s,x,y,z) := y^{1-\alpha} (1-s)^{-\frac d\alpha}\left(\frac{(1-s)^{1/\alpha}}{(1-s)^{1/\alpha}+|x-z|}\right)^{d+\gamma_1} s^{-\frac{d+1}{\alpha}}\left(\frac{s^{1/\alpha}}{s^{1/\alpha}+|z-y|}\right)^{d+\gamma_2}.
  $$
  By part (a), we bound
  \begin{align}
  \label{eq:lowerboundoptimalassumptionaux3}
    \begin{split}
      & \int_0^1 ds \int^y_{(1-\epsilon_1)y} dz\, F(s,x,y,z)
      \gtrsim \int_0^1 ds \int_{(1-\epsilon_1)y}^{(1+\epsilon_1)y}dz\, F(s,x,y,z) \\
      & \quad = \int_0^1 ds \int_\R dz\, F(s,x,y,z)- \int_0^1 ds \int_{\R\setminus[(1-\epsilon_1)y,(1+\epsilon_1)y]}dz\, F(s,x,y,z).
    \end{split}
  \end{align}
  Since
  \begin{align}
    \begin{split}
      \int_0^1 ds \int_{\R\setminus[(1-\epsilon_1)y,(1+\epsilon_1)y]}dz\,  F(s,x,y,z)
      \lesssim y^{-\alpha}\, \left(\frac{1}{1+|x-y|}\right)^{d+(\gamma_1\wedge\gamma_2)}
    \end{split}
  \end{align}
  by Remark~\ref{rem:boundf2} and $y^{-\alpha} \lesssim y^{1-\alpha}$ for $y>1$, it suffices to estimates the first summand on the right-hand side of \eqref{eq:lowerboundoptimalassumptionaux3}.
  By \eqref{eq- last inequality}, we have, for $\gamma=\gamma_1\vee\gamma_2$,
  \begin{align}
    \begin{split}
      & y^{1-\alpha} \int_0^1 ds \int_{\R^d} dz\, (1-s)^{-\frac d\alpha}\left(\frac{(1-s)^{1/\alpha}}{(1-s)^{1/\alpha}+|x-z|}\right)^{d+\gamma_1}\, s^{-\frac{d+1}{\alpha}}\left(\frac{s^{1/\alpha}}{s^{1/\alpha}+|z-y|}\right)^{d+\gamma_2} \\
      & \ \gtrsim y^{1-\alpha} \int_0^1 ds \int_{\R^d} dz\, (1-s)^{-\frac d\alpha}\left(\frac{(1-s)^{1/\alpha}}{(1-s)^{1/\alpha}+|x-z|}\right)^{d+\gamma}\, s^{-\frac{d}{\alpha}}\left(\frac{s^{1/\alpha}}{s^{1/\alpha}+|z-y|}\right)^{d+\gamma} \\
      & \ \gtrsim y^{1-\alpha} \left(\frac{1}{1+|x-y|}\right)^{d+\gamma},
    \end{split}
  \end{align}
  as desired.
  \end{enumerate}
  This concludes the proof.
\end{proof}

\subsection{On a variant of Proposition~\ref{prop-difference v2} involving powers of $\La$}
\label{a:reversedhardyoptimal}

We now discuss the possibility to prove a variant of the second version of the reversed Hardy inequality in Proposition~\ref{prop-difference v2}, where the terms involving $\Ln^{\frac{1-\gamma}{\alpha}}f$ are replaced with $\La^{\frac{1-\gamma}{\alpha}}$.
The idea to obtain an estimate involving $\La^{\frac{1-\gamma}{\alpha}}f$ is to insert $1=\La^{-\frac{1-\gamma}{\alpha}}\La^{\frac{1-\gamma}{\alpha}}$ in Duhamel's formula. To that end, we use Duhamel's formula in the form
\begin{align*}
  {p}_{t}(x,y) - \tilde{p}_{t}(x,y)
  & = \kappa \int_0^t\int_{\Rd}{\tilde p}_{t-s}(x,z){|z|^{-\alpha}z\cdot\nabla_z} p_{s}(z,y)\,dz\,ds \\
  & = -\kappa \int_0^t\int_{\Rd} \nabla_z\left[{\tilde p}_{t-s}(x,z){|z|^{-\alpha}z}\right] p_{s}(z,y)\,dz\,ds.
\end{align*}
Here, we integrated by parts to shift the $z$-derivative from $p_s(z,y)$ to $\tilde p_{t-s}(x,z)$ as derivative bounds for $\me{-t\La}$ are unavailable at the time of this writing and likely difficult to obtain.
In this situation, the heat kernel $\me{-t\La}$ can now act on a function $f$ and we can write $\me{-t\La}f=\La^{-\frac{1-\gamma}{\alpha}}\me{-t\La}\La^{\frac{1-\gamma}{\alpha}}f$. By the product rule,
\begin{align*}
  & \tilde {p}_{t}(x,y) - {p}_{t}(x,y) 
    = \kappa \int_0^t ds\int_{\Rd} dz\, \left[|z|^{-\alpha}z\cdot\nabla_z \tilde p_{t-s}(x,z) + \frac{d-\alpha}{|z|^\alpha} \tilde p_{t-s}(x,z)\right]p_s(z,y).
\end{align*}
The second term has a $|z|^{-\alpha}$-decay and can be treated as in \cite{Franketal2021}, thereby giving rise to terms involving $M_t^{\gamma,0}$ and $L_t^{\gamma}$.
To handle the first term, we proceed similarly as in the proof of Proposition~\ref{prop-difference v2}. The only difference is that in the analog of $Q_{1,1}^1$ (see \eqref{eq:qt11}), we use $\me{-\La}=\La^{-\frac{1-\gamma}{\alpha}}\me{-\La}\La^{\frac{1-\gamma}{\alpha}}$ to get the estimate in terms of $|\La^{\frac{1-\gamma}{\alpha}}f|$. We obtain, with $S_{x,1}=\{z\in\R^d:\,|x|/16<|z|<4|x|\}$,
\begin{align}
  \label{eq:qt11lkappa}
  \begin{split}
    \tilde Q_{1,1}^1f(x)
    & := \kappa \int_{\mathbb{R}^d}\int_{S_{x,1}} |z|^{-\alpha} z\cdot (\nabla_z{\tilde p}_{1/2}(x,z)) \, \Lambda_\kappa^{-\f{1-\gamma}{\alpha}} \me{-\f{1}{2}\Lambda_\kappa}(z,y)\Lambda_\kappa^{\f{1-\gamma}{\alpha}} f(y)\,dz\,dy \\
    & \quad + \kappa \int_{\mathbb{R}^d}\int_0^{1/2}\int_{S_{x,1}} |z|^{-\alpha} z\cdot (\nabla_z{\tilde p}_{1-s,1}(x,z)) \, \Lambda_\kappa^{-\f{1-\gamma}{\alpha}} \me{-s\Lambda_\kappa}(z,y)\Lambda_\kappa^{\f{1-\gamma}{\alpha}} f(y)\,dz\,ds\,dy \\
    & \quad + \kappa \int_{\mathbb{R}^d}\int_{1/2}^1 \int_{S_{x,1}} |z|^{-\alpha} z\cdot (\nabla_z {\tilde p}_{1-s}(x,z)) \, \Lambda_\kappa^{-\f{1-\gamma}{\alpha}+1} \me{-s\Lambda_\kappa}(z,y)\Lambda_\kappa^{\f{1-\gamma}{\alpha}} f(y)\,dz\,ds\, dy.
  \end{split}
\end{align}
We now argue that the right-hand side is not expected to give rise to $L_1^{\gamma}$ or $M_1^{\gamma,\gamma}$. Let us consider, e.g., the second summand on the right-hand side of \eqref{eq:qt11lkappa}, whose integral kernel is, in view of Lemma~\ref{derivativeheatkernel} and \eqref{eq:rieszheatkappa}, bounded by a multiple of $|x|^{1-\alpha}$ times
\begin{align}
  \label{eq:qt11kappaintegrand}
  \begin{split}
    & t^{-\frac{d+1}{\alpha}} s^{-\frac{d-\gamma\alpha}{\alpha}} \left(\frac{t^{1/\alpha}}{t^{1/\alpha}+|x-z|}\right)^{d+1+\alpha} \cdot \left(\frac{s^{1/\alpha}}{s^{1/\alpha}+|z-y|}\right)^{d-1+\gamma} \\
    & \quad + t^{-\frac{d+1}{\alpha}} s^{-\frac{d-\gamma\alpha}{\alpha}} \left(\frac{t^{1/\alpha}}{t^{1/\alpha}+|x-z|}\right)^{d+1+\alpha} \cdot \left(\frac{s^{1/\alpha}}{s^{1/\alpha}+|z-y|}\right)^{\beta-1+\gamma} \left(\frac{|y|}{s^{1/\alpha}}\right)^{\beta-d}.
  \end{split}
\end{align}
However, since $\gamma<1$ and $\beta<d$, we can now not argue as in the proof of Proposition~\ref{prop-difference v2} anymore, where we used Lemma~\ref{lem- composition of two kernels}. In particular, the $z$-integration is not expected to yield a function $F(x,y)$ satisfying $F(x,\cdot)\in L^1(\R^d)$ and $F(\cdot,y)\in L^1(\R^d)$ for all $x,y\in\R^d$, thereby leaving us behind with an integral kernel for which we cannot apply Schur tests.

On the other hand, if we \emph{assumed} that \eqref{eq:heatkernelgradientconjecture0} holds, then $\nabla_x\La^{-\frac{1-\gamma}{\alpha}}\me{-t\La}(x,y)$ obeys a bound similar to those in Lemma~\ref{lem-gradient of heat kernel 2}, namely
\begin{align}
  \label{eq:heatkernelgradientconjecture}
  |\nabla_x \La^{-\frac{1-\gamma}{\alpha}}\me{-\La}(x,y)|
  & \lesssim \left(\frac{1}{1+|x-y|}\right)^{d+\gamma} + \left(\frac{1}{1+|x-y|}\right)^{\beta+\gamma} \, |y|^{\beta-d}.
\end{align}
In turn, this bound would allow us to proceed as in the proof of Proposition~\ref{prop-difference v2} and show, for all $\gamma\in(0,1)$, $\beta\in(d-\gamma,d)$, $t>0$, and $x\in\R^d$,
\begin{align}
  \label{eq1-DifferenceKernels v2 conjecture}
  \begin{split}
    |(Q_t f)(x)|
    & \lesi \int_{\Rd} \big[L^{\gamma,1}_t(x,y)+ M^{\gamma,0}_t(x,y)\big] |f(y)|dy \\
    & \quad +\int_{\Rd} L^{\gamma,\gamma}_t(x,y)\big| (\Lambda_\kappa^{\f{1-\gamma}{\alpha}}f)(y)\big|dy \\
    & \quad + \int_{\Rd} M^{\gamma,\gamma}_t(x,y)\big||y|^{1-\gamma}(\Lambda_\kappa^{\f{1-\gamma}{\alpha}}f)(y)\big|dy.
  \end{split}
\end{align}
Consequently, by the proof of Theorem~\ref{thm-difference}, we would obtain, for all $\gamma\in(0,1)$ such that $1-\gamma\leq\alpha s<\alpha-\gamma$, and $\beta\in(d-\gamma,d)$,
\begin{align}
  \label{eq:thm-difference2conj}
  \begin{split}
    \left\|\left(\int_0^\vc t^{-2s}\left|\left(t\La \me{-t\La} -t\Lambda_0\me{-t\Lambda_0}\right)f\right|^2\f{dt}{t}\right)^{1/2}\right\|_{L^p(\R^d)} 
    \lesssim \left\|\frac{|\La^{\frac{1-\gamma}{\alpha}}f(x)|}{|x|^{\alpha s+\gamma-1}}\right\|_{L^p(\R^d)} + \left\|\f{f}{|x|^{\alpha s}}\right\|_{L^p(\R^d)}.
  \end{split}
\end{align}

\section{Conditional generalized Hardy inequality for gradient perturbations}
\label{s:newgenhardy}

Our current generalized Hardy inequality in Theorem~\ref{thm-HardyIneq} gives an upper bound of the scalar Hardy potential in terms of $\La$. Here, we prove a generalized Hardy inequality for the gradient perturbation in terms of $\La$ under the assumption that suitable bounds for $\nabla\me{-t\La}$ are available.

\begin{proposition}
  \label{thm-HardyIneqNew}
  Let $d\in\{3,4,...\}$, $\alpha\in(1,2)$, $\beta\in((d+\alpha)/2,d+\alpha)$, and $\kappa=\Psi(\beta)$ be defined by \eqref{eq:defbeta}. If \eqref{eq:heatkernelgradientconjecture0} holds true for some $\gamma>(2\alpha-d)\vee0$, then, for any $p\in(1\vee d/\beta,d/(\alpha-1))$,
  \begin{align}
      \||x|^{-\alpha}x\cdot\nabla f\|_{L^p(\R^d)}
      \lesssim \|\La f\|_{L^p(\R^d)}.
  \end{align}
\end{proposition}

Note that the range of allowed $p$ in Proposition~\ref{thm-HardyIneqNew} would be slightly larger than that in Theorem~\ref{thm-HardyIneq}.

\begin{proof}
  We argue as in the proof of Theorem~\ref{thm-HardyIneq} using Schur tests and bounds for the integral kernel of
  \begin{align*}
    \left||x|^{-\alpha}x\cdot\nabla_x \La^{-1}(x,y)\right|
    & = \left||x|^{-\alpha}x \int_0^\infty \nabla_x\me{-t\La}(x,y)\,dt\right| \\
    & \lesssim |x|^{1-\alpha}\int_0^\infty dt\, t^{-\frac{d+1}{\alpha}}\left(\frac{t^{1/\alpha}}{t^{1/\alpha}+|x-y|}\right)^{d+\gamma} \left(1\wedge\frac{|y|}{t^{1/\alpha}}\right)^{\beta-d}.
  \end{align*}
  Arguing as in the proof of Lemma~\ref{riesz}, we get, using $\gamma>2\alpha-d$,
  \begin{align*}
    \left||x|^{-\alpha}x\cdot\nabla_x \La^{-1}(x,y)\right|
    \lesssim |x|^{1-\alpha}\, |x-y|^{\alpha-1-d} \left(1\wedge\frac{|y|}{|x-y|}\right)^{\beta-d}.
  \end{align*}
  Thus, it remains to perform Schur tests similar to those in Theorem~\ref{thm-HardyIneq}. The conclusion follows by noting that the only difference between the relevant kernel here and in that proof is that $\alpha$ and $s$ in the proof of Theorem~\ref{thm-HardyIneq} have to replaced with $\alpha-1$ and $1$, respectively.
\end{proof}

\subsection*{Acknowledgments.} 
We are deeply grateful to Damir Kinzebulatov and Karol Szczypkowski for helpful discussions. We would like to thank the referees for useful comments and suggestions which helped to improve the paper. 
T. A. Bui and X. T. Duong were supported by the research grant ARC DP220100285 from the Australian Research Council.

%


\begin{thebibliography}{KMV{\etalchar{+}}18}

\bibitem[AKR03]{Albeverioetal2003}
Sergio Albeverio, Yuri Kondratiev, and Michael R\"{o}ckner.
\newblock Strong {F}eller properties for distorted {B}rownian motion and
  applications to finite particle systems with singular interactions.
\newblock In {\em Finite and infinite dimensional analysis in honor of
  {L}eonard {G}ross ({N}ew {O}rleans, {LA}, 2001)}, volume 317 of {\em Contemp.
  Math.}, pages 15--35. Amer. Math. Soc., Providence, RI, 2003.

\bibitem[Aus07]{Auscher2007}
Pascal Auscher.
\newblock On necessary and sufficient conditions for {$L^p$}-estimates of
  {R}iesz transforms associated to elliptic operators on {$\mathbb{R}^n$} and
  related estimates.
\newblock {\em Mem. Amer. Math. Soc.}, 186(871):xviii+75, 2007.

\bibitem[Azi69]{Aziz1969Vol2}
A.~K. Aziz.
\newblock {\em Lectures in Differential Equations. {Vol}. {II}}, volume No. 19
  of {\em Van Nostrand Mathematical Studies}.
\newblock 1969.

\bibitem[BBC03]{Bogdanetal2003}
Krzysztof Bogdan, Krzysztof Burdzy, and Zhen-Qing Chen.
\newblock Censored stable processes.
\newblock {\em Probab. Theory Related Fields}, 127(1):89--152, 2003.

\bibitem[BD95]{BouchutDolbeault1995}
F.~Bouchut and J.~Dolbeault.
\newblock On long time asymptotics of the {V}lasov-{F}okker-{P}lanck equation
  and of the {V}lasov-{P}oisson-{F}okker-{P}lanck system with {C}oulombic and
  {N}ewtonian potentials.
\newblock {\em Differential Integral Equations}, 8(3):487--514, 1995.

\bibitem[BD23]{BuiDAncona2023}
The~Anh Bui and Piero D'Ancona.
\newblock Generalized {H}ardy operators.
\newblock {\em Nonlinearity}, 36(1):171--198, 2023.

\bibitem[BG60]{BlumenthalGetoor1960}
R.~M. Blumenthal and R.~K. Getoor.
\newblock Some theorems on stable processes.
\newblock {\em Trans. Amer. Math. Soc.}, 95:263--273, 1960.

\bibitem[BGM10]{Bolleyetal2010}
Fran\c{c}ois Bolley, Arnaud Guillin, and Florent Malrieu.
\newblock Trend to equilibrium and particle approximation for a weakly
  selfconsistent {V}lasov-{F}okker-{P}lanck equation.
\newblock {\em M2AN Math. Model. Numer. Anal.}, 44(5):867--884, 2010.

\bibitem[BJ07]{BogdanJakubowski2007}
Krzysztof Bogdan and Tomasz Jakubowski.
\newblock Estimates of heat kernel of fractional {L}aplacian perturbed by
  gradient operators.
\newblock {\em Comm. Math. Phys.}, 271(1):179--198, 2007.

\bibitem[BJLPP22]{Bogdanetal2022}
Krzysztof Bogdan, Tomasz Jakubowski, Julia Lenczewska, and Katarzyna
  Pietruska-Pa{\l}uba.
\newblock Optimal {H}ardy inequality for the fractional {L}aplacian on {$L^p$}.
\newblock {\em J. Funct. Anal.}, 282(8):Paper No. 109395, 31, 2022.

\bibitem[BDY12]{BuiDuongYan2012}
Huy-Qui Bui, Xuan Thinh Duong, and Lixin Yan.
\newblock Calder\'{o}n reproducing formulas and new {B}esov spaces associated with operators.
\newblock {\em Adv. Math.}, 229(4):2449--2502, 2012.

\bibitem[BM23]{BuiMerz2023}
The~Anh {Bui} and Konstantin {Merz}.
\newblock Equivalence of {S}obolev norms in {L}ebesgue spaces for {H}ardy
  operators in a half-space.
\newblock {\em arXiv e-prints}, page arXiv:2309.02928, September 2023.

\bibitem[BR24]{BarbuRockner2024}
Viorel Barbu and Michael R\"{o}ckner.
\newblock Nonlinear {F}okker-{P}lanck equations with fractional {L}aplacian and
  {M}c{K}ean-{V}lasov {SDE}s with {L}\'{e}vy noise.
\newblock {\em Probab. Theory Related Fields}, 189(3-4):849--878, 2024.

\bibitem[{Cat}24]{Cattiaux2024}
Patrick {Cattiaux}.
\newblock Entropy on the path space and application to singular diffusions and
  mean-field models.
\newblock {\em arXiv e-prints}, page arXiv:2404.09552, April 2024.

\bibitem[Cav25]{Cavallazzi2025}
Thomas Cavallazzi.
\newblock Quantitative weak propagation of chaos for {M}c{K}ean-{V}lasov {SDE}s
  driven by stable processes.
\newblock {\em Ann. Inst. Henri Poincar\'{e} Probab. Stat.}, 61(3):1662--1764,
  2025.

\bibitem[CD18a]{CarmonaDelarue2018Vol1}
Ren\'{e} Carmona and Fran\c{c}ois Delarue.
\newblock {\em Probabilistic Theory of Mean Field Games with Applications.
  {I}}, volume~83 of {\em Probability Theory and Stochastic Modelling}.
\newblock Springer, Cham, 2018.
\newblock Mean field FBSDEs, control, and games.

\bibitem[CD18b]{CarmonaDelarue2018Vol2}
Ren\'{e} Carmona and Fran\c{c}ois Delarue.
\newblock {\em Probabilistic Theory of Mean Field Games with Applications.
  {II}}, volume~84 of {\em Probability Theory and Stochastic Modelling}.
\newblock Springer, Cham, 2018.
\newblock Mean field games with common noise and master equations.

\bibitem[CD22]{ChaintronDiez2022}
Louis-Pierre Chaintron and Antoine Diez.
\newblock Propagation of chaos: a review of models, methods and applications.
  {I}. {M}odels and methods.
\newblock {\em Kinet. Relat. Models}, 15(6):895--1015, 2022.

\bibitem[CT04]{ContTankov2004}
Rama Cont and Peter Tankov.
\newblock {\em Financial Modelling with Jump Processes}.
\newblock Chapman \& Hall/CRC Financial Mathematics Series. Chapman \&
  Hall/CRC, Boca Raton, FL, 2004.

\bibitem[CV10]{CaffarelliVasseur2010}
Luis~A. Caffarelli and Alexis Vasseur.
\newblock Drift diffusion equations with fractional diffusion and the
  quasi-geostrophic equation.
\newblock {\em Ann. of Math. (2)}, 171(3):1903--1930, 2010.

\bibitem[Dav90]{Davies1990}
E.~B. Davies.
\newblock {\em Heat Kernels and Spectral Theory}, volume~92 of {\em Cambridge
  Tracts in Mathematics}.
\newblock Cambridge University Press, Cambridge, 1990.

\bibitem[{\relax DLMF}23]{NIST:DLMF}
{\it NIST Digital Library of Mathematical Functions}.
\newblock \url{https://dlmf.nist.gov/}, Release 1.1.9 of 2023-03-15, 2023.
\newblock F.~W.~J. Olver, A.~B. {Olde Daalhuis}, D.~W. Lozier, B.~I. Schneider,
  R.~F. Boisvert, C.~W. Clark, B.~R. Miller, B.~V. Saunders, H.~S. Cohl, and
  M.~A. McClain, eds.

\bibitem[Dol91]{Dolbeault1991}
J.~Dolbeault.
\newblock Stationary states in plasma physics: {M}axwellian solutions of the
  {V}lasov-{P}oisson system.
\newblock {\em Math. Models Methods Appl. Sci.}, 1(2):183--208, 1991.

\bibitem[DPT25]{Duongetal2025}
Manh~Hong Duong, Grigorios~A. Pavliotis, and Julian Tugaut.
\newblock Multi-species {M}cKean-{V}lasov dynamics in non-convex landscapes.
\newblock {\em arXiv e-prints}, page arXiv:2507.07617, 2025.

\bibitem[DR96]{DuongRobinson1996}
Xuan~T. Duong and Derek~W. Robinson.
\newblock Semigroup kernels, {P}oisson bounds, and holomorphic functional
  calculus.
\newblock {\em J. Funct. Anal.}, 142(1):89--128, 1996.

\bibitem[FJ17]{FournierJourdain2017}
Nicolas Fournier and Benjamin Jourdain.
\newblock Stochastic particle approximation of the {K}eller-{S}egel equation
  and two-dimensional generalization of {B}essel processes.
\newblock {\em Ann. Appl. Probab.}, 27(5):2807--2861, 2017.

\bibitem[FL19]{FitzsimmonsLi2019}
Patrick~J. Fitzsimmons and Liping Li.
\newblock On the {D}irichlet form of three-dimensional {B}rownian motion
  conditioned to hit the origin.
\newblock {\em Sci. China Math.}, 62(8):1477--1492, 2019.

\bibitem[FLS08]{Franketal2008H}
Rupert~L. Frank, Elliott~H. Lieb, and Robert Seiringer.
\newblock Hardy-{L}ieb-{T}hirring inequalities for fractional {S}chr\"odinger
  operators.
\newblock {\em J. Amer. Math. Soc.}, 21(4):925--950, 2008.

\bibitem[FM23]{FrankMerz2023}
Rupert~L. Frank and Konstantin Merz.
\newblock On {S}obolev norms involving {H}ardy operators in a half-space.
\newblock {\em J. Funct. Anal.}, 285(10):Paper No. 110104, 2023.

\bibitem[FMS21]{Franketal2021}
Rupert~L. Frank, Konstantin Merz, and Heinz Siedentop.
\newblock Equivalence of {S}obolev norms involving generalized {H}ardy
  operators.
\newblock {\em International Mathematics Research Notices}, 2021(3):2284--2303,
  February 2021.

\bibitem[FMS23a]{Franketal2023T}
Rupert~L. Frank, Konstantin Merz, and Heinz Siedentop.
\newblock The {S}cott conjecture for large {C}oulomb systems: a review.
\newblock {\em Lett. Math. Phys.}, 113(1):Paper No. 11, 79, 2023.

\bibitem[FMS23b]{Franketal2023}
Rupert~L. Frank, Konstantin Merz, and Heinz Siedentop.
\newblock Relativistic strong {S}cott conjecture: a short proof.
\newblock In {\em Density Functionals for Many-Particle Systems---Mathematical Theory and Physical Applications of Effective Equations}, volume~41 of {\em
  Lect. Notes Ser. Inst. Math. Sci. Natl. Univ. Singap.}, pages 69--79. World
  Sci. Publ., Hackensack, NJ, March [2023] \copyright 2023.

\bibitem[FMSS20]{Franketal2020P}
Rupert~L. {Frank}, Konstantin {Merz}, Heinz {Siedentop}, and Barry {Simon}.
\newblock Proof of the strong {S}cott conjecture for {C}handrasekhar atoms.
\newblock {\em Pure Appl. Funct. Anal.}, 5(6):1319--1356, December 2020.

\bibitem[FS08]{FrankSeiringer2008}
Rupert~L. Frank and Robert Seiringer.
\newblock Non-linear ground state representations and sharp {H}ardy
  inequalities.
\newblock {\em J. Funct. Anal.}, 255(12):3407--3430, 2008.

\bibitem[Fun84]{Funaki1984}
Tadahisa Funaki.
\newblock A certain class of diffusion processes associated with nonlinear
  parabolic equations.
\newblock {\em Z. Wahrsch. Verw. Gebiete}, 67(3):331--348, 1984.

\bibitem[GL90]{GarofaloLanconelli1990}
Nicola Garofalo and Ermanno Lanconelli.
\newblock Level sets of the fundamental solution and {H}arnack inequality for
  degenerate equations of {K}olmogorov type.
\newblock {\em Trans. Amer. Math. Soc.}, 321(2):775--792, 1990.

\bibitem[Har19]{Hardy1919}
G.~H. Hardy.
\newblock Notes on some points in the integral calculus {LI}: On {H}ilbert's
  double-series theorem, and some connected theorems concerning the convergence
  of infinite series and integrals.
\newblock {\em Messenger of Mathematics}, 48:107--112, 1919.

\bibitem[Har20]{Hardy1920}
G.~H. Hardy.
\newblock Note on a theorem of {H}ilbert.
\newblock {\em Mathematische Zeitschrift}, 6(3--4):314--317, 1920.

\bibitem[Her77]{Herbst1977}
Ira~W. Herbst.
\newblock Spectral theory of the operator {$(p^2+m^2)^{1/2} - Ze^2/r$}.
\newblock {\em Comm.\ Math.\ Phys.}, 53:285--294, 1977.

\bibitem[HJM{\etalchar{+}}24]{Haoetal2024P}
Zimo {Hao}, Jean-Francois {Jabir}, St{\'e}phane {Menozzi}, Michael
  {R{\"o}ckner}, and Xicheng {Zhang}.
\newblock Propagation of chaos for moderately interacting particle systems
  related to singular kinetic McKean-Vlasov SDEs.
\newblock {\em arXiv e-prints}, page arXiv:2405.09195, May 2024.

\bibitem[HRZ24]{Haoetal2024}
Zimo Hao, Michael R\"{o}ckner, and Xicheng Zhang.
\newblock Strong convergence of propagation of chaos for {M}c{K}ean-{V}lasov
  {SDE}s with singular interactions.
\newblock {\em SIAM J. Math. Anal.}, 56(2):2661--2713, 2024.

\bibitem[Kac56]{Kac1956}
M.~Kac.
\newblock Foundations of kinetic theory.
\newblock In {\em Proceedings of the {T}hird {B}erkeley {S}ymposium on
  {M}athematical {S}tatistics and {P}robability, 1954--1955, vol. {III}}, pages
  171--197. Univ. California Press, Berkeley-Los Angeles, Calif., 1956.

\bibitem[Kat66]{Kato1966}
Tosio Kato.
\newblock {\em Perturbation Theory for Linear Operators}, volume 132 of {\em
  Grundlehren der mathematischen {W}issenschaften}.
\newblock Springer-Verlag, Berlin, 1 edition, 1966.

\bibitem[KMS24]{Kinzebulatovetal2024}
Damir Kinzebulatov, Kodjo~Rapha\"{e}l Madou, and Yuliy~A. Sem\"{e}nov.
\newblock On the supercritical fractional diffusion equation with {H}ardy-type
  drift.
\newblock {\em J. Anal. Math.}, 152(2):401--420, 2024.

\bibitem[KMV{\etalchar{+}}17]{Killipetal2017}
Rowan Killip, Changxing Miao, Monica Visan, Junyong Zhang, and Jiqiang Zheng.
\newblock The energy-critical {NLS} with inverse-square potential.
\newblock {\em Discrete Contin. Dyn. Syst.}, 37(7):3831--3866, 2017.

\bibitem[KMV{\etalchar{+}}18]{Killipetal2018}
R.~Killip, C.~Miao, M.~Visan, J.~Zhang, and J.~Zheng.
\newblock Sobolev spaces adapted to the {S}chr\"odinger operator with
  inverse-square potential.
\newblock {\em Math. Z.}, 288(3-4):1273--1298, 2018.

\bibitem[KMVZ17]{Killipetal2017T}
Rowan Killip, Jason Murphy, Monica Visan, and Jiqiang Zheng.
\newblock The focusing cubic {NLS} with inverse-square potential in three space
  dimensions.
\newblock {\em Differential Integral Equations}, 30(3-4):161--206, 2017.

\bibitem[KPS81]{Kovalenkoetal1981}
V.~F. Kovalenko, M.~A. Perelmuter, and Ya.~A. Semenov.
\newblock Schr\"odinger operators with ${L}_w^{l/2}(\mathbb{R}^l)$-potentials.
\newblock {\em J. Math. Phys.}, 22:1033--1044, 1981.

\bibitem[KRS08]{Klagesetal2008}
R.~Klages, G.~Radons, and I.M. Sokolov.
\newblock {\em Anomalous Transport: Foundations and Applications}.
\newblock Wiley, 2008.

\bibitem[KS06]{KimSong2006}
Panki Kim and Renming Song.
\newblock Two-sided estimates on the density of {B}rownian motion with singular
  drift.
\newblock {\em Illinois J. Math.}, 50(1-4):635--688, 2006.

\bibitem[KS20]{KinzebulatovSemenov2020}
Damir Kinzebulatov and Yuliy~A. Sem\"{e}nov.
\newblock On the theory of the {K}olmogorov operator in the spaces {$L^p$} and
  {$C_\infty$}.
\newblock {\em Ann. Sc. Norm. Super. Pisa Cl. Sci. (5)}, 21:1573--1647, 2020.

\bibitem[KS23]{KinzebulatovSemenov2023F}
Damir Kinzebulatov and Yuliy~A. Sem\"{e}nov.
\newblock Fractional {K}olmogorov operator and desingularizing weights.
\newblock {\em Publ. Res. Inst. Math. Sci.}, 59(2):339--391, 2023.

\bibitem[KSS21]{Kinzebulatovetal2021}
D.~Kinzebulatov, Yu.~A. Sem\"{e}nov, and K.~Szczypkowski.
\newblock Heat kernel of fractional {L}aplacian with {H}ardy drift via
  desingularizing weights.
\newblock {\em J. Lond. Math. Soc. (2)}, 104(4):1861--1900, 2021.

\bibitem[Lie81]{Lieb1981}
Elliott~H. Lieb.
\newblock {T}homas-fermi and related theories of atoms and molecules.
\newblock {\em Rev.\ Mod.\ Phys.}, 53(4):603--641, October 1981.

\bibitem[M\'el96]{Meleard1996}
Sylvie M\'{e}l\'{e}ard.
\newblock Asymptotic behaviour of some interacting particle systems;
  {M}c{K}ean-{V}lasov and {B}oltzmann models.
\newblock In {\em Probabilistic Models for Nonlinear Partial Differential Equations ({M}ontecatini {T}erme, 1995)}, volume 1627 of {\em Lecture Notes
  in Math.}, pages 42--95. Springer, Berlin, 1996.

\bibitem[McI86]{McIntosh1986}
Alan McIntosh.
\newblock Operators which have an {$H_\infty$} functional calculus.
\newblock In {\em Miniconference on Operator Theory and Partial Differential
  Equations ({N}orth {R}yde, 1986)}, volume~14 of {\em Proc. Centre Math. Anal.
  Austral. Nat. Univ.}, pages 210--231. Austral. Nat. Univ., Canberra, 1986.

\bibitem[McK66]{McKean1966}
H.~P. McKean, Jr.
\newblock A class of {M}arkov processes associated with nonlinear parabolic equations.
\newblock {\em Proc. Nat. Acad. Sci. U.S.A.}, 56:1907--1911, 1966.

\bibitem[McK67]{McKean1967}
Henry~P McKean.
\newblock Propagation of chaos for a class of non-linear parabolic equations.
\newblock {\em Stochastic Differential Equations (Lecture Series in
  Differential Equations, Session 7, Catholic Univ., 1967)}, pages 41--57,
  1967.

\bibitem[Mer21]{Merz2021}
Konstantin Merz.
\newblock On scales of {S}obolev spaces associated to generalized {H}ardy
  operators.
\newblock {\em Math. Z.}, 299(1):101--121, 2021.

\bibitem[Mer22]{Merz2022}
Konstantin Merz.
\newblock On complex-time heat kernels of fractional {S}chr{\"o}dinger
  operators via {P}hragm{\'e}n-{L}indel{\"o}f principle.
\newblock {\em J. Evol. Equ.}, 22(3):Paper No. 62, 30, 2022.

\bibitem[MK00]{MetzlerKlafter2000}
Ralf Metzler and Joseph Klafter.
\newblock The random walk's guide to anomalous diffusion: a fractional dynamics
  approach.
\newblock {\em Physics Reports}, 339(1):1--77, 2000.

\bibitem[MMW15]{Mischleretal2015}
St\'{e}phane Mischler, Cl\'{e}ment Mouhot, and Bernt Wennberg.
\newblock A new approach to quantitative propagation of chaos for drift,
  diffusion and jump processes.
\newblock {\em Probab. Theory Related Fields}, 161(1-2):1--59, 2015.

\bibitem[MNS18]{Metafuneetal2018}
G.~Metafune, L.~Negro, and C.~Spina.
\newblock Sharp kernel estimates for elliptic operators with second-order
  discontinuous coefficients.
\newblock {\em J. Evol. Equ.}, 18(2):467--514, 2018.

\bibitem[MS22]{MerzSiedentop2022}
Konstantin Merz and Heinz Siedentop.
\newblock Proof of the strong {Scott} conjecture for heavy atoms: the {Furry}
  picture.
\newblock {\em Annales Henri Lebesgue}, 5:611--642, 2022.

\bibitem[MSS17]{Metafuneetal2017}
G.~Metafune, M.~Sobajima, and C.~Spina.
\newblock Kernel estimates for elliptic operators with second-order
  discontinuous coefficients.
\newblock {\em J. Evol. Equ.}, 17(1):485--522, 2017.

\bibitem[Ris89]{Risken1989}
H.~Risken.
\newblock {\em The {F}okker-{P}lanck Equation}, volume~18 of {\em Springer
  Series in Synergetics}.
\newblock Springer-Verlag, Berlin, second edition, 1989.
\newblock Methods of solution and applications.

\bibitem[Sil12a]{Silvestre2012H}
Luis Silvestre.
\newblock H\"{o}lder estimates for advection fractional-diffusion equations.
\newblock {\em Ann. Sc. Norm. Super. Pisa Cl. Sci. (5)}, 11(4):843--855, 2012.

\bibitem[Sil12b]{Silvestre2012}
Luis Silvestre.
\newblock On the differentiability of the solution to an equation with drift
  and fractional diffusion.
\newblock {\em Indiana Univ. Math. J.}, 61(2):557--584, 2012.

\bibitem[Sim84]{Simon1984}
B.~Simon.
\newblock Fifteen problems in mathematical physics.
\newblock In {\em Perspectives in Mathematics}. Birkh{\"a}user, 1984.

\bibitem[Sim05]{Simon2005}
Barry Simon.
\newblock {\em Trace Ideals and Their Applications}, volume 120 of {\em
  Mathematical Surveys and Monographs}.
\newblock American Mathematical Society, Providence, RI, second edition, 2005.

\bibitem[SLD{\etalchar{+}}01]{Schertzeretal2001}
D.~Schertzer, M.~Larchev\^{e}que, J.~Duan, V.~V. Yanovsky, and S.~Lovejoy.
\newblock Fractional {F}okker-{P}lanck equation for nonlinear stochastic
  differential equations driven by non-{G}aussian {L}\'{e}vy stable noises.
\newblock {\em J. Math. Phys.}, 42(1):200--212, 2001.

\bibitem[Sti19]{Stinga2019}
Pablo~Ra\'{u}l Stinga.
\newblock User's guide to the fractional {L}aplacian and the method of
  semigroups.
\newblock In {\em Handbook of Fractional Calculus with Applications. {V}ol. 2},
  pages 235--265. De Gruyter, Berlin, 2019.

\bibitem[Szn91]{Sznitman1991}
Alain-Sol Sznitman.
\newblock Topics in propagation of chaos.
\newblock In {\em \'{E}cole d'\'{E}t\'{e} de {P}robabilit\'{e}s de
  {S}aint-{F}lour {XIX}---1989}, volume 1464 of {\em Lecture Notes in Math.},
  pages 165--251. Springer, Berlin, 1991.

\bibitem[Tao06]{Tao} Terence Tao. \newblock Nonlinear dispersive equations. Local and global analysis. \newblock CBMS Regional Conference Series in Mathematics, 106. Published for the Conference Board of the Mathematical Sciences, Washington, DC; by the American Mathematical Society, Providence, RI, 2006.

\bibitem[Tar24]{Tardy2024}
Yoan Tardy.
\newblock Weak convergence of the empirical measure for the {K}eller-{S}egel
  model in both subcritical and critical cases.
\newblock {\em Electron. J. Probab.}, 29:Paper No. 142, 35, 2024.

\bibitem[Tay00]{Taylor2000}
Michael~E. Taylor.
\newblock {\em Tools for {PDE}}, volume~81 of {\em Mathematical Surveys and
  Monographs}.
\newblock American Mathematical Society, Providence, RI, 2000.
\newblock Pseudodifferential operators, paradifferential operators, and layer
  potentials.

\bibitem[Vla68]{Vlasov1968}
Anatoli{\u\i}~Aleksandrovich Vlasov.
\newblock The vibrational properties of an electron gas.
\newblock {\em Soviet Physics Uspekhi}, 10(6):721, 1968.

\bibitem[WT15]{WeiTian2015}
Jinlong Wei and Rongrong Tian.
\newblock Well-posedness for the fractional {F}okker-{P}lanck equations.
\newblock {\em J. Math. Phys.}, 56(3):031502, 11, 2015.

\bibitem[Yaf99]{Yafaev1999}
D.~Yafaev.
\newblock Sharp constants in the {H}ardy-{R}ellich inequalities.
\newblock {\em Journ. Functional Analysis}, 168(1):121--144, October 1999.

\bibitem[Zha95]{Zhang1995}
Qi~Zhang.
\newblock A {H}arnack inequality for {K}olmogorov equations.
\newblock {\em J. Math. Anal. Appl.}, 190(2):402--418, 1995.

\bibitem[Zha97]{Zhang1997}
Qi~S. Zhang.
\newblock Gaussian bounds for the fundamental solutions of {$\nabla (A\nabla
  u)+B\nabla u-u_t=0$}.
\newblock {\em Manuscripta Math.}, 93(3):381--390, 1997.

\end{thebibliography}

\newcommand{\etalchar}[1]{$^{#1}$}
\def\cprime{$'$}

\end{document}